\documentclass[a4paper,12pt,reqno,twoside]{amsart}

\usepackage{amssymb}
\usepackage{amsmath}
\usepackage{graphicx}
\usepackage[colorlinks,citecolor=blue,urlcolor=blue]{hyperref}
\usepackage{cite}

\usepackage[hmargin=2cm,vmargin=2cm]{geometry}

\newcommand{\be}{\begin{eqnarray}}
\newcommand{\ee}{\end{eqnarray}}

\newcommand{\beq}{\begin{equation}}
\newcommand{\eeq}{\end{equation}}
\newcommand{\beqn}{\begin{equation*}}
\newcommand{\eeqn}{\end{equation*}}

\newtheorem{thm}{Theorem}
\newtheorem{prop}[thm]{Proposition}
\newtheorem{cor}[thm]{Corollary}
\newtheorem{lem}[thm]{Lemma}
\newtheorem{defn}[thm]{Definition}

\newtheorem{remark}[thm]{Remark}

\newtheorem*{thm1'}{Theorem 1'}

\newcommand\bfh{{\mathbf h}}
\newcommand{\ve}{\varepsilon}
\newcommand{\rd}{\mathrm{d}}
\newcommand\bzero{\mathbf{0}}

\def\bB{\mathbf{B}}

\def\bE{\mathbf{E}}
\def\bF{\mathbf{F}}

\def\bc{\mathbf{c}}

\def\be{\mathbf{e}}

\def\bq{\mathbf{q}}

\def\bu{\mathbf{u}}
\def\bv{\mathbf{v}}

\def\cA{\mathcal{A}}
\def\cB{\mathcal{B}}
\def\cC{\mathcal{C}}
\def\cD{\mathcal{D}}

\def\cF{\mathcal{F}}

\def\cI{\mathcal{I}}
\def \cJ{\mathcal{J}}

\def\cK{\mathcal{K}}

\def\cM{\mathcal{M}}
\def\cN{\mathcal{N}}

\def\cQ{\mathcal{Q}}

\def\cS{\mathcal{S}}

\def\cU{\mathcal{U}}

\def\cZ{\mathcal{Z}}

\def\IE{{\mathbb E}}

\def\IH{{\mathbb H}}
\def\IK{{\mathbb K}}

\def\IN{{\mathbb N}}

\def\IR{{\mathbb R}}
\def\IS{{\mathbb S}}
\def\IT{{\mathbb T}}

\def\IZ{{\mathbb Z}}

\def\fm{\mathfrak{m}}
\def\fs{\mathfrak{s}}

\def\fS{\mathfrak{S}}

\def\ft{\mathfrak{t}}

\def\numscatt{\fs}

\begin{document}

\title{Dispersing billiards with moving scatterers}

\author[Mikko Stenlund]{Mikko Stenlund}
\address[Mikko Stenlund]{
Department of Mathematics, University of Rome ``Tor Vergata''\\
Via della Ricerca Scientifica, I-00133 Roma, Italy; Department of Mathematics and Statistics, P.O.\ Box 68, Fin-00014 University of Helsinki, Finland.}
\email{mikko.stenlund@helsinki.fi}
\urladdr{http://www.math.helsinki.fi/mathphys/mikko.html}

\author[Lai-Sang Young]{Lai-Sang Young}
\address[Lai-Sang Young]{
Courant Institute of Mathematical Sciences\\
New York, NY 10012, USA.}
\email{lsy@cims.nyu.edu}
\urladdr{http://www.cims.nyu.edu/~lsy/ }

\author[Hongkun Zhang]{Hongkun Zhang}
\address[Hongkun Zhang]{Department of Mathematics \& Statistics\\University of Massachusetts\\Amherst, 01003, USA.}
\email{hongkun@math.umass.edu}
\urladdr{http://www.math.umass.edu/~hongkun/}




\begin{abstract}
We propose a model of Sinai billiards with moving scatterers, in which the locations 
and shapes of the scatterers may change by small amounts between collisions. 
Our main result is the exponential loss of memory of initial data at uniform rates, 
and our proof consists of a coupling argument
for non-stationary compositions of maps similar to classical billiard maps. This can be 
seen as a prototypical result on the statistical properties of time-dependent dynamical systems.
\end{abstract}

\maketitle


\subsection*{Acknowledgements}
Stenlund is supported by the Academy of Finland; he also
wishes to thank Pertti Mattila for valuable correspondence. Young is supported
by NSF Grant DMS-1101594, and Zhang is supported by NSF Grant DMS-0901448.


\section{Introduction}
\subsection{Motivation}\label{sec:motivation} The physical motivation for our paper is a setting in which
a finite number of larger and heavier particles move about slowly as they are bombarded
by a large number of lightweight (gas) particles. Following the language of billiards, 
we refer to the heavy particles as \emph{scatterers}. In classical billiards theory,
scatterers are assumed to be stationary, an assumption justified by first letting the ratios
of masses of heavy-to-light particles
tend to infinity. We do not fix the scatterers here. Indeed the system may be open --- gas particles can be injected or ejected, heated up or cooled down. We consider a window of observation $[0,T], \ T \le \infty$, and assume that during this time interval
the total energy stays uniformly below a constant value~$E>0$. This places an upper bound proportional to~$\sqrt{E}$ on the translational and rotational speeds of the scatterers. The constant of proportionality depends inversely on the masses and moments of inertia of the scatterers. Suppose the scatterers are also pairwise repelling due to an interaction with a short but positive effective range, such as a weak Coulomb force, whose strength tends to infinity with the inverse of the distance. The distance between any pair of scatterers has then a lower bound, which in the Coulomb case is proportional to~$1/E$. In brief, fixing a maximum value for the total energy $E$, the scatterers are guaranteed to be uniformly bounded away from each other; and assuming that the ratios of masses 
are sufficiently large, the scatterers will move arbitrarily slowly. 
Our goal is to study the dynamics of a tagged gas particle in such a system on the 
time interval $[0,T]$. As a simplification we assume our tagged particle is passive: it is
massless, does not interact with the other light particles, and does not interfere 
with the motion of the scatterers. It experiences an elastic collision each time it 
meets a scatterer, and moves on with its own energy unchanged.\footnote{The model 
here should not be confused with \cite{ChernovDolgopyatBBM}, which describes
the motion of a \emph{heavy} particle bombarded by a fast-moving light particle 
reflected off the walls of a bounded domain.}
This model was proposed in the paper~\cite{OttStenlundYoung}.

The setting above is an example of a \emph{time-dependent dynamical system}.
Much of dynamical systems theory as it exists today is concerned with autonomous
systems, i.e., systems for which the rules of the
dynamics remain constant through time. Non-autonomous systems studied include
those driven by a time-periodic or random forcing (as described by SDEs), or 
more generally, systems driven by another autonomous dynamical system 
(as in a skew-product setup).
For time-varying systems without any assumption of periodicity or
stationarity, even the formulation of results poses obvious mathematical
challenges, yet many real-world systems are of this type. Thus while the 
moving scatterers model above is of independent interest, we had 
another motive for undertaking the present project: we wanted to use this 
prototypical example to catch a glimpse of the challenges ahead, and at the same
time to identify techniques of stationary theory that 
carry over to time-dependent systems.

\subsection{Main results and issues}\label{sec:issues}
We focus in this paper on the evolution of densities.
 Let $\rho_0$ be an initial distribution, and $\rho_t$ its time evolution. 
In the case of an autonomous system with good statistical properties,
one would expect $\rho_t$ to tend to the system's natural invariant distribution
(e.g. SRB measure) as $t \to \infty$. The question is:
How quickly is $\rho_0$ ``forgotten"? Since ``forgetting" the features of an initial 
distribution is generally associated with mixing of the dynamical system, one may pose
the question as follows: Given two initial distributions $\rho_0$ and $\rho'_0$, 
how quickly does $|\rho_t-\rho'_t|$ tend to zero (in some measure of distance)?
In the time-dependent case, $\rho_t$ and $\rho'_t$ 
may never settle down, as the rules of the dynamics may be changing perpetually. 
Nevertheless the question continues to makes sense. We say a system has \emph{exponential memory loss} if $|\rho_t-\rho'_t|$ 
decreases exponentially with time. 

Since memory loss is equivalent to mixing for a fixed map,
a natural setting with exponential memory loss for time-dependent sequences is when
the maps to be composed have, individually, strong mixing properties, and 
the rules of the dynamics, or the maps to be composed, vary slowly.
(In the case of continuous time, this is equivalent to the vector field changing very 
slowly.) In such a setting, we may think of $\rho_t$ above as slowly varying
as well. Furthermore, in the case of exponential loss of memory, we may view these
probability distributions as representing, after an initial transient, \emph{quasi-stationary states}.

Our main result in this paper is the exponential memory loss of initial data for the collision
maps of 2D models of the type described in Section~\ref{sec:motivation}, where the scatterers are assumed to be
moving very slowly. Precise results are formulated in Section~\ref{sec:statements}. Billiard maps with fixed,
convex scatterers are known to have exponential correlation decay; thus the setting
in Section~\ref{sec:motivation} is a natural illustration of the scenario in the last paragraph.
(Incidentally, when the source and target configurations differ, the collision
map does not necessarily preserve the usual invariant measure). 
 
If we were to iterate a
single map long enough for exponential mixing to set in, then change the map ever
so slightly so as not to disturb the convergence in $|\rho_t-\rho'_t|$ already achieved,
and iterate the second map for as long as needed before making an even smaller
change, and so on, then exponential loss of memory for the sequence is immediate
for as long as all the maps involved are individually exponentially mixing. 
This is not the type of result we are after. A more meaningful result --- and this is what
we will prove --- is one in which one identifies a space of dynamical systems and 
an upper bound in the speed with which the sequence is allowed to vary, and prove exponential memory loss for any sequence in this space that varies slowly enough.
This involves more than the exponential mixing property of individual maps;
the class of maps in question has to satisfy a \emph{uniform mixing condition 
for slowly-varying compositions}. 
This in some sense is the crux of the matter.

A technical but fundamental issue has to do with stable 
and unstable directions, the staples of hyperbolic dynamics. 
In time-dependent systems with slowly-varying parameters, approximate 
stable and unstable directions can be defined, but they depend on the time interval
of interest, e.g., which direction is contracting depends on how long one chooses to look.
Standard dynamical tools have to be adapted to the new setting of 
non-stationary sequences; consequently technical estimates of single billiard maps 
have to be re-examined as well. 

\subsection{Relevant works}\label{sec:background}
Our work lies at the confluence of the following two sets of results:

The study of statistical properties of billiard maps in the case of fixed convex scatterers 
was pioneered  
by Sinai et al \cite{Sinai_1970,BSC_1990,BSC_1991}.
The result for exponential correlation decay was first proved in \cite{Young_1998};
another proof using a coupling argument is given in~\cite{Chernov-BilliardsCoupling}.
Our exposition here follows closely that in~\cite{Chernov-BilliardsCoupling}. Coupling, which is
the main tool of the present paper, is a standard technique in probability.
To our knowledge it was imported into hyperbolic dynamical systems in~\cite{Young_1999}. The very convenient formulation in~\cite{Chernov-BilliardsCoupling} was first
used in~\cite{ChernovDolgopyatBBM}. (Despite appearing in 2009, the latter circulated as a preprint already in 2004.) We refer the reader to~\cite{ChernovMarkarian_2006}, which
contains a detailed exposition of this and many other important
technical facts related to billiards.

The paper~\cite{OttStenlundYoung} proved exponential loss of memory for expanding 
maps and for one-dimensional piecewise expanding maps with slowly varying parameters. An earlier study in the same spirit is~\cite{LasotaYorke_1996}. A similar result was obtained for topologically transitive Anosov diffeomorphisms in two dimensions in~\cite{Stenlund_2011} and for piecewise expanding maps in higher dimensions in~\cite{GuptaOttTorok_2012}. We mention also~\cite{AyyerStenlund}, where exponential memory loss was established for arbitrary sequences of finitely many toral automorphisms satisfying a common-cone condition. Recent central-limit-type results in the time-dependent setting can be found in~\cite{ConzeRaugi,Nandori_2012,Stenlund_2012}.

\subsection{About the exposition} One of the goals of this paper is to stress the (strong) similarities between stationary dynamics and their time-dependent counterparts, and to highlight at the same time the new issues that need to be addressed. For this reason,
and also to keep the length of the manuscript reasonable, we have elected to omit
the proofs of some technical preliminaries for which no substantial modifications
are needed from the fixed-scatterers case, referring the reader instead to 
\cite{ChernovMarkarian_2006}. We do not know to what degree we have succeeded,
but we have tried very hard to make transparent the logic of the argument, 
in the hope that it will be accessible to a wider audience. The main ideas 
are contained in Section~\ref{sec:outline}.

The paper is organized as follows.
In Section~\ref{sec:statements} we describe the model in detail, after which we immediately state our main results in a form as accessible as possible, leaving generalizations
for later. Theorems~\ref{thm:weak_conv_compact}--\ref{thm:eq_mixing} of Section~\ref{sec:statements} are the main results of this paper, and Theorem~\ref{thm:weak_conv} 
is a more technical formulation which easily implies the other two.  Sections~\ref{sec:preliminariesI} and~\ref{sec:preliminariesII} contain a collection of facts about dispersing billiard maps that are easily adapted to the time-dependent case.  Section~\ref{sec:outline} gives a nearly complete outline of the proof of Theorem~\ref{thm:weak_conv}. 
In Section~\ref{sec:preliminaries_continued} we continue with technical preliminaries necessary for a rigorous proof of that theorem. Unlike Sections~\ref{sec:preliminariesI} and~\ref{sec:preliminariesII}, more stringent conditions on the speeds at which the scatterers are allowed to move are needed for the results in Section~\ref{sec:preliminaries_continued}. 
In Section~\ref{sec:proof_countable} we prove Theorem~\ref{thm:weak_conv} in 
the special case of initial distributions supported on countably many curves, and 
in Section~\ref{sec:completion} we prove the extension of Theorem 4 to more general 
settings. Finally, we collect in the Appendix some proofs which are deferred to the end in order not to disrupt the flow of the presentation in the body of the text. 


\section{Precise statement of main results}\label{sec:statements}

\subsection{Setup}\label{sec:setup}
We fix here a space of scatterer configurations, and make precise
the definition of billiard maps with possibly different source and target configurations.

Throughout this paper, the physical space of our system is the $2$-torus $\IT^2$. 
We assume, to begin with (this condition will be relaxed later on), 
that the number of scatterers as well as their sizes and shapes 
are fixed, though rigid rotations and translations are permitted. Formally, let 
$B_1,\dots,B_\numscatt$ be pairwise disjoint closed convex domains in
$\mathbb R^2$ with $C^3$ boundaries of strictly positive curvature. 
In the interior of each $B_i$ we fix a reference point $c_i$ and a unit vector 
$u_i$ at $c_i$.
A {\it configuration}~$\cK$ of $\{B_1,\dots,B_\numscatt\}$ in $\IT^2$ is
an embedding of $\cup_{i=1}^\numscatt B_i$ into $\IT^2$, one that maps
 each $B_i$ isometrically onto a set we call~$\bB_i$. Thus $\cK$ can
be identified with a point $(\bc_i,\bu_i)_{i=1}^\numscatt \in 
{(\IT^2\times \IS^1)}^\numscatt$, $\bc_i$ and $\bu_i$ being images of 
$c_i$ and $u_i$. The space of configurations~$\mathbb K_0$ is
the subset of ${(\IT^2\times \IS^1)}^\numscatt$ for which the~$\bB_i$ are
pairwise disjoint and every half-line in~$\IT^2$ meets a scatterer non-tangentially. More conditions will be imposed on $\cK$ later on.
The set $\mathbb K_0$ inherits the Euclidean metric from ${(\IT^2\times \IS^1)}^\numscatt$, and the $\ve$-neighborhood of $\cK$ is denoted by $\cN_\ve(\cK)$.

Given a configuration $\cK\in\IK_0$, let $\tau^{\min}_\cK$ be the shortest length of a line segment in $\IT^2\setminus \cup_{i=1}^\numscatt\bB_i$ which originates and terminates (possibly tangentially) in the set $\cup_{i=1}^\numscatt\partial\bB_i$,\footnote{In general, $\tau^{\min}_\cK\neq \min_{1\leq i<j\leq \numscatt}\operatorname{dist}(\bB_i,\bB_j)$, as the shortest path could be from a scatterer back to itself. If one lifts the $\bB_i$ to 
$\mathbb R^2$, then $\tau^{\min}_\cK$ is the shortest distance between distinct 
images of lifted scatterers.} and let $\tau^{\max}_\cK$ be the supremum of the lengths of all line segments in the closure of $\IT^2\setminus \cup_{i=1}^\numscatt\bB_i$ which originate and terminate \emph{non-tangentially} in the set $\cup_{i=1}^\numscatt\partial\bB_i$ 
(this segment may meet the scatterers tangentially between its endpoints). 
As a function of $\cK$, $\tau^{\min}_\cK$ is continuous, but 
 $\tau^{\max}_\cK$ in general is only upper semi-continuous. 
 Notice that $0<\tau^{\min}_\cK<\tau^{\max}_\cK \le \infty$.

\begin{figure}
\includegraphics{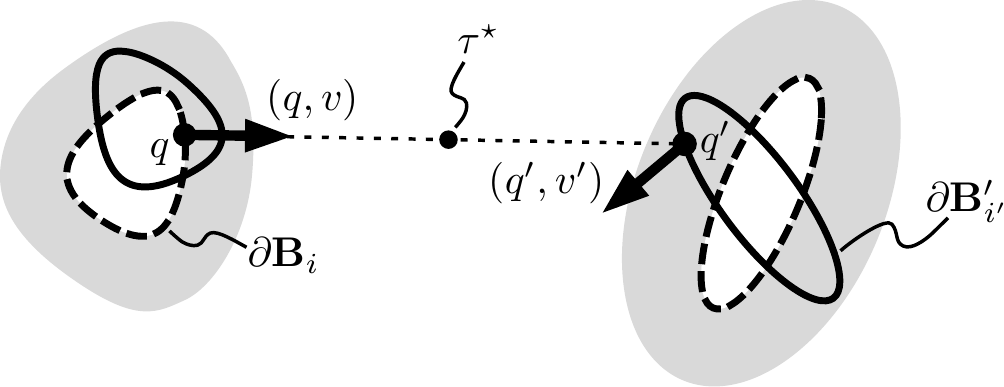}
\caption{Rules of the dynamics. Scatterers in source configuration $\cK$ and target configuration $\cK'$ are drawn in dashed and solid line, respectively.  A particle shoots off the boundary of a scatterer $\bB_i$ at the point $q$ with unit velocity $v$ and exits the gray buffer zone $\bB_{i,\beta}\setminus\bB_i$. Before it re-enters the buffer zone of any scatterer $\bB_j$, the configuration is switched instantaneously from $\cK$ to $\cK'$ at some time $\tau^\star$ during mid-flight. The particle then hits the boundary of a scatterer $\bB_{i'}'$ elastically at the point $q'$, resulting in post-collision velocity~$v'$.}
\label{fig:map}
\end{figure}

A basic question is: Given $\cK, \cK'\in\IK_0$, is there always
a well-defined billiard map (analogous to classical billiard maps) 
with source configuration $\cK$ and target configuration $\cK'$?
That is to say, if $\bB_1,\dots,\bB_\numscatt$ are the scatterers 
in configuration~$\cK$,  and $\bB_1',\dots,\bB_\numscatt'$ are the corresponding 
scatterers in~$\cK'$, is there a well defined mapping 
$$\bF_{\cK',\cK}: 
T^+_1(\cup_{i=1}^\numscatt \partial \bB_i) \to T^+_1(\cup_{i=1}^\numscatt \partial \bB'_i)$$
where $T^+_1(\cup_{i=1}^\numscatt \partial \bB_i)$ is
the set of $(q,v)$ such that $q \in \cup_{i=1}^\numscatt \partial \bB_i$ and
$v$ is a unit vector at $q$ pointing into the region 
$\IT^2\setminus \cup_{i=1}^\numscatt\bB_i$, and similarly for  
$T^+_1(\cup_{i=1}^\numscatt \partial \bB'_i)$?
Is the map $\bF_{\cK',\cK}$ uniquely defined, or does it depend on 
when the changeover from $\cK$
to $\cK'$ occurs? The answer can be very general, but let us confine ourselves  
to the special case where $\cK'$ is very close to $\cK$ and the changeover occurs 
when the particle is in ``mid-flight" (to avoid having scatterers land 
on top of the particle, or meet it at the exact moment of the changeover).

To do this systematically, we introduce the idea of a buffer zone. For $\beta>0$, we let 
$\bB_{i,\beta}\subset \IT^2$ denote the $\beta$-neighborhood of $\bB_i$,
and define $\tau^\textrm{esc}_\beta$, the {\it escape time} 
from the $\beta$-neighborhood of $\cup_i \bB_i$, to be 
the maximum length of a line in
$\cup_{i=1}^\numscatt(\bB_{i,\beta}\setminus\bB_i)$ connecting $\cup_{i=1}^\numscatt\partial\bB_i$ to $\cup_{i=1}^\numscatt \partial(\bB_{i,\beta})$.
We then fix a value of $\beta>0$ small enough that 
$\tau^\textrm{esc}_\beta<\tau^{\min}_\cK-\beta$, and require that 
$\bB_i'\subset \bB_{i,\beta}$ for each $i=1,\dots,\numscatt$. 
Notice that $\beta<\tau^\textrm{esc}_\beta$, so that $\beta <\tau^{\min}_\cK/2$,
implying in particular that the neighborhoods $\bB_{i,\beta}$ are pairwise disjoint.
For a particle starting from $\cup_{i=1}^\numscatt \partial \bB_i$, its trajectory is guaranteed
to be outside of $\cup_{i=1}^\numscatt\bB_{i,\beta}$ during the 
time interval $(\tau^\textrm{esc}_\beta, \tau^{\min}_\cK-\beta)$:
reaching $\cup_{i=1}^\numscatt\bB_{i,\beta}$ before time
$\tau^{\min}_\cK-\beta$ would contradict the definition of $\tau^{\min}_\cK$.
We permit the configuration change to take place at any time 
$\tau^\star \in (\tau^\textrm{esc}_\beta,\tau^{\min}_\cK-\beta)$.
Notice that $\tau^\textrm{esc}_\beta$ depends only on the shapes of the scatterers, 
not their configuration, and that the billiard trajectory starting from $\cup_i \partial \bB_i$ and ending in $\cup_i \partial \bB'_i$ does not depend on the precise 
moment $\tau^\star$ at which the configuration is updated. For the billiard map 
$\bF_{\cK',\cK}$ to be defined, every particle trajectory 
starting from $\cup_{i=1}^\numscatt \partial \bB_i$ 
must meet a scatterer in~$\cK'$. This is guaranteed by $\cK'\in\IK_0$, due to the requirement that any half-line intersects a scatterer boundary.

To summarize, we have argued that given $\cK, \cK'\in\IK_0$, there is a canonical 
way to define $\bF_{\cK',\cK}$  if 
$\bB_i' \subset
\bB_{i,\beta}$ for all $i$ where $\beta=\beta(\tau^{\min}_{\cK})>0$ depends only on 
$\tau^{\min}_{\cK}$ (and the curvatures of the $B_i$), and the flight time $\tau_{\cK',\cK}$ satisfies
$\tau_{\cK',\cK} \ge \tau^{\min}_\cK-\beta \ge \tau^{\min}_\cK/2$.

Now we would like to have all the $\bF_{\cK',\cK}$ operate on a single phase space 
$\cM$, so that our {\it time-dependent billiard system} defined by compositions of 
these maps can be studied in a way analogous to iterated classical billiard maps.
As usual, we let $\Gamma_i$ be a fixed clockwise parametrization by arclength 
of $\partial B_i$, and let 
\beqn
\cM = \cup_i\cM_i
\quad\text{with}\quad
\cM_i = \Gamma_i\times [-\pi/2,\pi/2]. 
\eeqn 
Recall that each $\cK \in \IK$ is defined by an isometric embedding of  
$\cup_{i=1}^\numscatt B_i$ into $\IT^2$. This embedding 
extends to a neighborhood of $\cup_{i=1}^\numscatt B_i \subset \mathbb R^2$, 
inducing a diffeomorphism $\Phi_\cK : \cM \to T^+_1(\cup_{i=1}^\numscatt \partial \bB_i)$.
For $\cK, \cK'$ for which  $\bF_{\cK',\cK}$ is defined then, we have
\beqn
F_{\cK',\cK} := \Phi_{\cK'}^{-1} \circ\bF_{\cK',\cK} \circ \Phi_{\cK}:\cM\to\cM\ .
\eeqn
Furthermore, given a sequence $(\cK_n)_{n=0}^N$ of configurations, we let  $F_n = F_{\cK_n,\cK_{n-1}}$ assuming this mapping
is well defined, and write
\beqn
\cF_{n+m,n}=F_{n+m}\circ\dots\circ F_n \qquad\text{and}\qquad \cF_{n}=F_{n}\circ\dots\circ F_1
\eeqn
 for all $n, m$ with $1\leq n\leq n+m\leq N$.

It is easy to believe --- and we will confirm mathematically 
--- that $F_{\cK',\cK}$ has many of
the properties of the section map of the 2D periodic Lorentz gas. 
The following differences, however, are of note: unlike
classical billiard maps, $F_{\cK',\cK}$ is in general
{\it neither one-to-one nor onto}, and as a result of that it also \emph{does not preserve the usual measure} on $\cM$. This is illustrated in Figure~\ref{fig:preimages}.

\begin{figure}
\includegraphics{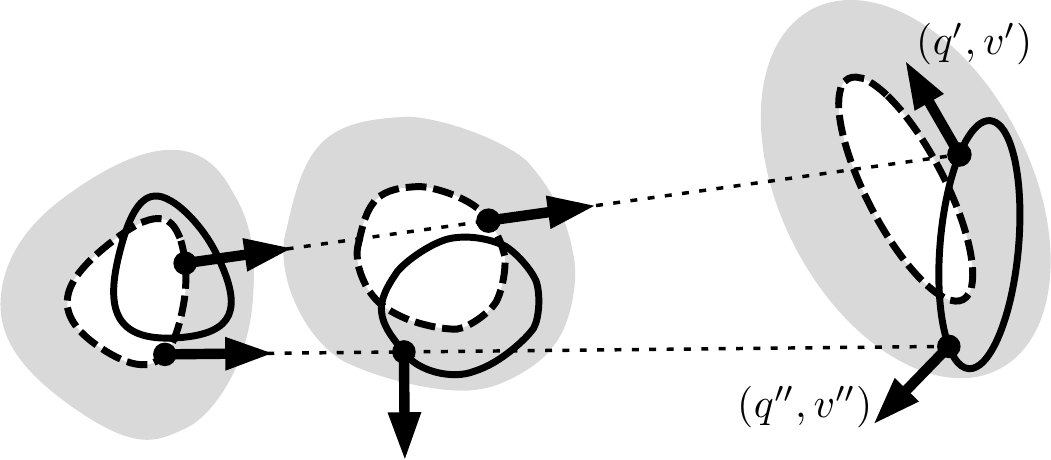}
\caption{Action of the map $F_{\cK',\cK}$. With the same conventions as in Figure~\ref{fig:map}, the point in $\cM$ corresponding to the plane vector $(q',v')$ has more than one preimage, whereas the point corresponding to $(q'',v'')$ has no preimage at all.}
\label{fig:preimages}
\end{figure}


\subsection{Main results}\label{sec:results}

First we introduce the following {\it uniform finite-horizon condition}:
For $\ft, \varphi>0$, $\varphi$ small,
we say $\cK \in \IK_0$ has $(\ft, \varphi)$-horizon 
if every directed open line segment in $\IT^2$ of length $\ft$ meets 
a scatterer $\cB_i$ of $\cK$ at an angle $> \varphi$ (measured from 
its tangent line), with the segment 
approaching this point of contact from $\IT^2 \setminus \bB_i$.
Other intersection points between our line segment and $\cup_j \partial \bB_j$
are permitted and no requirements are placed on the angles at which
they meet; we require
only that there be at least one intersection point meeting the condition above.
Notice that this condition is not affected by the sudden appearance or
disappearance of nearly tangential collisions of billiard trajectories with
scatterers as the positions of the scatterers are shifted.

The space in which we will permit our time-dependent configurations to
wander is defined as follows: We fix $0<\bar\tau^{\min}<\ft<\infty$ and
$\varphi>0$, chosen so that the set 
$$
\IK = \IK (\bar\tau^{\min}, (\ft, \varphi)) =  \{\cK \in \IK_0 : \bar\tau^{\min}<\tau^{\min}_\cK \ {\rm and} \ 
\cK \ {\rm has} \ (\ft, \varphi){\rm -horizon} \}
$$
is nonempty. Clearly, $\IK$ is an open set, and its closure $\bar{\IK}$
as a subset of ${(\IT^2\times \IS^1)}^\numscatt$ consists of those configurations
whose $\tau^{\min}$ will be $\ge \bar\tau^{\min}$, and line segments
of length $\ft$ with their end points added will meet scatterers with angles
$\ge \varphi$.  
From Section~\ref{sec:setup},
we know that there exists $\bar\beta = \beta(\bar\tau^{\min})>0$ 
such that $F_{\cK',\cK}$ is defined for all $\cK,\cK' \in \IK$ with 
$\bB_i' \subset \bB_{i,\bar\beta}$ for all $i$ where $\{\bB_i\}$ and $\{\bB_i'\}$ are the scatterers 
in $\cK$ and $\cK'$ respectively. For simplicity, we will call the pair
$(\cK, \cK')$ \emph{admissible} (with respect to $\IK$) if they satisfy the condition above.
Clearly, if $\cK,\cK' \in \IK$ are such that $d(\cK,\cK')<\ve$ for small
enough $\ve$, then the pair is admissible.
We also noted in Section~\ref{sec:setup} that for all admissible pairs, 
\beq\label{eq:flight_times}
\bar\tau^{\min}/2 \leq \tau_{\cK',\cK}\leq \ft \ .
\eeq

We will denote by ${|f|}_\gamma$ the H\"older constant of a $\gamma$-H\"older continuous $f:\cM\to\IR$.

Our main result is

\smallskip
\begin{thm}\label{thm:weak_conv_compact}
Given $\IK = \IK (\bar\tau^{\min}, (\ft, \varphi))$, there exists $\ve>0$ such that the following holds. Let $\mu^1$ and $\mu^2$ be probability measures on $\cM$, with
strictly positive, $\tfrac16$-H\"older continuous densities~$\rho^1$ and~$\rho^2$ 
with respect to the measure $\cos\varphi\,\rd r \,\rd\varphi$.
Given $\gamma>0$, there exist $0<\theta_\gamma<1$ and $C_\gamma>0$ such that
\beqn
\left|\int_{\cM} f\circ \cF_n\, \rd\mu^1 - \int_{\cM} f\circ \cF_n\, \rd\mu^2\right| \leq C_\gamma ({\| f \|}_\infty + {|f|}_\gamma)\theta_\gamma^n, \quad n\leq N,
\eeqn
for all finite or infinite sequences $(\cK_n)_{n=0}^N\subset\IK$ ($N\in \IN\cup\{\infty\}$) satisfying $d(\cK_{n-1},\cK_n)<\ve$ for $1\leq n\leq N$, and all
 $\gamma$-H\"older continuous $f:\cM\to\IR$. The constant $C_\gamma=C_\gamma(\rho^1,\rho^2)$ depends on the densities $\rho^i$ through the H\"older constants of $\log\rho^i$, while $\theta_\gamma$ does not depend 
on the $\mu^i$. Both constants depend on $\IK$ and $\ve$. 
\end{thm}

\smallskip
None of the constants in the theorem depends on $N$. We have included
the finite $N$ case to stress that our results do not depend on
knowledge of scatterer movements in the infinite future; requiring
such knowledge would be unereasonable for time-dependent systems.
The notation ``$(\cK_n)_{n=0}^N, \ N\in \IN\cup\{\infty\}$'' is intended as shorthand for $\cK_1, \dots, \cK_N$ for $N<\infty$, and
$\cK_1, \cK_2, \dots$ (infinite sequence) for $N=\infty$.

\bigskip

Our next result is an extension of Theorem~\ref{thm:weak_conv_compact} to a situation
where
the geometries of the scatterers are also allowed to vary with time. 
We use $\kappa$ to denote the curvature of the scatterers, and use the convention
that $\kappa>0$ corresponds to strictly convex scatterers.
For $0< \bar\kappa^{\min}<\bar\kappa^{\max}<\infty$, $0< \bar\tau^{\min}<\ft < \infty, \ \varphi>0$ and $0<\Delta<\infty$, we let 
$$
\widetilde\IK = \widetilde\IK(\bar\kappa^{\min}, \bar\kappa^{\max}; \bar\tau^{\min}, (\ft, \varphi); \Delta)
$$
denote the set of configurations $\cK = \bigl((\bB_1, o_1), \dots, (\bB_\fs, o_\fs)\bigr)$
where $(\bB_1, \dots, \bB_\fs)$ is an ordered set of disjoint scatterers on $\IT^2$,
$o_i \in \partial \bB_i$ is a marked point for each $i$, $\fs \in \mathbb N$ is arbitrary, 
and the following conditions are satisfied:

\smallskip

(i) the scatterer boundaries $\partial \bB_i$ are $C^{3+\mathrm{Lip}}$ with $\|D(\partial \bB_i)\|_{C^2}< \Delta$ and $\mathrm{Lip}(D^3(\partial \bB_i))< \Delta$,

(ii) the curvatures of $\partial \bB_i$ lie between~$\bar\kappa^{\min}$ 
and~$\bar\kappa^{\max}$, and 

(iii) $\tau_\cK^{\min} > \bar\tau^{\min}$, and $\cK$ has $(\ft, \varphi)$-horizon. 

\smallskip

\noindent In (i), $\|D(\partial \bB_i)\|_{C^2}$ and $\mathrm{Lip}(D^3(\partial \bB_i))$ are defined to be $\max_{1\leq k\leq 3}\| D^k\gamma_i\|_\infty$ and $\mathrm{Lip}(D^3\gamma_i)$, respectively,
where $\gamma_i$ is the unit speed clockwise parametrization of $\bB_i$.
For two configurations $\cK = ((\bB_1, o_1), \dots, (\bB_\fs, o_\fs))$ and
$\cK' = ((\bB'_1, o'_1), \dots, (\bB'_\fs, o'_\fs))$ with the same number of scatterers, 
we define 
$d_3(\cK, \cK')$ to be the maximum of $\max_{i\le\fs}\sup_{x\in\cM}d_\cM(\hat\gamma_i(x),\hat\gamma_i'(x))$ and $\max_{i \le \fs}\max_{1\leq k\leq 3} \|D^k\hat \gamma_i - D^k\hat \gamma'_i\|_\infty$
where $\hat \gamma_i: \IS^1 \to \IT^2$ denotes the constant speed clockwise parametrization of 
$\partial \bB_i$ with $\hat \gamma_i(0)=o_i$, $\hat \gamma_i'$ is the corresponding
parametrization of~$\partial \bB'_i$ with $\hat \gamma'_i(0)=o'_i$, and~$d_\cM$ is the natural distance on~$\cM$. 
The definition of \emph{admissibility} for~$\cK$ and~$\cK'$ is as above, and 
the billiard map~$F_{\cK',\cK}$ is defined as before for admissible pairs. 
Configurations~$\cK, \cK'$ with different numbers of scatterers are not admissible, and 
the distance between them is set arbitrarily to be $d_3(\cK, \cK')=1$.

\smallskip
\begin{thm} \label{var_geom} The statement of Theorem \ref{thm:weak_conv_compact} holds verbatim with $(\IK, d)$ replaced by $(\widetilde \IK, d_3)$.\footnote{The differentiability assumption on the scatterer boundaries can be relaxed, but the pursuit of minimal technical conditions is not the goal of our paper.}
\end{thm}

\medskip
\noindent
{\bf Theorems 1' and 2':} The regularity assumption on the measures $\mu^i$ 
in Theorems~\ref{thm:weak_conv_compact}--\ref{var_geom} above can be much relaxed. It suffices to assume that the $\mu^i$ have regular conditional
measures on unstable curves; they can be singular in the transverse direction
and can, e.g., be supported on a single unstable curve. Convex combinations
of such measures are also admissible. Precise conditions are given in 
Section \ref{sec:preliminariesII}, after we have introduced the relevant technical 
definitions. Theorems 1'--2', which are the extensions of Theorems~\ref{thm:weak_conv_compact}--\ref{var_geom} respectively
to the case where these relaxed conditions on $\mu^i$ are permitted, are stated
in Section~\ref{sec:primed_statements}. 

\bigskip
Theorems~\ref{var_geom} and 2' obviously apply as a special case to classical billiards, giving uniform bounds of the kind above for all $F_{\cK,\cK}$, $\cK \in \widetilde\IK$. 
It is also a standard fact that correlation decay results can be deduced from 
the type of convergence in Theorems~\ref{thm:weak_conv_compact}--\ref{var_geom}.
To our knowledge, the following result on correlation decay for classical billiards is new. (See also pp.~149--150 in~\cite{ChernovEyink_etal} for related observations.) The proof can be found in Section~\ref{sec:var_geom}.

\smallskip
\begin{thm}\label{thm:eq_mixing} Let $\mu$ denote the measure obtained by normalizing $\cos\varphi\,\rd r\,\rd\varphi$ to a probability measure. Let $\widetilde\IK$ be fixed,
and let $\gamma>0$ be arbitrary. Then for any $\gamma$-H\"older continuous~$f$ and any~$\tfrac16$-H\"older continuous $g$, there exists a constant $C'_\gamma$ such that
\beqn
\left|\int f\circ F^n\cdot g\,\rd\mu - \int f\,\rd\mu\int g\,\rd\mu \right| \leq C_\gamma'\,\theta_\gamma^n
\eeqn
hold for all $n\geq 0$ and for all $F = F_{\cK,\cK}$ with $\cK \in \widetilde\IK$.
Here $\theta_\gamma$ is as in the theorems above. The constant $C_\gamma'$ depends on $\|f\|_\infty$, $|f|_\gamma$, $\|g\|_\infty$ and $|g|_\frac16$.
\end{thm}

We remark that Theorem~\ref{thm:eq_mixing} can also be formulated for sequences of maps. In that
case the quantity bounded is $\int f\circ \cF_n\cdot g\,\rd\mu - \int f\circ\cF_n\,\rd\mu \, \int g\,\rd\mu $ and $\mu$ is an arbitrary measure satisfying the conditions in Theorems 1' and 2'. 
The proof is unchanged.

In addition to the broader class of measures, Theorem~\ref{var_geom} could be extended to less regular observables~$f$, which would allow for a corresponding generalization of Theorem~\ref{thm:eq_mixing}. In particular, the observables could be allowed to have discontinuities at the singularities of the map $F$; see, e.g.,~\cite{ChernovMarkarian_2006}. In order to keep the focus on what is new, we do not pursue that here.

\bigskip

We state one further extension of the above theorems, to include the situation where the test particle is also under the influence of an external field. Given an admissible pair $(\cK,\cK')$ in $\widetilde\IK$ and a vector field $\bE=\bE(\bq,\bv)$, 
we define first a continuous time system in which the trajectory of 
the test particle between collisions is determined by the equations
\beqn
\dot\bq = \bv \quad\text{and} \quad \dot\bv = \bE,
\eeqn
where $\bq$ is the position and $\bv$ the velocity of the particle, together with the initial condition. For the sake of simplicity, let us assume that the field is isokinetic --- that is, $\bv\cdot\bE=0$ --- which allows to normalize $|\bv| = 1$. This class of forced billiards includes ``electric fields with Gaussian thermostats'' studied in~\cite{MoranHooverBestiale_1987,ChernovEyink_etal} and many other papers. (Instead of the speed, more general integrals of the motion could be considered, allowing for other types of fields, such as gradients of weak potentials; see~\cite{Chernov_2001,Chernov_2008}.) 
Assuming that the field~$\bE$ is smooth and small, the trajectories are almost linear, and a billiard map $F^\bE_{\cK',\cK}:\cM\to\cM$ can be defined exactly as before. (See Section~\ref{sec:proof_fields} for more details.) Note that $F^\bzero_{\cK',\cK} = F_{\cK',\cK}$.

The setup for our time-dependent systems result is as follows: We consider
the space $\widetilde \IK \times \IE$ where $\widetilde \IK$ is as above and
$\IE = \IE(\ve^\bE)$ for some $\ve^\bE>0$ is the set of fields 
$\bE\in C^{2}$ with
$\|\bE\|_{C^2} = \max_{0\leq k\leq 2}\|D^k\bE\|_\infty < \ve^\bE$.  In the theorem below, it is to be understood that $F_n = F^{\bE_n}_{\cK_n,\cK_{n-1}}$ and $\cF_n = F_n\circ\dots\circ F_1$.

\medskip
\noindent
{\bf Theorem E.}
{\it Given $\widetilde \IK$, there exist $\ve>0$ and $\ve^\bE>0$
such that the statement of Theorem \ref{var_geom} holds for all sequences
$((\cK_n, \bE_n))_{n \le N}$ in $\widetilde \IK \times \IE(\ve^\bE)$ satisfying
$d_{3}(\cK,\cK') < \ve$ for all $n \le N$.} 

\medskip

Like the zero-field case, Theorem~E also admits a generalization of measures (and observables) and also implies an exponential correlation bound.

\subsection{Main technical result}\label{sec:tech}

To prove Theorem~\ref{thm:weak_conv_compact}, we will, in fact, prove the following technical result. 
All configurations below are in $\IK$. 
Let $(\widetilde \cK_q)_{q=1}^Q$ ($Q \in \mathbb Z^+$ arbitrary)
be a sequence of configurations, $(\tilde \ve_q)_{q=1}^Q$ 
a sequence of positive numbers, and $(\widetilde N_q)_{q=1}^Q$ a sequence of positive integers. We say the configuration sequence $(\cK_n)_{n=0}^N$ (arbitrary $N$)
is {\it adapted} to $(\widetilde \cK_q,\tilde\ve_q,\widetilde N_q)_{q=1}^Q$ 
if there exist numbers $0=n_0<n_1<\dots<n_Q = N$ such that for $1\leq q\leq Q$,
we have $n_q-n_{q-1}\geq \widetilde N_q$ and $\cK_n\in \cN_{\tilde\ve_q}(\widetilde \cK_q)$ for $n_{q-1}\leq n\leq n_q$. That is to say, we think of the $(\widetilde \cK_q)_{q=1}^Q$ 
as reference configurations, and view the sequence of interest, $(\cK_n)_{n=0}^N$, as going from one reference configuration to the next, spending a long time ($\ge \widetilde N_q$) 
near (within $\tilde \ve_q$ of) each $\widetilde \cK_q$.

\smallskip
\begin{thm}\label{thm:weak_conv}
For any $\cK\in\IK$, there exist $\widetilde N(\cK)\geq 1$ and $\tilde\ve(\cK)>0$ such that the following holds for every sequence of reference configurations 
$(\widetilde \cK_q)_{q=1}^Q$ ($Q<\infty$) with
 $\widetilde \cK_{q+1}\in \cN_{\tilde\ve(\widetilde \cK_q)}(\widetilde\cK_{q})$ for 
 $1\leq q< Q$ and every sequence $(\cK_n)_{n=0}^N$ adapted to $(\widetilde \cK_q,\tilde\ve(\widetilde \cK_q),\widetilde N(\widetilde \cK_q))_{q=1}^Q$, all configurations
 to be taken in $\IK$:
Let $\mu^1$ and $\mu^2$ be probability measures on $\cM$, with
strictly positive, $\tfrac16$-H\"older continuous densities~$\rho^1$ and~$\rho^2$
with respect to the measure $\cos\varphi\,\rd r\, \rd\varphi$. Given any $\gamma>0$, 
there exist $0<\theta_\gamma<1$ and $C_\gamma>0$ such that
\begin{equation}
\label{eq:convergerate}
\left|\int_{\cM} f\circ \cF_n\, \rd\mu^1 - \int_{\cM} f\circ \cF_n\, \rd\mu^2\right| \leq C_\gamma ({\| f \|}_\infty + {|f|}_\gamma)\theta_\gamma^n,  \qquad n\leq N,
\end{equation}
for  all $\gamma$-H\"older continuous $f:\cM\to\IR$.
The constants $C_\gamma$ and $\theta_\gamma$ depend on the collection
$\{\widetilde \cK_q, 1\leq q\leq Q\}$ {\rm (see Remark \ref{rem:reference_order} below)}; additionally 
$C_\gamma=C_\gamma(\rho^1,\rho^2)$ depends on the densities $\rho^i$ through the H\"older constants of $\log\rho^i$, while $\theta_\gamma$ does not depend 
on the $\mu^i$.
\end{thm}

\begin{remark}\label{rem:reference_order}
We clarify that the constants $C_\gamma$ and $\theta_\gamma$ 
depend on the collection of distinct configurations that appear in the
sequence $(\widetilde \cK_q)_{q=1}^Q$, not on the order in which 
these configurations are listed; in particular, 
each $\widetilde \cK_q$ may appear multiple times. 
This observation is essential for the proofs of Theorems~\ref{thm:weak_conv_compact}--\ref{thm:eq_mixing}.
\end{remark}

\begin{proof}[Proof of Theorem~\ref{thm:weak_conv_compact} assuming Theorem~\ref{thm:weak_conv}]
Given $\IK$, consider a slightly larger $\IK'\supset \bar\IK$, obtained by decreasing $\bar\tau^{\min}$ and $\varphi$ and increasing $\ft$. We apply Theorem 
~\ref{thm:weak_conv} to $\IK'$, obtaining $\tilde\ve(\cK)$ and $\widetilde N(\cK)$
for $\cK \in \IK'$. Since $\bar \IK$ is compact, there exists a finite collection of configurations $(\widetilde\cK_q)_{q\in \cQ}\subset \bar\IK$ such that the sets $\widetilde \cN_q=\cN_{\frac12 \tilde\ve(\widetilde \cK_q)}(\widetilde\cK_{q}) \cap \IK$, $q\in \cQ$, 
form a cover of~$\IK$. Let $\ve_\star = \min_{q\in \cQ} \tilde\ve(\widetilde \cK_q)$
and $N_\star = \max_{q\in\cQ}\widetilde N(\widetilde\cK_q)$. We claim that Theorem~\ref{thm:weak_conv_compact}
holds with $\ve = \ve_\star/(2N_\star)$.
Let $(\cK_n)_{n=0}^N\subset\IK$ with $d(\cK_n, \cK_{n+1})<\ve$ be given. 
Suppose $\cK_0 \in \widetilde \cN_q$.
Then $\cK_i$ is guaranteed to be in $\cN_{\tilde\ve(\widetilde \cK_q)}(\widetilde\cK_{q})$
for all $i < N_\star$. Before the sequence leaves 
$\cN_{\tilde\ve(\widetilde \cK_q)}(\widetilde\cK_{q})$, we select another 
$\widetilde \cN_{q'}$ and repeat the process. 
Thus, the assumptions of Theorem~\ref{thm:weak_conv} are satisfied 
(add more copies of $\cK_N$ at the end if necessary).
 Taking note of Remark~\ref{rem:reference_order}, this yields a uniform rate of memory loss for all sequences. Of course the constants thus obtained for~$\IK$ in Theorem~\ref{thm:weak_conv_compact} are the constants above obtained for the larger~$\IK'$ in Theorem~\ref{thm:weak_conv}.
\end{proof}


{\bf Standing Hypothesis for Sections~\ref{sec:preliminariesI}--\ref{sec:cont_proof}:} We assume $\IK$ as defined 
by $\bar\tau^{\min}$, $\ft$ and $\varphi$ is fixed throughout. For definiteness
we fix also $\bar\beta$, and declare once and for all that all pairs $(\cK, \cK')$ 
for which we consider the billiard map $F_{\cK',\cK}$ are assumed to be admissible, 
as are $(\cK_n, \cK_{n+1})$ in all the sequences $(\cK_n)$ studied. 
These are the only billiard maps we will consider.

\section{Preliminaries I: Geometry of billiard maps}\label{sec:preliminariesI}

In this section, we record some basic facts about time-dependent 
billiard maps related to their hyperbolicity, discontinuities, etc. The results 
here are entirely analogous to the fixed scatterers 
case. They depend on certain geometric facts that are uniform for all the 
billiard maps considered; indeed one does not know from step to step in
the proofs whether
or not the source and target configurations are different. Thus we will state the facts 
but not give the proofs, referring the reader instead
to sources where proofs are easily modified to give the results here. 

An important point is that the estimates
of this section are \emph{uniform}, i.e., the constants in the statements of the lemmas
depend only on $\IK$.

\medskip
\noindent {\it Notation:} Throughout the paper, the length of a smooth curve $W\subset \cM$ is denoted by $|W|$ and the Riemannian measure induced on~$W$ is denoted by~$\fm_W$. Thus, $\fm_W(W)=|W|$. 
We denote by~$\cU_\ve(E)$ the open $\ve$-neighborhood of a set $E$ in the phase space $\cM$. For $x=(r,\varphi) \in \cM$, we denote by~$T_x\cM$ the tangent space of~$\cM$ and by~$D_xF$ the derivative of a map~$F$ at~$x$. Where no ambiguity exists, we sometimes
write~$F$ instead of~$F_{\cK',\cK}$.

\subsection{Hyperbolicity}
Given $(\cK,\cK')$ and a point $x=(r,\varphi) \in \cM$, we let 
$x'=(r',\varphi')=F x$ and compute $D_xF$ as follows:
Let $\kappa(x)$ denote the curvature of $\cup_i\partial B_i$ at the point
corresponding to~$x$, and define $\kappa(x')$ analogously. 
The flight time between $x$ and $x'$ is denoted by $\tau(x)=\tau_{\cK',\cK}(x)$.
Then $D_xF$ is given by 
\beqn 
-\frac{1}{\cos\varphi'}
\begin{pmatrix}
\tau(x) \kappa(x) + \cos\varphi
&
\tau(x)
\\
\tau(x) \kappa(x)\kappa(x') + \kappa(x)\cos\varphi' + \kappa(x')\cos\varphi
&
\tau(x) \kappa(x') + \cos\varphi'
\\
\end{pmatrix}
\eeqn
provided $x\notin F^{-1}\partial\cM$, the discontinuity set of $F$.
This computation is identical to the case with fixed scatterers.
As in the fixed scatterers case, notice that as~$x$ approaches $F^{-1}\partial\cM$, 
$\cos\varphi' \to 0$ and the derivative of the map $F$ blows up.
Notice also that
\beq\label{eq:det}
\det D_x F = \cos\varphi/\cos\varphi',
\eeq
so that $F$ is locally invertible. 

The next result asserts the uniform hyperbolicity of $F$ for orbits
that do not meet $F^{-1}\partial\cM$. 
Let~$\kappa^{\min}$ and~$\kappa^{\max}$ denote the
minimum and maximum curvature of the boundaries of the scatterers~$B_i$.

\begin{lem}[Invariant cones]\label{lem:cones}
The unstable cones
\beqn
\cC_x^u = \{(\rd r,\rd\varphi)\in T_x\cM \,:\, \kappa^{\min}\leq \rd\varphi/\rd r \leq \kappa^{\max}+2/\bar\tau^{\min} \}, \quad x \in \cM,
\eeqn
are $D_xF$-invariant for all pairs $(\cK, \cK')$, i.e., $D_xF (\cC_x^u)\subset 
\cC_{Fx}^u$ for all $x\notin F^{-1}\partial\cM$, and there exist uniform constants $\hat c>0$ and $\Lambda>1$ such that for every $(\cK_n)_{n=0}^N$, 
\beq\label{eq:expansion}
\| D_x\cF_n v \| \geq \hat c \Lambda^n \| v \|
\eeq
for all $n\in\{1,\dots,N\}$, $v\in \cC^u_x$, and $x\notin \cup_{m=1}^N(\cF_m)^{-1}\partial\cM$.

Similarly, the stable cones
\beqn
\cC_x^s = \{(\rd r,\rd\varphi)\in T_x\cM \,:\, -\kappa^{\max}-2/\bar\tau^{\min}\leq \rd\varphi/\rd r \leq  -\kappa^{\min}\}
\eeqn
are $(D_xF)^{-1}$-invariant for all $(\cK, \cK')$, i.e.,  $(D_xF)^{-1} \cC_{Fx}^s\subset \cC_x^s$ for all  $x\notin  \partial\cM \cup F^{-1}\partial\cM$, and for every 
$(\cK_n)_{n=0}^N$, 
\beqn
\| (D_x\cF_n)^{-1} v \| \geq \hat c \Lambda^n \| v \|
\eeqn
for all $n\in\{1,\dots,N\}$, $v\in \cC^s_{\cF_n x}$, and $x\notin \partial\cM\cup \cup_{m=1}^N(\cF_m)^{-1}\partial\cM$.

\end{lem}

Notice that the cones here can be chosen independently of~$x$ and of the 
scatterer configurations involved. The proof follows verbatim that of the fixed scatterers case; see~\cite{ChernovMarkarian_2006}. 

\smallskip

Following convention, we introduce for purposes of controlling distortion
(see Lemma~\ref{lem:distortion}) the homogeneity strips 
\begin{eqnarray*}
\IH_k & = & \{ (r,\varphi)\in\cM\,:\, \pi/2-k^{-2}<\varphi \leq \pi/2-(k+1)^{-2}\}\\
\IH_{-k} &= &\{ (r,\varphi)\in\cM\,:\, -\pi/2+(k+1)^{-2}\leq \varphi < -\pi/2+k^{-2}\}
\end{eqnarray*}
for all integers $k\geq k_0$, where~$k_0$ is a sufficiently large uniform constant. 
It follows, for example, that for each~$k$,~$D_xF$ is uniformly bounded for 
$x\in F^{-1}(\IH_{-k}\cup\IH_{k})$, as
\beq\label{eq:cos_bound}
C_\mathrm{\cos}^{-1}k^{-2}\leq \cos\varphi' \leq C_\mathrm{\cos}k^{-2}
\eeq
for a constant $C_\mathrm{\cos}>0$. 
We will also use the notation
\beqn
\IH_0 = \{ (r,\varphi)\in\cM\,:\, -\pi/2+k_0^2\leq \varphi \leq \pi/2-k_0^{-2}\}\ .
\eeqn


\subsection{Discontinuity sets and homogeneous components}\label{sec:discont_homog}
For each $(\cK, \cK')$, the singularity set $(F_{\cK',\cK})^{-1}\partial\cM$ 
has similar geometry as in the case $\cK'=\cK$. In particular, it is the union 
of finitely many \mbox{$C^2$-smooth} curves which are negatively sloped, and 
there are uniform bounds depending only on $\IK$ for the number of smooth segments 
(as follows from \eqref{eq:flight_times}) and their derivatives. One of the geometric facts,
true for fixed scatterers as for the time-dependent case, that will be useful later is the
following:
Through every point in $F^{-1}\partial\cM$, there is a continuous path in 
$F^{-1}\partial\cM$ that goes monotonically in $\varphi$ from one component 
of $\partial\cM$ to the other.

In our proofs it will be necessary to know that the structure of the
singularity set varies in a controlled way with changing configurations. 
Let us denote
\beqn
\cS_{\cK',\cK} = \partial\cM\cup (F_{\cK',\cK})^{-1}\partial\cM.
\eeqn
If $\cK$ and $\cK'$ are small perturbations of $\widetilde\cK$, then $\cS_{\cK',\cK}$ is contained in a small neighborhood of $\cS_{\widetilde\cK,\widetilde\cK}$ (albeit the topology of $\cS_{\cK',\cK}$ may be slightly different from that of $\cS_{\widetilde\cK,\widetilde\cK}$). A proof of the following result, which suffices for our purposes, is given in the Appendix.

\begin{lem}\label{lem:map_cont}
Given a configuration $\widetilde\cK\in\IK$ and a compact subset $E\subset \cM\setminus \cS_{\widetilde\cK,\widetilde\cK}$, there exists $\delta>0$ such that the map $(x,\cK,\cK')\mapsto F_{\cK',\cK}(x)$ is uniformly continuous on $E \times  \cN_\delta(\widetilde\cK)\times \cN_\delta(\widetilde\cK)$.
\end{lem}

While $F^{-1}\partial \cM$ is the genuine discontinuity set for $F$, for purposes of
distortion control one often treats the preimages of homogeneity lines as though
they were discontinuity curves also. We introduce the following language: 
A set $E \subset \cM$ is said to be {\it homogeneous} if it is 
completely contained in a connected component of one of the~$\IH_k$, $|k| \ge k_0$ or $k=0$. Let 
$E \subset \cM$ be a homogeneous set. 
Then $F(E)$ may have more than one connected component.
We further subdivide each connected component into maximal homogeneous 
subsets and call these the \emph{homogeneous components of~$F(E)$}. 
For $n\geq 2$, the \emph{homogeneous components of $\cF_n(E)$} are defined inductively: Suppose $E_{n-1,i}$, $i\in I_{n-1}$, are the homogeneous components of $\cF_{n-1}(E)$, for some index set $I_{n-1}$ which is at most countable. For each $i\in I_{n-1}$, the set $E_{n-1,i}$ is a homogeneous set, and we can thus define the homogeneous components of the single-step image $F_n (E_{n-1,i})$ as above. The subsets so obtained, for all $i\in I_{n-1}$, are the homogeneous components of $\cF_n(E)$. 
Let $E_{n,i}^- = E \cap \cF_n^{-1}(E_{n,i})$. We call~${\{E_{n,i}^-\}}_i$ 
 \emph{the canonical $n$-step subdivision of~$E$}, leaving the dependence on the sequence implicit when there is no ambiguity.

For $x,y \in \cM$, we define the \emph{separation time} $s(x,y)$ to be the smallest $n\geq 0$ for which $\cF_nx$ and~$\cF_ny$ 
belong in different strips~$\IH_k$ or in different connected components of~$\cM$.
Observe that this definition is~$(\cK_n)$-dependent.


\subsection{Unstable curves}\label{sec:unstable_curves}
A connected $C^2$-smooth curve $W\subset \cM$ is called an {\it unstable curve} 
if $T_x W\subset \cC^u_x$ for every $x\in W$. It follows from the invariant cones condition
 that the image of an unstable curve under $\cF_n$ is a union of unstable curves.
Our unstable curves will be parametrized by $r$: for a curve $W$, 
we write $\varphi=\varphi_W(r)$.

For an unstable curve $W$, define $\hat \kappa_W
= \sup_W|\rd^2\varphi_W/\rd r^2|$.

\begin{lem}\label{lem:curvature}
There exist uniform constants $C_\mathrm{c}>0$ and $\vartheta_\mathrm{c}\in(0,1)$ such that the following holds. Let $W$ and $\cF_n W$ be unstable curves. Then
\beqn
\hat \kappa_{\cF_n W} \leq \frac{C_\mathrm{c}}2(1+\vartheta_\mathrm{c}^n \hat \kappa_W).
\eeqn
\end{lem}

We call an unstable curve $W$ \emph{regular} if it is homogeneous and satisfies 
the curvature bound $\hat \kappa_W \le C_c$. Thus for any unstable curve $W$,
all homogeneous components of $\cF_n(W)$ are regular for large enough $n$.

Given a smooth curve $W\subset\cM$, define 
\beqn
\cJ_W\cF_n(x) = {\|D_x\cF_n v\|}/{\|v\|}
\eeqn
for any nonzero vector $v\in T_x W$.
In other words, $\cJ_W\cF_n$ is the Jacobian of the restriction~$\cF_n|_W$.

\begin{lem}[Distortion bound]\label{lem:distortion}
There exist uniform constants $C_\mathrm{d}'>0$ and $C_\mathrm{d}>0$ such that the following holds. Given $(\cK_n)_{n=0}^N$, if $\cF_n W$ is a homogeneous unstable curve for $0\leq n\leq N$, then
\beq\label{eq:distortion}
C_\mathrm{d}^{-1}\leq e^{-C_\mathrm{d}'|\cF_n W|^{1/3}} \leq  \frac{\cJ_{W}\cF_n (x)}{\cJ_{W}\cF_n (y)} \leq e^{C_\mathrm{d}'|\cF_n W|^{1/3}} \leq C_\mathrm{d}
\eeq
for every pair $x,y\in W$ and $0\leq n\leq N$.
\end{lem}

Finally, we state a result which asserts that very short homogeneous curves 
cannot acquire lengths of order one arbitrarily fast, in spite of the fact that the local
expansion factor is unbounded.

\begin{lem}\label{lem:growth_bound}
There exists a uniform constant $C_\mathrm{e}\geq 1$ such that
\beqn
|\cF_n W| \leq C_\mathrm{e}|W|^{1/2^n},
\eeqn
if $W$ is an unstable curve and $\cF_{m}W$ is homogeneous for $0\leq m<n$.
\end{lem}

The proofs of these results also follow closely those for the fixed scatterers
case. For Lemma~\ref{lem:curvature}, see~\cite{ChernovDolgopyatBBM}. 
For Lemmas~\ref{lem:distortion} and \ref{lem:growth_bound}, see~\cite{ChernovMarkarian_2006}. (Lemma~~\ref{lem:growth_bound} follows readily by iterating the corresponding
one-step bound.)


\subsection{Local stable manifolds}\label{sec:Loc}
Given $(\cK_n)_{n\geq 0}$, a connected smooth curve $W$ is called a {\it homogeneous local stable manifold}, or simply {\it local stable manifold}, 
if the following hold for every $n \ge 0$:

\smallskip

(i) $\cF_n W$ is connected and homogeneous, and

(ii) $T_x(\cF_n W) \subset \cC^s_x$ for every $x \in \cF_n W$.

\smallskip

\noindent It follows from Lemma~\ref{lem:cones} that local stable manifolds are exponentially contracted under $\cF_n$. We stress that unlike unstable curves,
the definition of local stable manifolds
depends strongly on the infinite sequence of billiard maps defined by
$(\cK_n)_{n\geq 0}$.

For $x \in \cM$, let $W^s(x)$ denote the maximal local stable manifold through
$x$ if one exists.
An important result is the \emph{absolute continuity of local stable manifolds}. 
Let  two unstable curves $W^1$ and $W^2$ be given. 
Denote $W^i_\star = \{x\in W^i\,:\, W^s(x)\cap W^{3-i}\neq 0\}$ for $i=1,2$. The map 
$\bfh:W^1_\star\to W^2_\star$ such that $\{\bfh(x)\} = W^s(x)\cap W^2$ for every $x\in W^1_\star$ is called the holonomy map. The Jacobian~$\cJ\bfh$ of the holonomy is the Radon--Nikodym derivative of the pullback $\bfh^{-1}(\fm_{W^2}|_{W^2_*})$ with respect to~$\fm_{W^1}$. The following result gives a uniform bound on the Jacobian almost everywhere on $W^1_\star$. 

\begin{lem}\label{lem:Jac_holonomy}
Let $W^1$ and $W^2$ be regular unstable curves. Suppose $\bfh:W^1_\star\to W^2_\star$ is defined on a positive $\fm_{W^1}$-measure set $W^1_\star\subset W^1$. 
Then for $\fm_{W^1}$-almost every point $x\in W^1_\star$,
\beq\label{eq:Jac_expression}
\cJ\bfh(x) = \lim_{n\to\infty}\frac{\cJ_{W^1}\cF_n(x)}{\cJ_{W^2}\cF_n(\bfh(x))},
\eeq
where the limit exists and is positive with uniform bounds. In fact, there exist uniform constants $A_\bfh>0$ and $C_\bfh>0$ such that the following holds: If $\alpha(x)$ denotes the difference between the slope of the tangent vector of $W^1$ at $x$ and that of $W^2$ at $\bfh(x)$, and if~$\delta(x)$ is the distance between~$x$ and~$\bfh(x)$, then
\beq\label{eq:Jac_bound}
A_\bfh^{-\alpha-\delta^{1/3}}\leq \cJ\bfh \leq A_\bfh^{\alpha+\delta^{1/3}}
\eeq
almost everywhere on $W^1_\star$. Moreover, with $\theta=\Lambda^{-1/6} \in(0,1)$,
\beq\label{eq:Jac_regular} 
|\cJ\bfh(x)-\cJ\bfh(y)|\leq C_\bfh\theta^{s(x,y)}
\eeq
holds for all pairs $(x,y)$ in $W^1_\star$, where $s(x,y)$ is the separation time 
defined in Section~\ref{sec:discont_homog}.
\end{lem}

The proof of Lemma~\ref{lem:Jac_holonomy} follows closely its counterpart for fixed configurations. The identity in \eqref{eq:Jac_expression} is standard for uniformly hyperbolic systems (see \cite{AnosovSinai_1967,Sinai_1970}), as is \eqref{eq:Jac_regular},
except for the use of separation time as a measure of distance in discontinuous
systems; see \cite{Young_1998,ChernovMarkarian_2006}.


\section{Preliminaries II: Evolution of measured unstable curves}\label{sec:preliminariesII}


\subsection{Growth of unstable curves}\label{sec:growth}

Given a sequence $(\cK_m)$, an unstable curve $W$, a point $x\in W$, 
and an integer $n\geq 0$, we denote by $r_{W,n}(x)$ the distance between
$\cF_n x$ and the boundary of the homogeneous component of $\cF_n W$
containing $\cF_n x$.

The following result, known as the Growth lemma, is key in the analysis of billiard dynamics.  It expresses the fact that the expansion of unstable curves dominates the cutting by $\partial\cM\cup\cup_{|k|\geq k_0}\partial\IH_k$, in a uniform fashion for all sequences. The reason behind this fact is that unstable curves expand at a uniform \emph{exponential} rate, whereas the cuts accumulate at tangential collisions. A short unstable curve can meet no more than $\ft/\bar\tau^{\min}$ tangencies in a single time step (see~\eqref{eq:flight_times}), so the number of encountered tangencies grows polynomially with time until a characteristic length has been reached. The proof 
follows verbatim that in the fixed configuration case; see \cite{ChernovMarkarian_2006}.

\begin{lem}[Growth lemma]\label{lem:growth}
There exist uniform constants $C_\mathrm{gr}>0$ and $\vartheta\in(0,1)$ such that, for all (finite or infinite) sequences $(\cK_n)_{n=0}^N$, unstable curves $W$ and $0\leq n\leq N$: 
\beqn
\fm_W\{r_{W,n}<\ve\}\leq C_\mathrm{gr}(\vartheta^n +|W|)\ve\ .
\eeqn
\end{lem}

This lemma has the following interpretation: It gives no information for small $n$
when $|W|$ is small. For  $n$ large enough, such as $n\geq |\log|W||/|\log\vartheta|$, one has $\fm_W\{r_{W,n}<\ve\}\leq 2C_\mathrm{gr}\ve|W|$. In other words, after a sufficiently long time $n$ (depending on the initial curve $W$), \emph{the majority} of points in $W$ have their images in homogeneous components of $\cF_n W$ that are longer than 
$1/(2C_\mathrm{gr})$, and the family of points belonging to shorter ones has a linearly decreasing tail.

\subsection{Measured unstable curves}\label{sec:curves}

A \emph{measured unstable curve} is a pair~$(W,\nu)$ where~$W$ is an unstable curve
and~$\nu$ is a finite Borel measure supported on it. 
Given a sequence $(\cK_n)_{n=0}^\infty$ and a measured unstable curve~$(W,\nu)$ with density $\rho=\rd\nu/\rd\fm_W$, we are interested in the following
dynamical H\"older condition of $\log \rho$: 
For $n\geq 1$, let~${\{W_{n,i}^-\}}_i$ be the canonical $n$-step subdivision of~$W$ as defined in Section~\ref{sec:discont_homog}.

\begin{lem}\label{lem:pair_image}
There exists a constant $C'_\mathrm{r}>0$ for which the following holds:
Suppose $\rho$ is a density on an unstable curve $W$ satisfying $|\log\rho(x)-\log\rho(y)|\leq C\theta^{s(x,y)}$ for all $x,y\in W$. Then, for any homogeneous component $W_{n,i}$, the density $\rho_{n,i}$ of the push-forward of $\nu|_{W_{n,i}}$
by the (invertible) map $\cF_n |_{W_{n,i}^-}$ satisfies
\beq\label{eq:reg_bound}
|\log\rho_{n,i}(x)-\log\rho_{n,i}(y)| \leq \left(\frac{C'_\mathrm{r}}2 + \Bigl(C-\frac{C'_\mathrm{r}}2\Bigr)\theta^n\!\right)\!\theta^{s(x,y)}
\eeq
for all $x,y\in W_{n,i}$. 
\end{lem}

Here $\theta$ is as in Lemma~\ref{lem:Jac_holonomy}.
We fix $C_\mathrm{r} \ge \max\{C'_\mathrm{r}, C_\bfh, 2\}$, where
$C_\bfh$ is also introduced in  Lemma~\ref{lem:Jac_holonomy}, and say a  
measure~$\nu$ supported on an unstable curve~$W$ is \emph{regular} if it is absolutely continuous with respect to~$\fm_W$ and its density $\rho$ satisfies 
\beq\label{eq:density_reg}
|\log\rho(x)-\log\rho(y)| \leq C_\mathrm{r}\theta^{s(x,y)}
\eeq
for all $x,y\in W$. As with $s(\cdot, \cdot)$, the regularity of $\nu$
is $(\cK_n)$-dependent. Notice that under this definition, if a measure on $W$
is regular, then so are its forward images. More precisely, in the notation of Lemma \ref{lem:pair_image}, if $\rho$ is regular, then so is each $\rho_{n,i}$.
We also say the pair $(W,\nu)$ is \emph{regular} if both the unstable curve $W$
and the measure $\nu$ are regular.

\begin{remark}\label{rem:Holder_regular}
The separation time $s(x,y)$ is connected to the Euclidean distance $d_\cM(x,y)$ in the following way. If $x$ and $y$ are connected by an unstable curve $W$, then $|\cF_n W|\geq \hat c\Lambda^n|W|\geq \hat c\Lambda^n d_\cM(x,y)$ for $0\leq n<s(x,y)$. Since $\cF_{s(x,y)-1}W$ is a homogeneous unstable curve, its length is uniformly bounded above. Thus,
\beq\label{eq:d&s}
d_\cM(x,y) \leq C_\mathrm{s}\Lambda^{-s(x,y)} = C_\mathrm{s}\theta^{6 s(x,y)}
\eeq
for a uniform constant $C_\mathrm{s}>0$. In particular, if $\rho$ is a nonnegative density on an unstable curve~$W$ such that $\log\rho$ is H\"older continuous with exponent $1/6$ and constant $C_\mathrm{r}C_\mathrm{s}^{-1/6}$, then~$\rho$ is regular with respect to \emph{any configuration sequence}.
\end{remark}

\begin{proof}[Proof of Lemma~\ref{lem:pair_image}]
Consider $n=1$ and take two points $x,y$ on one of the homogeneous components $W_{1,i}$. Let the corresponding preimages be $x_{-1},y_{-1}$. Since $s(x_{-1},y_{-1})=s(x,y)+1$, the bound \eqref{eq:distortion} yields
\beqn
\begin{split}
|\log\rho_{1,i}(x)-\log\rho_{1,i}(y)| &\leq \left|\log\frac{\rho(x_{-1})}{\cJ_WF_1(x_{-1})}-\log\frac{\rho(y_{-1})}{\cJ_WF_1(y_{-1})}\right| 
\\
& \leq  |\log\rho(x_{-1})-\log\rho(y_{-1})| + |\log \cJ_WF_1(x_{-1})-\log \cJ_WF_1(y_{-1})|
\\
& \leq C\theta^{s(x,y)+1}+ C_\mathrm{d}' |W_{1,i}(x,y)|^{1/3}, 
\end{split}
\eeqn
where $|W_{1,i}(x,y)|$ is the length of the segment of the unstable curve $W_{1,i}$ connecting $x$ and~$y$. Because of the unstable cones, the latter is uniformly proportional to the distance $d_\cM(x,y)$ of $x$ and $y$. Recalling \eqref{eq:d&s}, we thus get $C_\mathrm{d}' |W_{1,i}(x,y)|^{1/3}\leq C_\mathrm{d}''\theta^{2s(x,y)}\leq C_\mathrm{d}''\theta^{s(x,y)}$ for another uniform constant $C_\mathrm{d}''>0$. Let us now pick any $C_\mathrm{r}$ such that $C_\mathrm{r}\geq 2C_\mathrm{d}''/(1-\theta)$. For then $C_\mathrm{r}\theta+C_\mathrm{d}''\leq \frac{1+\theta}{2}C_\mathrm{r}$, and we have $|\log\rho_{1,i}(x)-\log\rho_{1,i}(y)|\leq C'\theta^{s(x,y)}$ with
\beq\label{eq:reg_iteration}
C' \leq C\theta+C_\mathrm{d}'' = (C-C_\mathrm{r})\theta + (C_\mathrm{r}\theta+C_\mathrm{d}'') \leq  (C-C_\mathrm{r})\theta + \frac{1+\theta}{2}C_\mathrm{r} = C\theta + \frac{1-\theta}{2}C_\mathrm{r}. 
\eeq
We may iterate \eqref{eq:reg_iteration} inductively, observing that at each step the constant $C$ obtained at the previous step is contracted by a factor of $\theta$ towards $ \frac{1-\theta}{2}C_\mathrm{r}$. It is now a simple task to obtain~\eqref{eq:reg_bound}. The constant $C_\mathrm{r}$ was chosen so that the image of a regular density is regular and the image of a non-regular density will become regular in finitely many steps.
\end{proof}

The following extension property of \eqref{eq:density_reg} will be necessary. We give a proof in the Appendix.
\begin{lem}\label{lem:regular_extension}
Suppose $W_\star$ is a closed subset of an unstable curve $W$, and that $W_\star$ includes the endpoints of $W$. Assume that the function $\rho$ is defined on $W_\star$ and that there exists a constant $C>0$ such that $|\log\rho(x)-\log\rho(y)| \leq C\theta^{s(x,y)}$ for every pair $(x,y)$ in $W_\star$. Then, $\rho$ can be extended to all of $W$ in such a way that the inequality involving $\log \rho$ above holds on $W$, 
the extension is piecewise constant, $\min_{W}\rho = \min_{W_\star}\rho$, 
and $\max_{W}\rho = \max_{W_\star}\rho$.
\end{lem}

\subsection{Families of measured unstable curves}\label{sec:stacks}
Here we extend the idea of measured unstable curves to measured families of unstable
stacks. We begin with the following definitions:

\smallskip \noindent
(i)  We call $\cup_{\alpha \in E} W_\alpha \subset
\cM$ a \emph{stack of unstable curves}, or simply an \emph{unstable stack}, if $E\subset \IR$ is an open interval, each $W_\alpha$ is an unstable curve, and there
is a map $\psi: [0,1]\times E\to \cM$ which is
a homeomorphism onto its image 
so that, for each $\alpha \in E$, $\psi\bigl([0,1]\times \{\alpha\}\bigr)=W_\alpha$. 

\smallskip \noindent
(ii) The unstable stack $\cup_{\alpha\in E} W_\alpha$ is said to be \emph{regular} if each
$W_\alpha$ is regular as an unstable curve.

\smallskip \noindent
(iii) We call $(\cup_{\alpha\in E} W_\alpha, \mu)$ a \emph{measured unstable stack}
if $U = \cup_{\alpha\in E} W_\alpha$ is an unstable stack and $\mu$ is a finite Borel measure
on $U$. 

\smallskip \noindent
(iv) we say  $(\cup_{\alpha\in E} W_\alpha, \mu)$ is
\emph{regular} if (a) as a stack $\cup_{\alpha\in E} W_\alpha$ is regular and 
(b) the conditional probability measures $\mu_\alpha$ of~$\mu$ on~$W_\alpha$ are regular. More precisely, 
\mbox{$\{W_\alpha,\alpha\in E\}$} is a measurable partition of 
$\cup_{\alpha\in E} W_\alpha$, and $\{\mu_\alpha\}$ is a version of
the disintegration of $\mu$ with respect to this partition, that is to say, 
for any Borel set $B\subset \cM$, we have 
\beqn
\mu(B) = \int_E \mu_\alpha(W_\alpha\cap B)\,\rd P(\alpha)
\eeqn
where $P$ is a finite Borel measure on $I$. The conditional measures $\{\mu_\alpha\}$
are unique up to a set of $P$-measure $0$, and (b) requires that
 $(W_\alpha, \mu_\alpha)$ be regular
in the sense of Section~\ref{sec:curves} for $P$-a.e.~$\alpha$.

\smallskip
Consider next a sequence $(\cK_n)_{n=0}^\infty$ and a fixed~$n\geq 1$. Denote by~$\cD_{n,i}$, $i\geq 1$, the countably many connected components of the set $\cM\setminus \cup_{1\leq m\leq n}(\cF_m)^{-1}(\partial\cM\cup\cup_{|k|\geq k_0}\partial\IH_k)$. In analogy
with unstable curves, we define the \emph{canonical $n$-step subdivision} of a regular 
unstable stack $\cup_\alpha W_\alpha$: Let $(n,i)$ be such that 
$\cup_\alpha W_\alpha \cap \cD_{n,i}\neq\emptyset$, and let
$E_{n,i} = \{\alpha \in E: W_\alpha \cap \cD_{n,i}\neq\emptyset\}$. Pick one of the (finitely our countably many) components $E_{n,i,j}$ of $E_{n,i}$.

We claim $\cup_{\alpha \in E_{n,i,j}} (W_\alpha \cap \overline\cD_{n,i})$ is an unstable stack,
and define 
$$\psi_{n,i,j}: [0,1]\times E_{n,i,j} \ \to \ \cup_{\alpha\in E_{n,i,j}} W_\alpha \cap \overline\cD_{n,i}
$$
as follows: for $\alpha \in E_{n,i}$, $\psi_{n,i}|_{[0,1]\times \{\alpha\}}$ is
equal to $\psi|_{[0,1]\times \{\alpha\}}$ followed by a linear contraction from~$W_\alpha$ to~$W_\alpha \cap \overline\cD_{n,i}$. For this construction to work, it is imperative
that $W_\alpha \cap \cD_{n,i}$ be connected, and that is true, for
by definition, $W_\alpha \cap \cD_{n,i}$ is an element of the canonical $n$-step 
subdivision of $W_\alpha$.
It is also clear that 
$\cup_{\alpha\in E_{n,i,j}} \cF_n(W_\alpha \cap \overline\cD_{n,i})$ is an unstable stack, with 
the defining homeomorphism given by~$\cF_n \circ \psi_{n,i,j}$.

What we have argued in the last paragraph is that the $\cF_n$-image of an unstable
stack $\cup_\alpha W_\alpha$ is the union of at most countably many
unstable stacks. Similarly, the $\cF_n$-image of a measured unstable stack is
the union of measured unstable stacks, and
by Lemmas~\ref{lem:curvature} and~\ref{lem:pair_image}, regular measured unstable stacks are mapped to
unions of regular measured unstable stacks.

\medskip
\begin{figure}[!ht]
\includegraphics[width=0.5\linewidth]{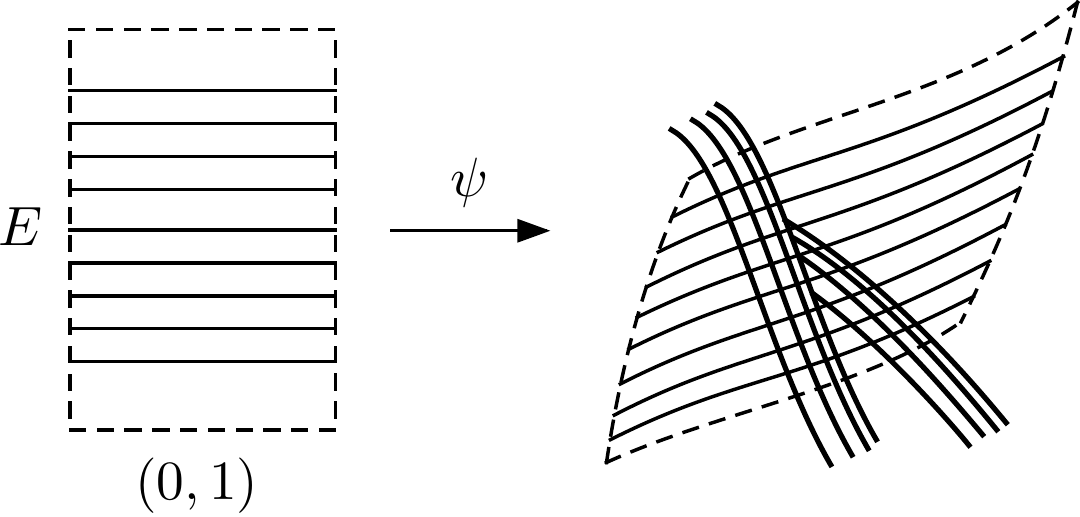}
\caption{A schematic illustration of an unstable stack and its dynamics. The regular unstable curves on the right are the images of the horizontal lines under the homeomorphism $\psi$. The curves with negative slopes represent the countably many branches of the $n$-step singularity set $\cup_{1\leq m\leq n}(\cF_m)^{-1}(\partial\cM\cup\cup_{|k|\geq k_0}\partial\IH_k)$. The canonical $n$-step subdivision of the unstable curves yields countably many unstable stacks.}
\label{fig:stack}
\end{figure}

The discussion above motivates the definition of \emph{measured unstable families},
defined to be convex combinations of measured unstable stacks. That is to say,
we have a countable collection of unstable stacks $\cup_{\alpha \in E_j} W^{(j)}_\alpha$
parametrized by $j$, and a measure $\mu=\sum_j a^{(j)} \mu^{(j)}$ with the property
that for each~$j$, $(\cup_\alpha W^{(j)}_\alpha, \mu^{(j)})$ is a measured unstable stack
and $\sum_j a^{(j)} = 1$. We permit the stacks to overlap, i.e., for $j \ne j'$, we permit
$\cup_\alpha W^{(j)}_\alpha$ and $\cup_{\alpha} W^{(j')}_{\alpha}$ to meet.
This is natural because in the case of moving scatterers, the maps $\cF_n$ are not
one-to-one; even if two stacks have disjoint supports, this property is not retained
by the forward images.
Regularity for measured unstable families is defined similarly. 
The idea of canonical $n$-step subdivision passes
easily to measured unstable families, and we can sum up the discussion by saying
that given $(\cK_n)$, push-forwards of measured unstable families are again
measured unstable families, and regularity is preserved.

\medskip
So far, we have not discussed the lengths of the unstable curves in an unstable
stack or family. Following \cite{ChernovMarkarian_2006}, we introduce, for a measured
unstable family defined by $(\cup_\alpha W^{(j)}_\alpha, \mu^{(j)})$ and 
$\mu=\sum_j a^{(j)} \mu^{(j)}$, the quantity
\beq\label{eq:Z_cont}
\cZ =  \sum_j a^{(j)}\int_{\IR} \frac{1}{|W^{(j)}_\alpha|}\,\rd P^{(j)}(\alpha).
\eeq
Informally, the smaller the value of $\cZ/\mu(\cM)$ the smaller the fraction of $\mu$ supported on short unstable curves.


For $\mu$ as above, let $\cZ_n$ denote the quantity
corresponding to $\cZ$ for the push-forward $\sum_k a_{n,k}\, \mu_{n,k}$ 
of the canonical $n$-step subdivision of $\mu$ discussed earlier. We have
the following control on $\cZ_n$:

\begin{lem}\label{lem:Z}
There exist uniform constants $C_\mathrm{p}>1$ and $\vartheta_\mathrm{p}\in(0,1)$ such that
\beq\label{eq:Z_decay}
\frac{\cZ_n}{\mu(\cM)} \leq \frac{C_\mathrm{p}}2 \! \left(1+\vartheta_\mathrm{p}^n\frac{\cZ}{\mu(\cM)}\right)
\eeq
holds true for any regular measured unstable family.
\end{lem}

This result can be interpreted as saying that given an initial measure $\mu$ which has a high fraction of its mass supported on short unstable curves --- yielding a large value of $\cZ/\mu(\cM)$ --- the mass gets quickly redistributed by the dynamics to the longer homogeneous components of the image measures, so that $\cZ_n/\mu(\cM)$ decreases exponentially, until a level safely below $C_\mathrm{p}$ is reached. As regular densities remain regular (Lemma~\ref{lem:pair_image}), and as the supremum and the infimum of a regular density are uniformly proportional, Lemma~\ref{lem:Z} is a direct consequence of the Growth lemma; see above.
See \cite{ChernovMarkarian_2006} for the fixed configuration case; the time-dependent
case is analogous.

\begin{defn}\label{defn:proper}
A regular measured unstable family is called \emph{proper} if 
$\cZ<C_\mathrm{p}\mu(\cM)$. 
\end{defn}

Notice that $\cZ<C_\mathrm{p}\mu(\cM)$ implies, by Markov's inequality and $\sum_j a^{(j)}P^{(j)}(\IR) =\mu(\cM)$, that
\beq\label{eq:proper_curves}
\sum_j a^{(j)} P^{(j)}\! \left\{\alpha\in\IR\,:\, |W^{(j)}_\alpha |\geq (2C_\mathrm{p})^{-1}\right\} \geq \tfrac 12\mu(\cM).
\eeq
In other words, if~$\mu$ is proper, then at least half of it is supported on unstable curves
of length~$\geq(2C_\mathrm{p})^{-1}$. Notice  also that for any measured unstable family with  $\cZ<\infty$, Lemma~\ref{lem:Z} shows that the push-forward of such a family will
eventually become proper. Starting from a proper family, it is possible that $\cZ_n\geq C_\mathrm{p}\mu(\cM)$ for a finite number of steps; however, \eqref{eq:Z_decay} implies that there exists a uniform constant $n_\mathrm{p}$ such that $\cZ_n<C_\mathrm{p}\mu(\cM)$ for all $n\geq n_\mathrm{p}$.

\begin{remark} \label{summary}
{\rm The results of Sections~\ref{sec:growth}--\ref{sec:stacks} can be summarized as follows:}

\noindent
(i) The $\cF_n$-image of an unstable stack  is the union of
at most countably many such stacks. 

\noindent
(ii) Regular measured unstable stacks are
mapped to unions of the same, and 

\noindent
(iii) the $\cF_n$-image
of a proper measured unstable family is proper for $n \ge n_p$.
\end{remark}

\smallskip
\subsection{Statements of Theorems 1'--2' and 4'}\label{sec:primed_statements}
We are finally ready to give a precise statement of Theorem~1', which permits more
general initial distributions than Theorem~\ref{thm:weak_conv_compact} as stated in Section~\ref{sec:results}. 

\bigskip \noindent
{\bf Theorem 1'.} \ {\it
There exists $\ve>0$ such that the following holds. Let $(\cup_\alpha W^i_\alpha,\mu^i)$, $i=1,2$, be measured unstable stacks. Assume~$\cZ^i<\infty$ and that the conditional densities satisfy $|\log\rho^i_\alpha(x)-\log\rho^i_\alpha(y)|\leq C^i\theta^{s(x,y)}$ for all $x,y\in W^i_\alpha$.
Given $\gamma>0$, there exist $0<\theta_\gamma<1$ and $C_\gamma>0$ such that
\beqn
\left|\int_{\cM} f\circ \cF_n\, \rd\mu^1 - \int_{\cM} f\circ \cF_n\, \rd\mu^2\right| \leq C_\gamma ({\| f \|}_\infty + {|f|}_\gamma)\theta_\gamma^n, \quad n\leq N,
\eeqn
for all sequences $(\cK_n)_{n=0}^N\subset\IK$ ($N\in \IN\cup\{\infty\}$) satisfying $d(\cK_{n-1},\cK_n)<\ve$ for $1\leq n\leq N$, and all
 $\gamma$-H\"older continuous $f:\cM\to\IR$. The constant $C_\gamma$ depends on $\max(C^1,C^2)$ and $\max(\cZ^1,\cZ^2)$, while $\theta_\gamma$ does not.
}

\bigskip
Let us say a finite Borel measure $\mu$ is \emph{regular on unstable curves} if
it admits a representation as the measure in a regular measured unstable family
with $\cZ<\infty$, {\it proper} if additionally it admits a representation 
that is proper. It follows from Lemmas \ref{lem:pair_image} and 
\ref{lem:Z} that in Theorem~1', after a finite number of steps depending
on $C^i$ and $\cZ^i$, the pushforward of $\mu^i$ will become regular
on unstable curves and proper.

For completeness, we provide a proof of the following in the Appendix.

\begin{lem} \label{1'implies1}
Theorem~1' generalizes Theorem~\ref{thm:weak_conv_compact}.
\end{lem}

\smallskip
\noindent
{\bf Theorem 2'.} {\it This is obtained from Theorems~\ref{var_geom} in exactly the same way
as Theorem~1' is obtained from Theorem~\ref{thm:weak_conv_compact}, namely by relaxing the condition on $\mu^i$ as stated.}

\medskip
As a matter of fact, instead of just Theorem~\ref{thm:weak_conv}, the following generalization is proved in Section~\ref{sec:cont_proof}.

\medskip
\noindent
{\bf Theorem 4'.} {\it This is a similar extension of Theorem~\ref{thm:weak_conv}, i.e., of the \emph{local} result, to initial measures as stated in Theorem~1'.}

\smallskip

\medskip
We finish by remarking on our use of measured unstable stacks 
and families:
The primary reason for considering these objects is that in the proof we really work
with measured unstable \emph{curves} and their images under $\cF_n$. 
``Thin enough" stacks 
of unstable curves behave in a way very similar to unstable curves, and are treated
similarly. Other generalizations are made so we can include a larger class of
initial distributions; moreover, to the extent that is possible, it is always 
convenient to work with a class of objects closed under the operations of interest. See Remark \ref{summary}.
We note also that  our formulation here has deviated from \cite{ChernovMarkarian_2006} because of (fixable) measurability issues with their 
formulation. 

In view of the fact that our proof really focuses on curves, we will, for pedagogical 
reasons consider separately the following two cases:
\begin{itemize}
\item[(1)] \emph{The countable case}, in which we assume that each initial distribution
$\mu^i$ is supported on a countable family of unstable curves, 
i.e., the stacks above consist of single curves.
\item[(2)] \emph{The continuous case}, where we allow the $\mu^i$ to be
as in Theorem~1'.
\end{itemize}
For clarity of exposition, we first focus on the countable case, presenting
a synopsis of the proof followed by a complete proof; this is carried out in Sections~\ref{sec:outline}--\ref{sec:proof_countable}. 
Extensions to the continuous case is discussed in Section~\ref{sec:completion}.




\section{Theorem~\ref{thm:weak_conv}: synopsis of proof}\label{sec:outline}

This important section contains a sketch of the proof, from beginning to end, of 
the ``countable case" of Theorem~\ref{thm:weak_conv}; it will serve 
as a guide to the supporting technical estimates in the sections to follow.
We have divided the discussion into four parts: Paragraph A contains an overview
of the coupling argument on which the proof is based. The coupling procedure
itself follows closely \cite{ChernovMarkarian_2006}; it is reviewed in 
Paragraphs B and C. Having an outline of the proof in hand permits us to
isolate the issues directly related to the time-dependent setting;
this is discussed in Paragraph D.
As mentioned in the Introduction, one of the goals of this paper is to
stress the (strong) similarities between stationary dynamics and their time-dependent counterparts, and to highlight at the same time the new issues that need to be
addressed. 

For simplicity of notation, we will limit the discussion here to the ``countable case" 
of Theorem~\ref{thm:weak_conv}. 
That is to say, we assume throughout that the initial distributions
$\mu^i, i=1,2$, are proper measures supported on a countable number of 
unstable curves; see Section~\ref{sec:stacks}.


\bigskip
\noindent \textbf{A. Overview of coupling argument}

\medskip
The following scheme is used to produce the exponential
bound in Theorem~\ref{thm:weak_conv}. 
Let $(\cK_n), n \le N \le \infty$  be an admissible (finite or infinite) 
sequence of configurations with associated
composed maps $\cF_n = F_n \circ \dots \circ F_1$. 
Given initial probability distributions $\mu^1$ and $\mu^2$ on $\cM$, 
we will produce two sequences of nonnegative measures $\bar\mu^i_n, \ n\le N$, 
with properties (i)--(iii) below:
\vskip 2mm
\begin{itemize}
\item[(i)] for $i=1,2$, $\mu^i = \sum_{j\leq n} \bar\mu^i_j+\mu^i_n$ with
$\bar\mu^1_n(\cM) = \bar\mu^2_n(\cM)$ for each $n$;
\vskip 1mm
\item[(ii)]
$\mu^1_n(\cM)=\mu^2_n(\cM)\leq Ce^{-a n}$;
\vskip 1mm
\item[(iii)]  $|\int f\circ \cF_{n+m} \,\rd \bar\mu^1_n - \int f\circ \cF_{n+m} \,\rd\bar\mu^2_n|
\leq C_fe^{-a m} \bar\mu^i_n(\cM)$, for any test function $f$.  
\end{itemize}
\vskip 2mm
Here $\bar\mu^i_n, i=1,2$, are the components of $\mu^i$ {\it coupled} at time $n$;
their relationship from time $n$ on is given by (iii). By (ii), the
yet-to-be-coupled part decays exponentially. In practice, a coupling occurs at a sequence of times $0<t_1<t_2<\cdots<t_K<N$. In particular, $\bar\mu^i_j=0$, when $j\neq t_k$ for all $1\leq k\leq K$, which means that $\mu_n^i$ remains unchanged between successive coupling times.

It follows from (i)--(iii) above that
\beq\label{eq:diff}
\begin{split}
\left|\int f \circ \cF_n \,\rd\mu^1 - \int f \circ \cF_n \,\rd\mu^2 \right|
 & \leq  2 \|f\|_\infty \cdot \mu^i_{n/2}(\cM)  + 
\sum_{j \leq n/2} C_f e^{-a (n-j)} \bar\mu^i_j(\cM)\ 
\\
& \leq 2 \|f\|_\infty C e^{-a n/2} + 
C_f e^{-a n/2}.
\end{split}
\eeq

We indicate briefly below how, at time $n$ where $n=t_k$ is a coupling time,
we extract $\bar\mu^i_n$ from~$\mu^i_{n-1}$ and couple 
$\bar\mu^1_n$ to $\bar\mu^2_n$. Recall that in the hypotheses of Theorem~\ref{thm:weak_conv}, 
$(\cK_m)_{m=0}^N$ is adapted to $(\widetilde \cK_q,\tilde\ve(\widetilde \cK_q),\widetilde N(\widetilde \cK_q))_{q=1}^Q$. We assume $\cK_n \in \cN_{\tilde\ve(\widetilde \cK_q)}(\widetilde \cK_q)$ for some $q$. In fact, coupling times are chosen so that
$\cK_m$ is in the same neighborhood for a large number of $m \le n$ leading up to $n$.
For simplicity, we write $\widetilde \cK=\widetilde \cK_q$, and 
$\widetilde F = F_{\widetilde \cK, \widetilde \cK}$.

For the coupling at time $n$, we construct a coupling set $\fS_n\subset \cM$
analogous to the ``magnets" in \cite{ChernovMarkarian_2006} --- except that 
it is a time-dependent object. Specifically, 
$$\fS_n=\cup_{x\in\widetilde W_{n}} W^s_{n}(x)$$ 
where $\widetilde W$ is a piece of unstable manifold of $\widetilde F$
(here we mean unstable manifold of a fixed map in the usual sense 
and not just an ``unstable curve" as defined in Section~\ref{sec:unstable_curves}), 
$\widetilde W_{n} \subset \widetilde W$ is a Cantor subset with
$\fm_{\widetilde W}(\widetilde W_{n})/|\widetilde W|\geq \frac{99}{100}$,
and $W^s_{n}(x)$ is the stable manifold of length $\approx |\widetilde W|$ centered at $x$ 
for the sequence $F_{n+1},F_{n+2},\dots$ (if $N< \infty$, let $\cK_m=\cK_N$ 
for all $m>N$).

It will be shown that at time $n$, the $\cF_n$-image of each of the measures 
$\mu^i_{n-1}, \ i=1,2$, is again the union of countably many regular measures
supported on unstable curves. Temporarily let us denote by
$\widetilde \nu^i_n$ the part of $(\cF_n)_*\mu^i_{n-1}$ 
that is supported on unstable curves each one of which crosses~$\fS_n$ 
in a suitable fashion, meeting every $W^s_{n}(x)$ in particular. 
We then show that there is a lower bound (independent
of $n$) on the fraction of $(\cF_n)_*\mu^i_{n-1}$ that $\widetilde \nu^i_n$ comprises,
and couple a fraction of $\widetilde \nu^1_n$ to $\widetilde \nu^2_n$ 
by matching points that lie on the same local stable manifold.

We comment on our construction of $\fS_n$: Given that $F_m$ is close to 
$\widetilde F$ for many $m$ before $n$, $\cF_n$-images of unstable curves 
will be roughly aligned with unstable manifolds of $\widetilde F$, hence our choice
of~$\widetilde W$. In order to achieve the type of relation in (iii) above, 
we need to have $|\cF_{n+m}(x)-\cF_{n+m}(y)| \to 0$ exponentially in $m$ 
for two points $x$ and $y$ ``matched" in our coupling at time $n$, hence
our choice of $W^s_n$. Observe that  in our setting, 
the ``magnets" $\fS_n$ are necessarily time-dependent. 

\medskip
To further pinpoint what needs to be done, it is necessary to better acquaint ourselves 
with the coupling procedure. For simplicity, we assume in 
Paragraphs B and C below that {\it all the configurations in question lie 
in a small neighborhood $\cN_{\tilde\ve}(\widetilde\cK)$ of a single reference 
configuration~$\widetilde \cK$}. As noted earlier, 
details of this procedure follow \cite{ChernovMarkarian_2006}. 
We review it to set the stage both for the discussion in Paragraph D and
for the technical estimates in the sections to follow.

\bigskip
\noindent \textbf{B. Building block of procedure: 
coupling of two measured unstable curves}

\medskip

We assume in this paragraph that $\mu^i, i=1,2$, is supported on a homogeneous 
unstable curve~$W^i$, and that the following hold at some time~$n \ge 0$:~(a) the image~$\cF_m W^i$ is a homogeneous 
unstable curve for $1 \le m \le n$; (b) the push-forward measure $(\cF_n)_*\mu^i=\nu^i_n$ has a regular density $\rho^i_n$ on~$\cF_n W^i=W^i_n$; and (c) $W^i_n$ crosses the magnet $\fS_n$ ``properly", which means roughly that (i) it meets each stable manifold $W^s_n(x)$, $x\in \widetilde W_{n}$, (ii) the excess pieces sticking ``outside'' of the
magnet~$\fS_n$ are sufficiently long, and (iii) part of~$W^i_n$ is very close to
and nearly perfectly aligned with~$\widetilde W$ (for a precise definition of a proper
crossing, see Definition~\ref{defn:proper_crossing}). 

Due to their regularity, the probability densities $\rho^i_n$ are strictly positive. Moreover, the holonomy map $\bfh^{1,2}_n:W^1_n\cap \fS_n\to W^2_n\cap \fS_n$ has bounds on its Jacobian (Section~\ref{sec:Loc}). Thus, we may extract a fraction~$\bar\nu^i_n$ from each measure 
$\nu^i_n|_{(W^i_n \cap\fS_n)}$ with $(\bfh^{1,2}_n)_*\bar\nu^1_n = \bar\nu^2_n$ and $\bar\nu^1_n(\cM)=\bar\nu^2_n(\cM)=\zeta$ for some $\zeta>0$. Because each $x\in W^1_n\cap\fS_n$ lies on the same stable manifold as $\bfh^{1,2}_n(x)\in W^2_n\cap\fS_n$, 
\beq\label{eq:memloss_curve}
\begin{split}
& \left|\int f\circ \cF_{n+m,n+1} \, \rd\bar\nu^1_n - \int f\circ \cF_{n+m,n+1} \, \rd\bar\nu^2_n\right|
\\
& 
\qquad\qquad\qquad
= \left|\int f\circ \cF_{n+m,n+1} \, \rd\bar\nu^1_n - \int f\circ \cF_{n+m,n+1}\circ \bfh^{1,2}_n \, \rd\bar\nu^1_n\right|
\leq |f|_\gamma (\hat c^{-1}\Lambda^{-m})^{\gamma} \zeta
\end{split}
\eeq
by \eqref{eq:expansion}, for all $\gamma$-H\"older functions $f$ and $m\geq 0$. (We have assumed that the local stable manifolds associated with the holonomy map have lengths $\leq 1$.) The splitting in Paragraph A is given by $\mu^i=\mu^i_n+\bar\mu^i_n$, 
where $(\cF_n)_*\bar\mu^i_n = \bar\nu^i_n$ corresponds to the part coupled at time $n$
and $(\cF_n)_*\mu^i_n=\nu^i_n-\bar\nu^i_n$ to the part that remains uncoupled. For $0\leq m<n$, we set $\bar\mu_m^i=0$ and $\mu_m^i=\mu^i$.

In more detail, it is in fact convenient to couple each measure to a reference measure~$\widetilde\fm_n$ supported on $\widetilde W$: Once two measures are coupled to the same reference measure, they are also coupled to each other. 
Define the uniform probability measure 
\beq\label{eq:ref_meas}
\widetilde\fm_n(\,\cdot\,) = \fm_{\widetilde W}(\,\cdot\,\cap\fS_n)/\fm_{\widetilde W}(\widetilde W\cap \fS_n)
\eeq
on $\widetilde W \cap\fS_n$ and write $\bfh_n^i$ for the holonomy map $\widetilde W \cap\fS_n\to W^i_n \cap\fS_n$. Then $(\bfh_n^i)_*\widetilde\fm_n$ is a probability measure on $W^i_n \cap\fS_n$. We assume that $\bfh = \bfh_n^i$ satisfies 
\beq\label{eq:hol_distort}
|(\cJ \bfh\circ \bfh^{-1})^{-1}-1|\leq \tfrac 1{10}\ .
\eeq
By the regularity of the probability densities $\rho^i_n$, there exists a number $\zeta>0$ such that
\beq\label{eq:cross_lowerbound}
\nu^i_n(W^i_n\cap\fS_n) \geq 2\zeta e^{C_\mathrm{r}}\qquad(i=1,2).
\eeq
Setting $\bar\nu^i_n = \zeta (\bfh_n^i)_*\widetilde\fm_n$, we have 
$\bar\nu^1_n(\cM)=\bar\nu^2_n(\cM)= \zeta.$ Let $\bar\rho^i_n$ be the density
of $\bar\nu^i_n$ (so that it is supported on $W^i_n\cap\fS_n$).
By \eqref{eq:hol_distort} and \eqref{eq:cross_lowerbound}, one checks that
$
\sup \bar\rho^i_n\leq \frac58\cdot\inf_{W^i_n}\rho^i_n,
$
so that what we couple is strictly a fraction of $\nu^i_n$.

In preparation for future couplings, we look at $\nu^i_n-\bar\nu^i_n$,
the images of the uncoupled parts of the measures. First, $W^i_n$ can be expressed as the union of $W^i_n\cap \fS_n$ and $W^i_n\setminus \fS_n$, the latter consisting of countably many gaps ``inside'' the magnet $\fS_n$ and two excess pieces sticking ``outside'' of it. 
Moreover, there is a one-to-one correspondence between the gaps $V^i\subset W^i_n\setminus \fS_n$ and the gaps $\widetilde V\subset \widetilde W\setminus\widetilde W_{n}$.
Notice that $\nu^i_n-\bar\nu^i_n$ has a positive density bounded away from zero 
on the curve~$W_n^i$, but that density is not regular as $\bar\nu^i_n$ is only 
supported on the Cantor set $W^i_n \cap\fS_n$. 
We decompose $\nu^i_n-\bar\nu^i_n$ 
as follows: First we separate the part that lies on the excess pieces of $W^i_n$ ``outside'' the magnet. Let $(W^i_n)'$ denote the curve that remains. 
Viewed as a density on
 $W^i_n \cap\fS_n$, $\bar\rho^i_n$ is regular, since $C_\bfh\leq C_\mathrm{r}$. It can be continued to a regular density on all of $(W^i_n)'$ without affecting its bounds (Lemma~\ref{lem:regular_extension}).
Letting $\check\rho^i_n$ denote this extension,
we have $$(\rho_n^i-\bar\rho_n^i)|_{(W^i_n)'} \ = \ 
(\rho_n^i-\check\rho_n^i) |_{(W^i_n)'} \ + 
\sum_{V^i\subset (W^i_n)'\setminus\fS_n}\mathbf{1}_{V^i}\check\rho_n^i \ , $$
where the sum runs over the gaps $V^i$ in $(W^i_n)'$. 
Notice that each of the densities $\check\rho_n^i |_{V^i}$ on the gaps is regular. While 
$(\rho_n^i-\bar\rho_n^i) |_{(W^i_n)'}$ is in general not regular, 
it is not far from regular because both $\rho_n^i$ and $\check\rho_n^i$ are regular
and $(\rho_n^i-\check\rho_n^i)> \frac38 \rho^i_n$.

\bigskip
\noindent \textbf{C. The general procedure}

\medskip
Still assuming that all $\cK_n$ lie in a small neighborhood $\cN_{\tilde\ve}(\widetilde\cK)$ of a single reference configuration~$\widetilde \cK$, we now consider a proper initial probability measure $\mu=\sum_{\alpha\in\cA}\nu_\alpha$, consisting of countably many regular measures $\nu_\alpha$, each supported on a regular unstable curve $W_\alpha$. 
As explained in Paragraph~B, the problem is reduced to coupling a single initial
distribution to reference measures on $\widetilde W$.
Leaving the determination of suitable coupling times
$t_1<t_2<\cdots$ for later, we first discuss what happens at the first coupling:

\medskip
\noindent{\it The first coupling at time $n=t_1$.} 
Denote by $W_{\alpha,n,i}$ the components of
 $\cF_n W_\alpha$ resulting from its canonical subdivision, where~$i$ runs over an at-most-countable index set. Similarly, the push-forward measure is $(\cF_n)_*\nu_\alpha=\sum_i\nu_{\alpha,n,i}$, where each $\nu_{\alpha,n,i}$ is supported on $W_{\alpha,n,i}$.
  As before, each $\nu_{\alpha,n,i}$ has a regular density~$\rho_{\alpha,n,i}$ on~$W_{\alpha,n,i}$.

For each $\alpha\in\cA$, let $\cI_{\alpha,n}$ be the set of indices $i$ for which $W_{\alpha,n,i}$ crosses $\fS_n$ properly, as discussed earlier. This set is finite, as $|\cF_n W_\alpha|<\infty$ and $|W_{\alpha,n,i}|$ is uniformly bounded from below (by the width of the magnet) for~$i\in\cI_{\alpha,n}$.

Let $\zeta_1\in(0,1)$ be such that
\beq\label{eq:cross_bound}
\sum_{\alpha\in\cA}\sum_{i\in\cI_{\alpha,n}}\nu_{\alpha,n,i}(W_{\alpha,n,i}\cap\fS_n) \geq 2\zeta_1 e^{C_\mathrm{r}}.
\eeq
As in~\eqref{eq:ref_meas}, let $\widetilde\fm_n$ denote the uniform probability measure on $\widetilde W \cap\fS_n$ and write $\bfh_{\alpha,n,i}$ for the holonomy map $\widetilde W \cap\fS_n\to W_{\alpha,n,i} \cap\fS_n$, for each $i\in\cI_{\alpha,n}$. Then $(\bfh_{\alpha,n,i})_*\widetilde\fm_n$ is a probability measure on $W_{\alpha,n,i} \cap\fS_n$ which is regular and nearly uniform. We set
\beq\label{eq:component}
\bar\nu_{\alpha,n,i} = \lambda_{\alpha,n,i}\,(\bfh_{\alpha,n,i})_*\widetilde\fm_n,
\eeq
for each $i\in\cI_{\alpha,n}$, where 
\beqn
\lambda_{\alpha,n,i} = \zeta_1\cdot 
\frac{\nu_{\alpha,n,i}(W_{\alpha,n,i}\cap\fS_n)}
{\sum_{\beta\in\cA}\sum_{j\in\cI_{\beta,n}}\nu_{\beta,n,j}(W_{\beta,n,j}\cap \fS_n)}.
\eeqn
Then
\beqn
\sum_{\alpha\in\cA}\sum_{i\in\cI_{\alpha,n}}\bar\nu_{\alpha,n,i}(\cM)=\sum_{\alpha\in\cA}\sum_{i\in\cI_{\alpha,n}}\lambda_{\alpha,n,i} = \zeta_1 \ .
\eeqn
Moreover, the density~$\bar\rho_{\alpha,n,i}$ of~$\bar\nu_{\alpha,n,i}$ on $W_{\alpha,n,i}\cap\fS_n$ is regular (and in fact nearly constant). As in Paragraph~B, the density can be extended in a regularity preserving way (Lemma~\ref{lem:regular_extension}) to the curve $(W_{\alpha,n,i})'$ obtained from~$W_{\alpha,n,i}$ by cropping the excess pieces outside the magnet. We denote the extension by $\check\rho_{\alpha,n,i}$. 
As before, $(\rho_{\alpha,n,i}-\check\rho_{\alpha,n,i})|_{(W_{\alpha,n,i})'}$ is generally
not regular, and to control it, we record the following bounds:

\begin{lem}\label{lem:rhobar_bound}  For each $\alpha \in \cA$ and $i \in \cI_{\alpha,n}$,
\beq\label{eq:rhobar_bound} 
 \tfrac45 \zeta_1 e^{-C_\mathrm{r}} \cdot \sup_{W_{\alpha,n,i}}\rho_{\alpha,n,i}\leq \inf_{(W_{\alpha,n,i})'}\check\rho_{\alpha,n,i}\leq \sup_{(W_{\alpha,n,i})'}\check\rho_{\alpha,n,i} \leq \tfrac58\cdot \inf_{W_{\alpha,n,i}}\rho_{\alpha,n,i} \ .
\eeq
\end{lem}
\begin{proof}
We begin by observing that the density of $(\bfh_{\alpha,n,i})_*\widetilde\fm_n$ on $W_{\alpha,n,i}\cap\fS_n$ has the expression $(\fm_{\widetilde W}(\widetilde W\cap \fS_n)\,\cJ \bfh_{\alpha,n,i}\circ (\bfh_{\alpha,n,i})^{-1})^{-1}$ and that the supremum of a regular density is bounded by~$e^{C_\mathrm{r}}$ times its infimum. By~\eqref{eq:hol_distort} and~\eqref{eq:cross_bound}, the third inequality in \eqref{eq:rhobar_bound} follows easily. Coming to the first inequality in \eqref{eq:rhobar_bound}, it is certainly the case that
\beqn
\sum_{\beta\in\cA}\sum_{j\in\cI_{\beta,n}}\nu_{\beta,n,j}(W_{\beta,n,j}\cap \fS_n) \leq \mu(\cM) \leq 1.
\eeqn
As the density $\rho_{\alpha,n,i}$ is regular,
\beqn
\begin{split}
\lambda_{\alpha,n,i}  &\geq \zeta\cdot \nu_{\alpha,n,i}(W_{\alpha,n,i}\cap\fS_n)
\geq \zeta e^{-C_\mathrm{r}} \cdot \sup_{W_{\alpha,n,i}}\rho_{\alpha,n,i}\cdot \fm_{W_{\alpha,n,i}}(W_{\alpha,n,i}\cap\fS_n).
\end{split}
\eeqn
Again by \eqref{eq:hol_distort},
\beqn
\begin{split}
\inf_{(W_{\alpha,n,i})'}\check\rho_{\alpha,n,i} 
&= \lambda_{\alpha,n,i} \cdot \inf_{W_{\alpha,n,i}\cap\fS_n} (\fm_{\widetilde W}(\widetilde W\cap \fS_n)\,\cJ \bfh_{\alpha,n,i}\circ (\bfh_{\alpha,n,i})^{-1})^{-1}
\geq
\tfrac9{10}\frac{\lambda_{\alpha,n,i}}{\fm_{\widetilde W}(\widetilde W\cap \fS_n)}
\\
&\geq
\tfrac9{10}\zeta e^{-C_\mathrm{r}} \cdot \frac{\fm_{W_{\alpha,n,i}}(W_{\alpha,n,i}\cap\fS_n)}{\fm_{\widetilde W}(\widetilde W\cap \fS_n)} \cdot \sup_{W_{\alpha,n,i}}\rho_{\alpha,n,i}
\geq
\bigl (\tfrac9{10} \bigr)^2 \zeta e^{-C_\mathrm{r}} \cdot \sup_{W_{\alpha,n,i}}\rho_{\alpha,n,i}\ . 
\end{split}
\eeqn
This finishes the proof.
\end{proof}

To recapitulate, in the language of Paragraph~A, we have split $\mu$ into $\bar\mu_n+\mu_n$ with
\beq\label{eq:components}
(\cF_n)_*\bar\mu_n = \sum_{\alpha\in\cA}\sum_{i\in\cI_{\alpha,n}} \bar\nu_{\alpha,n,i}\ .
\eeq
The measures $(\cF_n)_*\bar\mu_n$ and $\zeta_1\, \widetilde\fm_n$ are coupled.

\medskip
\noindent{\it Going forward.} To proceed inductively, we need to discuss the uncoupled part $\mu_n$ (for $n=t_1$), which has the form
$$
(\cF_n)_*\mu_n = \sum_{\alpha\in\cA}\sum_{i\notin\cI_{\alpha,n}} \nu_{\alpha,n,i} + \sum_{\alpha\in\cA}\sum_{i\in\cI_{\alpha,n}} (\nu_{\alpha,n,i}-\bar\nu_{\alpha,n,i}).
$$
The measures $\nu_{\alpha,n,i}$ in the first term above are regular, so we leave them alone. The measures $\nu_{\alpha,n,i}-\bar\nu_{\alpha,n,i}$ in the second term are further subdivided as in Paragraph~B, into the regular densities on the excess pieces, $\check\rho_{\alpha,n,i}\mathbf{1}_V$ on the gaps $V\subset (W_{\alpha,n,i})'\setminus\fS_n$ of the Cantor sets~$\fS_n\cap W_{\alpha,n,i}$, and 
$(\rho_{\alpha,n,i}-\check\rho_{\alpha,n,i})|_{(W_{\alpha,n,i})'}$, which are in general
not quite regular. Because of the arbitrarily small gaps in $(W_{\alpha,n,i})'\setminus\fS_n$,
the resulting family is \emph{not} proper.

We allow for a recovery period of $r_1>0$ time steps during which canonical
subdivisions continue but no coupling takes place. The purpose of this period is to
allow the regularity of densities of the type $(\rho_{\alpha,n,i}-\check\rho_{\alpha,n,i})|_{(W_{\alpha,n,i})'}$ to be restored, and short curves to become longer on average
(as a result of the hyperbolicity). 
Because of the arbitrarily short gaps,
a fraction of the measure will not recover sufficiently to become proper
no matter how long we wait, but this fraction decreases exponentially with time.
Specifically, for all sufficiently large $m$, the $m$-step push-forward $(\cF_{t_1+m})_*\mu_{t_1}$ of the uncoupled measure $(\cF_{t_1})_*\mu_{t_1}$ can be split into the sum of two measures $\mu_{{t_1},m}^\mathrm{P}$ and $\mu_{{t_1},m}^\mathrm{G}$, both consisting of countably many regular measured unstable curves, such that $\mu_{{t_1},m}^\mathrm{P}$ is a \emph{proper} measure and $\mu_{{t_1},m}^\mathrm{G}(\cM) = C_1 \lambda_1^m$ for some $C_1\geq 1$ and $\lambda_1\in(0,1)$. Choosing $r_1$ large enough, $\mu_{{t_1},r_1}^\mathrm{G}(\cM)$ is thus as small as we wish.  

At time $t_1+r_1$ we are left with a proper measure $\mu_{t_1,r_1}^\mathrm{P}$ 
having total mass $1-\zeta_1-C_1\lambda_1^{r_1}$, and another measure 
$\mu_{t_1,r_1}^\mathrm{G}$ supported on a countable union of short curves. 
We consider $\mu_{t_1,r_1}^\mathrm{P}$, and assume 
that after \mbox{$s_1>0$} steps a sufficiently large fraction of the 
push-forward of this measure crosses the magnet ``properly". 
At time $t_2 = t_1+r_1+s_1$, we perform another coupling in the same fashion as
the one performed at time $t_1$, this time coupling a $\zeta_2$-fraction of 
$(\cF_{t_2,t_1})_*\mu_{t_1,r_1}^\mathrm{P}$ to the measure 
$\zeta_2(1-\zeta_1-C_1\lambda_1^{r_1})\, \widetilde \fm_{t_2}$.

The cycle is repeated: Following a recovery period of length $r_2$, i.e.,
at time $t_2+r_2$, the measure of mass $(1-\zeta_2)(1-\zeta_1-C_1\lambda_1^{r_1})$ left from the second coupling can be split into a proper part $\mu^\mathrm{P}_{t_2,r_2}$ 
and a non-proper part $\mu^\mathrm{G}_{t_2,r_2}$, the latter
having mass $C_2\lambda_2^{r_2}(1-\zeta_1-C_1\lambda_1^{r_1})$. At the same time, most of
$\mu_{{t_1},r_1}^\mathrm{G}$ has now become proper: the fraction of
$\mu_{{t_1},r_1}^\mathrm{G}$ that still has not recovered at time $t_2+r_2$ has
mass $C_1\lambda_1^{r_2+(t_2-t_1)}$. We wait another $s_2$ steps, 
until time $t_3=t_2+r_2+s_2$, for a sufficiently large fraction of the push-forward
measure to cross the magnet properly. At time $t_3$, we couple
a $\zeta_3$-fraction of $(\cF_{t_3,t_2})_*\mu_{t_2,r_2}^\mathrm{P}$ plus 
the $\cF_{t_3,t_1}$-image of the part of $\mu_{{t_1},r_1}^\mathrm{G}$ that
has recovered,
to a measure on $\widetilde W$, and so on.

\medskip
{\it Our main challenge is to prove that the estimates above are uniform}, i.e., there exist $C\geq 1$,
$\zeta, \lambda \in (0,1)$ and $r, s \in \mathbb Z^+$, independently of the sequence
$(\cK_n)$ provided each $\cK_n \in \cN_{\tilde\ve}(\widetilde\cK)$, so that
the scheme above can be carried out with
$C_i = C$, $\zeta_i=\zeta, r_i=r, \lambda_i=\lambda$ and $s_i=s$ for all~$i$. 
Assuming these uniform estimates, 
the situation for $\cK_n \in \cN_{\tilde\ve}(\widetilde\cK)$, all $n$, can
be summarized as follows:

\medskip
\noindent \emph{Summary.} We push forward the initial distribution,
performing couplings with the aid of a time-dependent ``magnet" 
at times $t_1<t_2 < \dots$, and performing
canonical subdivisions (for connectedness and distortion control) in between.
The $t_k$'s are $r+s$ steps apart, with $t_1$ depending additionally on
the initial distribution $\mu$.
At each coupling time $t_k$, a $\zeta$-fraction of the uncoupled measure that is proper
is coupled. At the same time, a small fraction of the still uncoupled measure
becomes non-proper due to the small gaps in the magnet. This non-proper part regains ``properness" thereby returning
to circulation exponentially fast, the exceptional set constituting a fraction 
$C\lambda^m$ after $m$ steps. Simple arithmetic 
shows that by such a scheme,
the yet-to-be coupled part of $\mu$ has exponentially small mass.
This implies exponential memory loss.

\bigskip
\noindent \textbf{D. What makes the proof work in the time-dependent case}

\medskip
We now return to the full setting of Theorem~\ref{thm:weak_conv}, where we are handed a sequence
$(\cK_n)_{n=0}^N$ adapted to $(\widetilde \cK_q,\tilde\ve(\widetilde \cK_q),\widetilde N(\widetilde \cK_q))_{q=1}^Q$. Exponential memory loss of this sequence must necessarily
come from the corresponding property for $\widetilde F_q = 
F_{\widetilde \cK_q, \widetilde \cK_q}$. The question is: how does the exponential
mixing property of a system pass to compositions of nearby systems? 
Such a result cannot be taken for granted,
for in general mixing involves sets of all sizes, and smaller sets naturally take longer
to mix, while two systems that are a positive
distance apart will have trajectories that follow each other \emph{up to finite precision
for finite times}.  That is to say, 
once the neighborhood is fixed, perturbative arguments are
not effective for treating arbitrarily small scales. 
These comments apply to iterations of fixed maps as well 
as time-dependent sequences.

What is special about our situation is that there is a characteristic length $\ell$ 
to which images of all unstable curves tend to grow exponentially fast under $F^n$,
for all $F=F_{\cK,\cK}, \cK \in \IK$, before they get cut --- with the exception of exponentially 
small sets (see Lemma~\ref{lem:growth}). The presence of such a characteristic
length suggests that to prove exponential mixing, it may suffice to consider 
rectangles aligned with stable and unstable directions that are $\ge \ell$ in size,
and to treat separately growth properties starting from arbitrarily small length scales. 
These ideas have been used successfully to prove exponential correlation decay for 
classical billiards, and will be used here as well.\footnote{The ideas alluded to here
are applicable to large classes of dynamical systems with some hyperbolic properties
including but not limited to
billiards; they were enunciated in some generality in \cite{Young_1998}, which also
proved exponential correlation decay for the periodic Lorentz gas.}

\medskip
To carry out the program outlined in Paragraphs A--C, we need to prove that for each
$\widetilde \cK_q$, the following holds, with uniform bounds, for all 
$(\cK_n)$ in a sufficiently small neighborhood of $\widetilde \cK_q$:

\medskip \noindent
(1) {\it Uniform mixing on finite scales.} We will show that there is a uniform lower 
bound on the speeds of mixing for rectangles of sides $\ge \ell$ for the time-dependent
maps defined by $(\cK_n)$. For $\widetilde F_q = F_{\widetilde \cK_q,\widetilde \cK_q}$, this is proved in \cite{BSC_1990,BSC_1991}, and what we prove here is effectively
a perturbative version for time-dependent sequences in a small enough neighborhood of
$\widetilde F_q$. Such a result is 
feasible because it involves only finite-size objects for finite times. Caution must be
exercised still, as the maps involved are discontinuous. 
This result gives the $s=s(\widetilde \cK_q)$ asserted in  Paragraph C.

\medskip \noindent
(2) {\it Uniform structure of magnets.} To ensure that a definite fraction of measure is
coupled when a measured unstable curve crosses the magnet, a uniform lower bound
on the density 
of local stable manifolds in $\fS_n$ is essential: we require 
$\fm_{\widetilde W}(\widetilde W_{n})/|\widetilde W|\geq \frac{99}{100}$;
see Paragraph A. In fact, we need more than just a minimum fraction: uniformity in the  
distribution of small gaps in $\fS_n$ is also needed. Following a coupling,
they determine how far from being proper the uncoupled part of the measure is; 
see Paragraph C. As $\fS_n$, the magnet used for coupling at time $n$,
 is constructed using the local stable manifolds of 
$F_n, F_{n+1}, \dots$, the results above must hold uniformly for all relevant sequences.

\medskip \noindent
(3) {\it Uniform growth of unstable curves.} This very important fact, which takes into
consideration both the expansion due to hyperbolicity of the map and the cutting by
discontinuities and homogeneity lines, is used in more ways than one: It is used
to ensure that regularity of densities is restored and most of the uncoupled measure becomes proper at the end of the ``recovery periods". The uniform 
$r$ and $\lambda$ asserted in Paragraph C are obtained largely from
the uniform structure of magnets, i.e., item (2) above, together with the growth results
in Section~\ref{sec:preliminariesII} (as well as inductive control from previous steps). 
Growth results are also used to
produce a large enough fraction of sufficiently long unstable curves at 
times $t_k+r$. That together with the uniform mixing in item~(1) permits us 
to guarantee the coupling of a fixed fraction $\zeta$ at time $t_{k+1}$. 

\medskip
Item (1) above is purely perturbative as we have discussed; item (2) is partially perturbative: 
proximity to $\widetilde \cK_q$ is used to ensure that $\fS_n$ has some of 
the desired properties. Item (3) is strictly nonperturbative: we do not derive the properties in question from
the proximity of the composed sequence $F_n \circ \cdots \circ F_2 \circ F_1$ 
to $\widetilde F_q^n$. Instead, we show that these properties hold true, with
uniform bounds, for all sequences $(\cK_n)$ with $\cK_n \in \IK$.
In the case of genuinely moving scatterers, the constants $r,C$ and~$\zeta$ all depend on the relevant reference configuration~$\widetilde \cK_q$, through the curve~$\widetilde W$
in whose neighborhood the couplings occur. A priori, the same is true of $\lambda$, 
although, as we will show, $\lambda$ is in fact independent of $\widetilde\cK_q$.


\section{Main ingredients in the proof of Theorem~\ref{thm:weak_conv}}\label{sec:preliminaries_continued}

We continue to develop the main ideas needed in the proof
of Theorem~\ref{thm:weak_conv}, focusing first on the countable case and
addressing issues that have been raised in the synopsis in the last
section. As in Sections~\ref{sec:preliminariesI} and~\ref{sec:preliminariesII}, all configuration pairs whose billiard maps are
discussed are assumed to be admissible and in $\IK$. Further 
conditions on $(\cK_n)$, such as close proximity to a reference configuration,  
will be stated explicitly. Many of the results below are parallel to known results
for classical billiards; see e.g. \cite{ChernovMarkarian_2006}.

\subsection{Local stable manifolds}\label{sec:LSM}
Given $(\cK_n)_{n = 0}^\infty$, we let $W^s(x)$ denote the maximal (possibly empty, homogeneous) local stable manifold passing through the point $x\in\cM$ for the sequence 
of maps $(F_n)_{n\geq 1}$. Recall that $W^s(x)$ has positive length if and only if the trajectory $\cF_n x$ does not approach the ``bad set'' $\partial\cM\cup\cup_{|k|\geq k_0}\partial\IH_k$ too fast as a function of $n$. 
Based on this fact, the size of local stable manifolds may be quantified as follows: Let $r^s(x)$ denote the distance of $x$ from the nearest endpoint of $W^s(x)$ as measured along $W^s(x)$. A standard computation, which we omit, shows that for an arbitrary unstable curve $W$ through $x$,
\beq\label{eq:stable_length}
r^s \geq \widetilde C^{-1} \inf_{n\geq 0} \Lambda^n\, r_{W,n} \equiv u_{W}^s,
\eeq
where $\widetilde C>0$ is a uniform constant and $r_{W,n}$ was introduced in the beginning of Section~\ref{sec:growth}. 

In Paragraph D, item (2), of the Synopsis, we identified the need for certain
uniform properties of local stable manifolds, such as the density of stable
manifolds of uniform length on unstable curves. The next lemma provides
a basic result in this direction.

\begin{lem}\label{lem:fundamental} Given $a>0$ and $A>0$, there exist 
$s' \in \mathbb Z^+$ and $L>0$ such that for any $(\cK_n)_{n=0}^\infty$,
every unstable curve $\widetilde W$ has the property 
\beq\label{eq:fundamental}
\fm_{\widetilde W}\bigl\{u_{W}^s\geq A\bigl |\widetilde W\bigr |\bigr\}\geq (1-a)
|\widetilde W|\ 
\eeq
provided (i) $\widetilde W$ is located in the middle third of a homogeneous 
unstable curve $W$ for which $\cF_{s'}W$ has a single homogeneous component,
and (ii) $|\widetilde W| \le L |W|/3$.
\end{lem}

\begin{proof}
Since $\cF_{s'} W$ consists of a single homogeneous component, we have 
\begin{align*}
& r_{W,n}(x) \geq \hat c\Lambda^n r_{W,0}(x) \qquad \forall\, x\in W,\,0\leq n\leq s',
\\
& r_{W,n}(x) \geq r_{\widetilde W,n}(x) \qquad  \forall\, x\in\widetilde W,\, 0\leq n\leq s',
\\
& r_{W,0}(x) \geq  |W|/3 + r_{\widetilde W,0}(x) \qquad \forall\,x\in\widetilde W.
\end{align*}
Using these facts and the Growth Lemma, we can estimate that
\beqn
\begin{split}
& \fm_{\widetilde W}\{u^s_W<A|\widetilde W|\} \leq \sum_{n\geq 0} \fm_{\widetilde W}\{r_{W,n}<\widetilde CA|\widetilde W| \Lambda^{-n}\} 
\\
& \qquad\qquad\quad \leq \sum_{n\leq s'} \fm_{\widetilde W}\{r_{W,0}<\hat c^{-1}\widetilde CA|\widetilde W| \Lambda^{-2n}\} + \sum_{n> s'} \fm_{\widetilde W}\{r_{\widetilde W,n}<\widetilde CA|\widetilde W| \Lambda^{-n}\}
\\
& \qquad\qquad\quad \leq \sum_{n\leq s'} \fm_{\widetilde W}\{r_{\widetilde W,0}<\hat c^{-1}\widetilde CA|\widetilde W| \Lambda^{-2n}- |W|/3\} + \sum_{n> s'} C_\mathrm{gr}(\vartheta^n+|\widetilde W|)\widetilde CA|\widetilde W| \Lambda^{-n}\ .
\end{split}
\eeqn
In the last line, the first sum vanishes if we take $L \le \hat c\widetilde C^{-1}A^{-1}$ and the Growth Lemma yields the bound on the second sum. The second sum is then $< a |\widetilde W|$ if~$s'$ is so large that
 $\frac{C_\mathrm{gr}(1+LL_0/3)\widetilde CA}{(\Lambda-1)\Lambda^{s'}}\leq a$,
 where $L_0$ here is the maximum length of a homogeneous unstable curve. 
\end{proof}


\subsection{Magnets}\label{sec:magnets}

We now define more precisely the objects $\fS_n$ in Paragraph A of the Synopsis.
Recall that $\fS_n$ is constructed using stable manifolds $W^s_n(x)$ with respect to
 the sequence of maps $F_n, F_{n+1}, F_{n+2}, \dots$. 
 When what happens before time $n$ is irrelevant
to the topic under discussion, it is simpler notationally to set $n=0$ (by shifting and
renaming indices in the original sequence). That is what we will do here as well as
in the next few subsections.

We fix a reference configuration $\widetilde \cK \in \IK$,  and denote
$\widetilde F = F_{\widetilde \cK, \widetilde \cK}$.
Let  $s'$ and $L$ be given by Lemma~\ref{lem:fundamental} with
$A=\frac12$ and $a=0.01$. We pick a piece of unstable manifold $\widetilde W_+^u$ 
of $\widetilde F$ (more than just an unstable curve) with
the property that $\widetilde W_+^u$ is homogeneous and 
$\widetilde F^{s'}\widetilde W_+^u$ has a single homogeneous component.
Let  $\widetilde W^u \subset \widetilde W_+^u$ be the subsegment of 
$\widetilde W_+^u$ half as long and located at the center.
Then there exists $\ve'>0$ such that $\cF_{s'}\widetilde W^u$ has a single 
homogeneous component for any $\cK_n \in \cN_{\ve'}(\widetilde \cK), \ 1 \le n \le s'$.
Let $\widetilde W \subset \widetilde W^u$ be located at the center of 
$\widetilde W^u$ with $|\widetilde W| = L |\widetilde W^u|/3$. 
Lemma~\ref{lem:fundamental} then tells us that for $(\cK_n)$ as above and
$\widetilde W_{0}:=\{x\in\widetilde W\,:\,u_{\widetilde W^u}^s(x)\geq  |\widetilde W|/2 \}$,
we are guaranteed that $\fm_{\widetilde W}(\widetilde W_{0})/|\widetilde W|  \geq \frac{99}{100}$.
The set
$
\fS_0=\cup_{x\in\widetilde W_0} W^s(x) 
$
is the magnet defined by $\widetilde W$ and the sequence $(\cK_n)$.

Additional upper bounds will be imposed on $|\widetilde W|$ to obtain the magnet 
used in the proof of Theorem~\ref{thm:weak_conv}. The size of the neighborhood
$\cN_{\ve'}(\widetilde \cK)$ will also be shrunk a finite number of times as we go along.

Now let $W$ be any unstable curve that crosses $\fS_0$ completely, in the sense that it meets $W^s(x)$ for each $x\in\widetilde W_{0}$ with excess pieces on both sides, and let $\bfh$ 
denote the holonomy map from $\widetilde W \cap\fS_0\to W\cap\fS_0$. 

\begin{defn}\label{defn:proper_crossing}
We say the crossing is \emph{proper} if for a uniform constant $\aleph>0$ to be determined, the following
hold: (i) $W$ is regular, (ii) the distance between any $x\in\widetilde W_{0}$ and $\bfh(x)$ as measured along~$W^s(x)$ is less than 
$\aleph |\widetilde W|$, and (iii) each of the two excess pieces ``outside'' the magnet is
more than $|\widetilde W|$ units long. 
\end{defn}

We need $\aleph$ to be small enough that ~\eqref{eq:hol_distort}, i.e.,
$|(\cJ \bfh\circ \bfh^{-1})^{-1}-1|\leq \tfrac 1{10}$, is guaranteed in
 proper crossings. To guarantee ~\eqref{eq:hol_distort},
we need, by \eqref{eq:Jac_bound}, both (i) the distance between $x$ and $\bfh(x)$ and 
(ii) the difference between the slopes of $\widetilde W_{0}$ and $\widetilde W$ 
at $x$ and $\bfh(x)$ respectively, to be small. 
(i) is bounded by $\aleph |\widetilde W|$. Observe that (ii) is also (indirectly) controlled
by $\aleph$: since both $W$ and $\widetilde W$ are regular curves (real unstable 
manifolds of $\widetilde F$ are automatically regular), there is a fixed upper bound
on their curvatures. Thus the shorter the curves, the closer they are to straight lines.
Now since $W$ meets the stable manifolds at the two ends of $\widetilde W$
at distances $< \aleph |\widetilde W|$ from $\widetilde W$, taking $\aleph$ small
forces the slopes of $W$ and $\widetilde W$ to be close. 
Further upper bounds on $\aleph$ may be imposed later.

\medskip \noindent
\emph{Note on terminology.}
In the discussion to follow, the setting above is assumed, and 
a number of constants referred to as ``uniform constants" will be introduced. 
This refers to constants that are independent of 
$\widetilde \cK$ for as long as $\widetilde \cK\in\IK$, and they are independent
of $(\cK_n)$, $\widetilde W^u$, $\widetilde W$ or $\widetilde W_0$
provided these objects are chosen according to the recipe above.

\subsection{Gap control}\label{sec:gaps} We discuss here the distribution of 
gap sizes of the magnet, issues about which were raised in the Synopsis. 
The setting, including $\fS_0$, is as in Section~\ref{sec:magnets}.

Recall that a point $x\in \widetilde W$ belongs to the Cantor set $\widetilde W_0$ if and only if $r_{{\widetilde W^u},k}(x)\geq \widetilde C \Lambda^{-k}|\widetilde W|/2$ for every $k\geq 0$. We define the {\it rank} of a gap $\widetilde V$ in $\widetilde W \setminus \widetilde W_0$ to be the smallest $R$ such that $r_{{\widetilde W^u},R}(x)<\widetilde C \Lambda^{-R}|\widetilde W|/2$ holds for some $x\in\widetilde V$. Observe that $R$ so defined is also the smallest number
for which $\cF_R \widetilde V$ meets the ``bad set'' $\partial\cM\cup\cup_{|k|\geq k_0}\partial\IH_k$: Clearly, $\cF_k \widetilde V$ could not have met the bad set for
$k<R$. On the other hand, $\cF_R \widetilde V$ must meet the bad set, or the 
minimum of $r_{{\widetilde W^u},R}$ on $\widetilde V$ would occur on one of its 
end points, which cannot happen for a gap (excess pieces are not gaps). 
Notice that this implies
that $\cF_{R-1} \widetilde V$ must cross (transversally) $F_R^{-1}\partial\IH_k$
for some $k$, and if it crosses $F_R^{-1}\partial M$, then it automatically crosses 
$F_R^{-1}\partial\IH_k$ for infinitely many $k$. 

Consider next an unstable curve $W$ that crosses $\fS_0$ properly, and let
$W_0 = W\cap\fS_0$. 
Each gap $V\subset W\setminus W_{0}$ corresponds canonically to a unique gap $\widetilde V\subset \widetilde W\setminus \widetilde W_{0}$, as their corresponding end points are connected by local stable manifolds $\gamma^s_1$ and $\gamma^s_2$
in $\fS_0$. We define the {\it rank} of
$V$ to be that of $\widetilde V$, and claim that 
the rank of $V$ is also the first time $\cF_R V$ meets the bad set. 
To see this, consider the $\cF_n$-images of the region bounded by $V, \widetilde V,
\gamma^s_1$ and $\gamma^s_2$. Since $\cF_n(\gamma^s_i)$
avoid the bad set (which consists of horizontal lines), 
it follows that for each $n$, $\cF_n(V)$ crosses the bad set 
if and only if $\cF_n(\widetilde V)$ does.

Let $W$ be as above. For $b>0$, we consider the {\it dynamically defined} Cantor set
\begin{equation} \label{cantor}
W^b_0 = \{x \in W' \,: \,r_{W,k} \ge b \Lambda^{-k} |\widetilde W| \,\,\, \forall \,k \ge 0\}.
\end{equation}
For $W = \widetilde W^u$ and $b=\frac12 \widetilde C$, $W^b_0 = \widetilde W_0$. 
We observe that, for $b\leq \hat c$, with $\hat c$ as in~\eqref{eq:expansion}, the definition of the set $W_0^b$ does not depend on the part $W\setminus W'$ ``outside'' the magnet. This is  because of the length of each of the two components of $W\setminus W'$ and the expansion in~\eqref{eq:expansion}. Like $\widetilde W_0$, $W^b_0$ is a Cantor set, and the ranks of the gaps of $W\setminus W_0^b$ have the same characterizations as the gaps of 
$\widetilde W^u\setminus \widetilde W_0$. 

The proofs below are a little sketchy, as there are no new issues in the time-dependent case.

\begin{lem}\label{lem:dyn_Cantor}
There exists a uniform constant $\bar b\leq \hat c$ such that the following hold for
$\aleph$ small enough: Let $W$ be an arbitrary unstable curve  crossing $\fS_0$ properly, and let $W_0 = W\cap\fS_0$. Then

(i) $W_0\subset W^{\bar b}_0$, and

(ii) through every point of $W^{\bar b}_0$ there is a local stable manifold which meets $\widetilde W$.
\end{lem}

\begin{proof} (i) Let $\tilde x \in \widetilde W_0$, and denote by $x \in W_0$ the
intersection of $W^s(\widetilde x)$ and $W$. The assertion follows since for some $a\in(0,1)$ depending only on the cones, $d_\cM(\cF_n\tilde x,\partial\cM\cup\cup_{|k|\geq k_0}\partial\IH_k)\geq a  \widetilde C\Lambda^{-n}|\widetilde W|/2$ for all $n\geq 0$, while $d_\cM(\cF_n x,\cF_n \tilde x)\leq \hat c^{-1}\Lambda^{-n}\aleph|\widetilde W|$ for all 
$n\geq 0$. In particular, $\bar b= \min(\hat c,a\widetilde C/4)$ suffices. (ii) At each point $x\in W_0^{\bar b}$ the local stable manifold $W^s(x)$ extends at least $\bar b \widetilde C^{-1}|\widetilde W|$ units on both sides of $W$, proving (ii) for $\aleph$ sufficiently small.
\end{proof}

We record next a tail bound for gaps of  dynamically defined Cantor sets.

\begin{lem}\label{lem:Cantor}
There exists a uniform constant $C'_\mathrm{g}>0$ such that if $W$ crosses~$\fS_0$ properly, then for any $b>0$ and $R\geq 1$, we have \footnote{We use the convention that the infimum equals $\infty$ if it does not exist in $\IN$.}
\beq\label{eq:Cantor}
\fm_{W}\Bigl\{x\in W'\,:\,\inf\bigl\{k\in\IN: r_{W,k}(x) <b\Lambda^{-k}|\widetilde W|\bigr\}\in[R,\infty)\Bigr\} 
\leq b\,C_\mathrm{g}' \Lambda^{-R}|\widetilde W| \ .
\eeq
\end{lem}

\begin{proof} 
The Growth Lemma yields the following upper bound on the left side of~\eqref{eq:Cantor}:
\beqn
\begin{split}
\sum_{k\geq R}\fm_{W}\{r_{{W},k}<b\Lambda^{-k}|\widetilde W|\} 
\leq 
b\, C_\mathrm{gr}  |\widetilde{W}| \sum_{k\geq R}(\vartheta^k  + |W|) \Lambda^{-k} 
\leq  b\,C_\mathrm{gr}\widetilde C|\widetilde W| \frac{(1  + L_0)}{1-\Lambda^{-1}}\Lambda^{-R} ,
\end{split}
\eeqn
where $L_0$ is the maximum length of a (connected) unstable curve.
\end{proof}

The following is the result we need. 

\begin{lem}\label{lem:gap_tail}
There exist uniform constants~$C_\mathrm{g}>0$ and~$c_\mathrm{g}>0$ for which
the following hold: Let~$\fS_0$ be as above, and let~$W$ be an arbitrary 
unstable curve which crosses~$\fS_0$ properly. Then

\smallskip
\noindent (a)  for every $R \ge 0$, 
\beq\label{eq:rank_tail} 
\fm_{W}\{x\in W\,:\, \text{$x$ is in a gap of rank $\geq R$}\} 
\leq C_\mathrm{g} \Lambda^{-R}| W| \ ;
\eeq

\smallskip
\noindent (b) for any gap $V\subset W\setminus W_{0}$ of any rank $R\geq 0$, $\cF_{R-1}V$ is a homogeneous unstable curve and
\beq\label{eq:rank_growth}
|\cF_{R} V| \geq  c_\mathrm{g}\Lambda^{-R}|\widetilde W|.
\eeq
\end{lem}

\begin{proof}[Proof of Lemma~\ref{lem:gap_tail}]  
(a) For $W=\widetilde W$, the result follows immediately from 
Lemma~\ref{lem:Cantor}. 
 For general~$W$, a separate argument is needed as $W \cap \fS_0$ is not exactly 
 of the form in (\ref{cantor}). Let $\bar b$ be such that $W_0 \subset W^{\bar b}_0$
(Lemma \ref{lem:dyn_Cantor}(i)), and let $V$ be a gap of $W \setminus W_0$ of 
rank~$R$. Then $V$ is the union $V \cap W^{\bar b}_0$ and a collection of gaps of 
$W \setminus W_0^{\bar b}$. We observe that these gaps have ranks $\ge R$,
because of the characterization of rank (for both kinds of gaps)
as the first time their images meet
the bad set.
As for the measure of $V \cap W^{\bar b}_0$, by Lemma \ref{lem:dyn_Cantor}(ii)
and the properties 
of the Jacobian of the holonomy map ${\bf h} : \widetilde W \cap (\cup_{x \in W^{\bar b}_0} W^s(x))
 \to W$, we have $\fm_W(V \cap W^{\bar b}_0) \le \frac{11}{10} \fm_{\widetilde W}({\bf h}^{-1}(V \cap W^{\bar b}_0)) \le \frac{11}{10}
\fm_{\widetilde W}(\widetilde V)$. 

Summing over gaps $V$ of rank $\ge R$ in 
$W \setminus W_0$ and applying Lemma~\ref{lem:Cantor}
to the gaps of $W^{\bar b}_0$ and~$\widetilde W_0$ (recalling $|\widetilde W|<|W|$), we obtain 
\beqn
\fm_{W}\bigl(\text{union of all gaps $V\subset W\setminus W_0$ of rank $\geq R$}\bigr)
\leq 
(\bar b+\tfrac{11}{10})\, C'_\mathrm{g}\Lambda^{-R}|W|
\eeqn
and (a) is proved by choosing $C_\mathrm{g}$ large enough. 

\smallskip 
To prove (b), first make the argument for gaps $\widetilde V$ of $\widetilde W_0$ (which is straightfoward), and then leverage the connection between $V$ and $\widetilde V$ 
via the stable manifolds connecting their end points.
\end{proof}


\subsection{Recovery of densities}\label{sec:recovery}
As explained in the Synopsis, the uncoupled part of the measure has to `recover'
and become proper again before it is eligible for the next coupling. Postponing the full
picture to later, we focus here on the situation of the last two subsections, i.e., 
a reference configuration $\widetilde \cK$, a sequence $(\cK_n)$ with
$\cK_n \in \cN_\ve(\widetilde \cK)$ for $0 \le n \le s'$, and a magnet $\fS_0$. 
We \emph{assume} that a coupling takes place at time $0$, and 
consider the recovery process thereafter.

\subsubsection{Single measured unstable curve}

We treat first the case of a single measured unstable curve $(W,\nu)$ making 
a proper crossing of the magnet, as described in Paragraph B of the Synopsis with $n=0$. We denote by $\rho$ and $\bar\rho$ the densities of $\nu$ and of the coupled
part of $\nu$ respectively. We also let $W'$ be the shortest subsegment 
of $W$ containing $W \cap \fS_0$, so that $W \setminus W'$ consists 
of the two excess pieces. 
As in Paragraph B, we extend $\bar \rho$ to a regular density called $\check \rho$ on~$W'$, 
and decompose the uncoupled part of~$\rho$ into densities 
of the following types: 

(a) $\rho|_{W \setminus W'}$;

(b) $(\rho-\check\rho)|_{W'}$ (this will be referred to as the ``top density"), and 

(c) $\check\rho|_V$ as $V$ ranges over all the gaps of $W$.

\noindent We consider separately each of these densities (counting $\check\rho|_V$
for different $V$ as different measures), 
and discuss their recovery times, meaning the time it takes for such a measure 
to become proper (see Definition~\ref{defn:proper}). 

\medskip
\noindent (a) Since these are regular to begin with,
the only reason why they may not be proper is that the excess pieces 
may be too short. Thus their recovery times may depend on $|\widetilde W|$
(each excess curve has length $\ge |\widetilde W|$),
but are otherwise uniformly bounded and independent of $(\cK_n)$. 


\medskip
\noindent (c) As discussed in the Synopsis, 
the density $\check\rho|_V$ for each $V$ is regular to begin with. Thus recovery time 
 has only to do with length. For a gap of rank $R$, 
recall that the image $\cF_{R-1}V$ is a regular homogeneous unstable curve.
Denoting $\cZ=\check \nu(V)/|V|$ where $\check \nu$ is the measure on $V$ with density~$\check\rho$, we have 
$\cZ_{R-1}/\check \nu(V)=1/|\cF_{R-1}V|\leq C_\mathrm{e}^2/|\cF_R V|^2 \leq C_\mathrm{e}^2 c_\mathrm{g}^{-2}|\widetilde W|^{-2}\Lambda^{2R}$ by Lemma~\ref{lem:growth_bound} and~\eqref{eq:rank_growth}. 
In the next step, the curve $V_{R}=\cF_{R-1}V$ will get cut, but we may proceed with the aid of Lemma~\ref{lem:Z} and obtain $\cZ_n/\nu(V)\leq C_\mathrm{p}$ for
$
n \geq R-1+(2R\log\Lambda+|\log C_\mathrm{e}^2 c_\mathrm{g}^{-2}|+2|\log|\widetilde W||)/|\log\vartheta_\mathrm{p}|.
$
In other words, we have proved
\begin{lem}\label{lem:gap_proper}
There exists a uniform constant $c_\mathrm{p}>0$ such that the following holds.
In the setting of Lemma~\ref{lem:gap_tail}, after $c_\mathrm{p}(R+|\log|\widetilde W||)$ steps, each one of the  measures on the gaps of rank~$R$ will have become a \emph{proper} measure. 
\end{lem} 

\medskip
\noindent (b)
We begin by stating a general result. See the Appendix for a proof.
\begin{lem}\label{lem:top_density}
There exists a uniform constant $C_\mathrm{top}>0$ such that the following holds. Given a sequence $(\cK_n)_{n\geq 0}$, suppose that two regular densities $\varphi,\check\varphi$ on the same unstable curve~$W$ satisfy $b\leq \varphi/\check\varphi\leq B$ for some $B>b>1$
everywhere on $W$, and let $\psi=\varphi-\check\varphi$. Then
\beq\label{eq:dyn_Holder}
|\log \psi(x)-\log\psi(y)| \leq C_\mathrm{top}\frac{B+1}{b-1}\,\theta^{s(x,y)}
\eeq
for all $x,y \in W$.
\end{lem}

By Lemma~\ref{lem:rhobar_bound}, these assumptions are satisfied for $\varphi=\rho$
and $\check \varphi = \check\rho$ 
with $b=\frac85$ and $B=\frac54\zeta^{-1}e^{C_\mathrm{r}}$ where $\zeta$ is the fraction
of the measure coupled (see Section~\ref{sec:outline}). 
Even after the densities become regular, it may take additional time for the measure
to become proper. The next lemma, proved in the Appendix, is suited for such situations.

\begin{lem}\label{lem:top_aid}
There exists a uniform constant $\bar C_\mathrm{p}>0$ such that the following holds. Given an admissible sequence $(\cK_n)_{n\geq 0}$, suppose $\nu$ is a measure on a regular unstable curve~$W$ whose density~$\psi$ satisfies $|\log\psi(x)-\log\psi(y)|\leq C\theta^{s(x,y)}$ for some $C>C_\mathrm{r}$. Then the push-forward $(\cF_n)_*\nu$ will be a proper measure, if $n\geq \bar C_\mathrm{p}(|\log|W||+C)$.
\end{lem}

From these lemmas, we conclude that
\begin{lem}\label{lem:top_proper}
The maximum time it takes for each of the ``top'' measures to become a proper measure is $\bar c_\mathrm{p}+\bar C_\mathrm{p}(|\log|\widetilde W||+|\log \zeta|)$, where $\bar c_\mathrm{p}>0$ is another uniform constant.
\end{lem}


\subsubsection{More general initial measures}
Let the initial measure $\mu=\sum_\alpha\nu_\alpha$ be a proper probability measure consisting of countably many regular measured unstable curves, and assume that a fraction of $\mu$ crosses the magnet $\fS_0$ at time $0$ with \eqref{eq:cross_bound} holding 
for some~$\zeta$. As explained in Section~\ref{sec:outline}, precisely  $\zeta$ units of its mass will be coupled to the reference measure~$\widetilde\fm_0$. The remaining measure,
which we denote by $\mu_0$, has mass $1-\zeta$ and consists of the three kinds of measures described earlier. The next result summarizes some of the results above and is very convenient for bookkeeping.

\begin{lem}\label{lem:proper_part}
There exist constants $C\geq 1$, $\lambda\in(0,1)$, and $r>0$ such that, for any $m\geq r$, $(\cF_m)_*\mu_0$ can be split into the sum of two nonnegative measures $\mu_{m}^\mathrm{P}$ and $\mu_{m}^\mathrm{G}$, both consisting of countably many regular measured unstable curves, with the properties that
 $\mu_{m}^\mathrm{P}$ is \emph{proper}, and $\mu_{m}^\mathrm{G}(\cM)=C\lambda^m$.
\end{lem}

The constants $C$ and $r$ in the lemma depend on the reference configuration used in the construction of the magnet (through $|\widetilde W|$) and $r$ also depends on $\zeta$, but neither of them depend on the initial measure $\mu$ or the sequence of configurations. The constant $\lambda$ is uniform.

We will use the notation in Section~\ref{sec:outline} in the proof -- except for omitting the subscript $n$.

\begin{proof}[Proof of Lemma~\ref{lem:proper_part}]
By Lemma~\ref{lem:top_proper}, we can choose $r$ large enough that the $m$-step push-forwards of the densities on the excess pieces and the 
``top'' densities yield proper measures for all $m\geq r$. 
From that point on the question is about the gaps. 
Here we turn to Lemmas~\ref{lem:gap_tail} and~\ref{lem:gap_proper}.  By regularity of the measures~$\nu_{\alpha}$, \eqref{eq:rank_tail} implies
 $\nu_{\alpha}\{x\in W_{\alpha}\,:\, \text{$x$ is in a gap of rank $\geq R$}\} 
\leq C_\mathrm{g}'' \Lambda^{-R}\nu_{\alpha}(W_{\alpha})$ for another uniform constant~$C_\mathrm{g}''>0$. 
Summing over $\alpha$ yields the estimate $\mu_0\{\text{gaps of rank $\geq R$}\} 
\leq C_\mathrm{g}'' \Lambda^{-R}$. With the aid of Lemma~\ref{lem:gap_proper}, we see that the quantity $\mu_0\{\text{gaps needing $\geq m$ steps to yield a proper measure}\}$ is bounded above by the expression
\linebreak
$C_\mathrm{g}'' \Lambda^{-(m/c_\mathrm{p}-|\log|\widetilde W||)+1}\leq  C_\mathrm{g}'' \Lambda^{1+|\log|\widetilde W||}\lambda^m\leq C\lambda^m$ with $C=\max(1,C_\mathrm{g}'' \Lambda^{1+|\log|\widetilde W||})$ and $\lambda=\Lambda^{-1/c_\mathrm{p}}\in(0,1)$. Taking $r$ large, we may assume $C\lambda^r<1-\zeta=\mu_0(\cM)$. We first collect all the gap measures from $(\cF_{r})_*\mu_0$ into~$\mu_{r}^\mathrm{G}$; they have total mass $\leq C\lambda^r$. Next, we take a suitable constant multiple of the remaining, proper, measure from $(\cF_{r})_*\mu_0$ and include it into~$\mu_{r}^\mathrm{G}$ so that finally $\mu_{r}^\mathrm{G}(\cM)=C\lambda^r$ holds exactly. (This is mostly for purposes of keeping the statements clean.) Note that $\mu_{r}^\mathrm{P} = (\cF_{r})_*\mu_0-\mu_{r}^\mathrm{G}$ is proper. By our earlier results, the push-forwards of~$\mu_{r}^\mathrm{P}$ remain proper; these will always be included in the measures $\mu_{m}^\mathrm{P}$ for $m>r$. The real gap measures included in $\mu_{r}^\mathrm{G}$, on the one hand, continue to recover into being proper measures at least at the rate $\lambda$. On the other hand, the proper measures included in $\mu_{r}^\mathrm{P}$ will continue to be proper under push-forwards. Hence, for $m>r$, we return to $\mu_{m}^\mathrm{P}$ a suitable constant multiple of the proper part of $(\cF_{m,r+1})_*\mu_{r}^\mathrm{G}$, as necessary, so that the statements of the lemma continue to hold. 
\end{proof}



\subsection{Uniform mixing}\label{sec:mixing}

We discuss here the primary reason behind the asserted exponential memory loss
for the sequence $(\cK_m)$. 
Since events that occur prior to
couplings are involved, we cannot assume that the coupling of interest occurs at $n=0$, 
as was done in Sections~\ref{sec:magnets}--\ref{sec:recovery}.
Our goal is to address item (1) in
Paragraph D of the Synopsis.

Recall from Section~\ref{sec:magnets} that given a configuration $\widetilde\cK\in\IK$,
there exist unstable manifolds $\widetilde W\subset \widetilde W^u$ of 
$\widetilde F = F_{\widetilde \cK,\widetilde \cK}$, $s' \in \mathbb Z^+$ and
$\ve'>0$ such that for any sequence $(\cK_m)_{m\geq 0}$ with
$\cK_{n},\dots,\cK_{n+s'}\in \cN_{\ve'}(\widetilde\cK)$, a magnet $\fS_n$ with desirable
properties (see Sections~\ref{sec:magnets} and \ref{sec:gaps}) can be constructed out of 
$\widetilde W$ and stable manifolds for $(F_{n+m})_{m \ge 1}$. We assume for each
$\widetilde \cK$ that $s', \ve'$ and $\widetilde W \subset \widetilde W^u$ are fixed.

\begin{prop}\label{prop:magnets}
Given $\widetilde\cK\in\IK$, there exist $\zeta>0$, $\ve \in (0, \ve')$ 
and $s \in \mathbb Z^+$ such
that the following holds for every $(\cK_m)_{m\geq 0}$ with 
$\cK_0,\dots,\cK_{s+s'}\in \cN_{\ve}(\widetilde\cK)$: Let $\fS_s$ be the magnet
defined by $\widetilde W$ and $(F_{s+m})_{m \ge 1}$.  Then every regular measured unstable curve $(W,\nu)$ with $|W|\geq (2C_\mathrm{p})^{-1}$ has the property
that if $W_{s,i}=\cF_{s}(W_{s,i}^-)$ are the homogeneous components of 
$\cF_{s}(W)$ which cross~$\fS_s$ properly, then
\beq\label{eq:cross_fraction}
\sum_i(\cF_{s})_*(\nu|_{W_{s,i}^-}) (W_{s,i}\cap\fS_s) \geq 4\zeta e^{C_\mathrm{r}}\,
\nu(W).
\eeq
\end{prop}

This proposition asserts that starting 
from \emph{an arbitrary} regular measured unstable curve $(W,\nu)$ with $|W|\geq (2C_\mathrm{p})^{-1}$, at least a \emph{uniform} fraction of its \mbox{$\cF_s$-image} has sufficiently many (homogeneous) components crossing the magnet $\fS_s$ \emph{properly} provided $\cK_m$ remains in a sufficiently small neighborhood $\cN_\ve(\widetilde\cK)$ of the reference configuration 
$\widetilde\cK$ for time $\le s+s'$. The last $s'$ steps is used to make sure that the magnet has a high density of sufficiently long local stable manifolds (see Sections~\ref{sec:LSM} and~\ref{sec:magnets}), whereas the first $s$ is directly related to the mixing property
of $\widetilde F$.

Note the uniformity in Proposition~\ref{prop:magnets}: the constants depend 
on $\widetilde\cK$,  with $s, \zeta$ and $\ve$ depending also on the choice of $\widetilde W$;
but $(\cK_m)_{m\geq 0}$ is arbitrary as long as it satisfies the conditions above.

\begin{proof}[Proof of Proposition~\ref{prop:magnets}]
We begin by recalling the following known result on the mixing property of $\widetilde F^n$
(see e.g. \cite{ChernovMarkarian_2006}):
Let $\widetilde \fS$ be the magnet defined by $\widetilde W$ and powers of 
$\widetilde F$, and let $\sigma>0$ be given. Then there exist $s>0$ (large) and $\zeta'>0$ (small) 
such that for any regular measured unstable curve $(W,\nu)$ with $|W|\geq (2C_\mathrm{p})^{-1}$ and any $n\geq s$, 

\smallskip
(i) finitely many  components of $\widetilde F^n W$ cross $\widetilde \fS$ \emph{super}-properly and

(ii) denoting these components $W_{n,i}^\text{super}$, 
\beq\label{eq:super_cross_fraction}
\sum_i (\widetilde F^n)_*\nu\,(W_{n,i}^\text{super}\cap\fS^{\widetilde\cK}) \geq \zeta'\nu(W)\ .
\eeq
Here super-proper crossing means that the crossing is proper with room to spare.  
Specifically, the excess pieces are twice as long (i.e., at least $2|\widetilde W|$ units),
 and $|\varphi_{W_{n,i}^\text{super}} - \varphi_{\widetilde W} | < \sigma$, where
  $\widetilde W$ extended by $\frac{1}{10} |\widetilde W|$ along $\widetilde W^u$
  on each side is the graph of $\varphi_{\widetilde W}$ as a function of $r$, 
  and $W_{n,i}^\text{super}$ suitably restricted is the graph
of $\varphi_{W_{n,i}^\text{super}}$ defined on the same $r$-interval.

\medskip
Let $\sigma$ be such that for any unstable curve $U$,
$|\varphi_U - \varphi_{\widetilde W} | < 2 \sigma$ implies that condition (ii) in the definition
of proper crossing (Definition \ref{defn:proper_crossing})   is satisfied by $U$ independently
of the maps used to define the stable manifolds in the magnet (provided all configurations are
in $ \IK$). We have used here the fact that there
are uniform stable and unstable cones and that they are bounded away
from each other. Let $s$ be given by the above result for $\widetilde F$. Our next step is to view $\cF_s$ as a perturbation of $\widetilde F^s$, and to argue that the following holds for $\ve$ sufficiently small: 
Let $(W, \nu)$ be as in
the proposition, and suppose $W_{s,i}^\text{super}\subset \widetilde F^s W$ 
crosses $\widetilde \fS$ super-properly.  Then for every $(\cK_m)$ with
$\cK_m \in \cN_{\ve}(\widetilde \cK)$ for $m \le s+s'$, 
there exists a subcurve $V\subset \widetilde V =\widetilde F^{-s}W_{s,i}^\text{super}\subset W$ 
such that $\cF_s V$ 
crosses $\fS_s$ properly.

First observe the following facts about $\widetilde F^m(\widetilde V)$: 
(a) There exists $\bar k \in \IN$ independent of $W$ such that 
$\widetilde F^m \widetilde V \subset\cup_{|k|\leq \bar k}\IH_k\setminus \cS_{\widetilde\cK,\widetilde\cK}$ for $0\leq m < s$. This is because for each $m$,
$\widetilde F^m \widetilde V$ is contained in a homogeneity strip $\IH_k$,
and would be arbitrarily short if $k$ was arbitrarily large, and that is not possible
by Lemma~\ref{lem:growth_bound} since $|\widetilde F^s(\widetilde V)| 
\gtrsim 5|\widetilde W|$. (b) Let $V \subset \widetilde V$ be the subsegment
with the property that $\widetilde F^s(V)$ crosses $\widetilde \fS$ and the excess pieces 
have length $\frac53 |\widetilde W|$. By uniform distortion bounds (Lemma \ref{lem:distortion})
and (a) above, there exists $\delta>0$ independent of $W$ or $\widetilde V$
such that for $1 \le m < s$, $\widetilde F^m(V)$ 
has distance $>\delta$ from
$\cup_{|k|\leq \bar k}\partial \IH_k \cup \cS_{\widetilde\cK,\widetilde\cK}$.

We wish to choose $\ve$ small enough that (i) for each $m=1,\dots, s$,
$\cF_m(V)$ and $\widetilde F^m(V)$ differ in Hausdorff distance by $<\frac12 \delta$,
and (ii) each of the excess pieces of $\cF_s(V)$ are $> \frac43 |\widetilde W|$ 
in length, and $|\varphi_{\cF_s(V)} -\varphi_{\widetilde W}|<2\sigma$. The purpose of (i) is to ensure that $\cF_s(V)$ is a single homogeneous component, and (ii) is intended to ensure
proper crossing, the $\frac43$ providing some room to accommodate the slight difference between $\fS_s$ and $\widetilde \fS$
(the ``end points" of $\fS_s \cap \widetilde W$ and $\widetilde \fS \cap \widetilde W$
may differ by $\frac{99}{100}|\widetilde W|$).
It is straightforward to check that (i) and (ii) are assured if each of the constituent 
maps~$F_m$ is sufficiently close to $\widetilde F$ in $C^0$-distance
and $\widetilde F$ has a uniformly bounded derivative on the relevant domain (which is bounded away from the bad set). These properties can be guaranteed by  
taking~$\ve$ small.

\medskip
To finish the proof, it suffices to show that there exists a constant $c>0$ (not depending
on $(\cK_m)$ or on $W$) such that if $W_{s,i}^\text{super}$ and $V \subset \widetilde F^{-s}
W_{s,i}^\text{super}$ are as above, then 
$$
(\cF_s|_{V})_*\nu\,(\cF_s V\cap \fS_s) \geq c \cdot (\widetilde F^{s})_*\nu\,(
\widetilde F^{s} V\cap \widetilde \fS)\ .
$$
By Lemma \ref{lem:distortion}, 
\beqn
(\cF_s|_{V})_*\nu\,(\cF_s V\cap \fS_s) \geq \nu(V) e^{-C_\mathrm{r}} {\fm_{\cF_s V}(\fS_s)}/{|\cF_s V|}
\eeqn
and
\beqn
(\widetilde F^{s})_*\nu\,(\widetilde F^{s} V\cap \widetilde \fS) \leq \nu(V)e^{C_\mathrm{r}} {\fm_{\widetilde F^{s} V}(\widetilde \fS})/{|\widetilde F^{s} V|}.
\eeqn
Since $|\widetilde F^{s} V|\geq \fm_{\widetilde F^{s} V}(\widetilde \fS)$ and $|\cF_s V|$ is uniformly bounded from above, it remains to show that $\fm_{\cF_s V}(\fS_s)$ is uniformly
bounded from below, and that is true by the absolute continuity of stable manifolds
in $\fS_s$ (Lemma~\ref{lem:Jac_holonomy}) and the fact that $\fm_{\widetilde W}(\fS_s \cap 
\widetilde W) \ge \frac{99}{100}|\widetilde W|$.
\end{proof}

We remark that $s$ and $\ve$ in Proposition \ref{prop:magnets} 
depend strongly on $\widetilde \cK$ but are independent of $(\cK_m)$ or $W$. 
The argument is a perturbative one, and it is feasible only because
it does not involve more than a finite, namely $s$, number of iterates.
We remark also that stronger estimates on $\zeta$ than the one above can probably
be obtained by leveraging the $C^1$ proximity of $F_m$ to $\widetilde F$,

The next result extends Proposition \ref{prop:magnets} to more general initial measures
and coupling times. 
See Sections~\ref{sec:stacks} and~\ref{sec:primed_statements} for definitions.

\begin{cor}\label{cor:magnets} Let $\widetilde \cK$ be fixed, and let 
$s, s', \ve$ and $\zeta$ be as in Proposition~\ref{prop:magnets}. Then the following holds for every $n \in \IZ^+$ with $n \ge s+n_p$ and
every sequence $(\cK_m)$ satisfying $\cK_m \in \cN_\ve(\widetilde \cK)$ for $n-s \le m \le n+s'$:
Let $\mu$ be an initial probability measure that is regular on unstable curves and proper 
(i.e., $\cZ<C_\mathrm{p}$). Then
\eqref{eq:cross_bound} holds for $(\cF_n)_*\mu$ with $\zeta_1=\zeta$. 
\end{cor}

\begin{proof} The discussion immediately following Definition \ref{defn:proper}  is relevant
here. Since $n-s \ge n_p$, the measure $(\cF_{n-s})_*\mu$ is again proper; thus at least half of $(\cF_{n-s})_*\mu$
can be disintegrated into measures supported on regular unstable curves of length $\ge (2C_p)^{-1}$. 
Then Proposition \ref{prop:magnets} can be applied, giving \eqref{eq:cross_fraction}
with a factor of~$\frac12$ on the right side.
\end{proof}


\section{Proof of Theorem~\ref{thm:weak_conv}: the countable case}\label{sec:proof_countable}

The purpose of this section is to go through the proof of the {\it countable case} of 
Theorem~\ref{thm:weak_conv} from beginning to end, connecting the individual ingredients 
discussed in the last section. For initial distributions,
we start from the most general kind permitted in this paper, namely those introduced in Section~\ref{sec:primed_statements} under Theorem~1'.  It was observed in 
Section~\ref{sec:preliminariesII} that by delaying the first coupling, 
each initial distribution can be assumed to be regular on unstable curves and proper.
For simplicity, we will start from that. Also, as noted before,
it suffices to consider a single initial distribution, for the two measures will be coupled
to reference measures and hence to each other.


\subsection{Coupling times}  \label{sec:couplingtimes}
To each \mbox{$\widetilde \cK\in\IK$}, we first assign values to the constants $\tilde\ve(\widetilde\cK)$ and $\widetilde N(\widetilde\cK)$ appearing in the formulation of the theorem. Namely, we set $\tilde\ve(\widetilde\cK)=
\ve, s(\widetilde\cK)=s$ and $s'(\widetilde\cK)=s'$ where $\ve, s$ and $s'$ are as in Proposition~\ref{prop:magnets}, and let $r=r(\widetilde\cK)$ be the maximum of the similarly named constant in Lemma~\ref{lem:proper_part} and of $s'(\widetilde\cK)$. We then set $\widetilde N(\widetilde\cK) = s(\widetilde\cK)+r(\widetilde\cK)$. For future use, let $\zeta=\zeta(\widetilde\cK)$ be as in Proposition~\ref{prop:magnets}, and $C=C(\widetilde\cK)$ and $\lambda$ as in Lemma~\ref{lem:proper_part}. 

Next, we fix reference configurations $(\widetilde\cK_q)_{q=1}^Q$ with
 $\widetilde \cK_{q+1}\in \cN_{\tilde\ve(\widetilde \cK_q)}(\widetilde\cK_{q})$ for 
 $1\leq q< Q$ and a sequence $(\cK_n)_{n=0}^N$ adapted to $(\widetilde \cK_q,\tilde\ve(\widetilde \cK_q),\widetilde N(\widetilde \cK_q))_{q=1}^Q$. Such a sequence
is admissible. If the sequence is finite ($N<\infty$), augment it to an infinite one by setting $\cK_n=\cK_N$ for all $n>N$
 (so stable manifolds are well defined). 

The following are considerations in our choice of coupling times. 
\begin{itemize}
\item[(a)] Suppose the $k$th coupling occurs at time $t_k$ and $n_{q-1} \le t_k \le n_q$. 
We require that for $m \in [t_k-s(\widetilde \cK_q), t_k+r(\widetilde \cK_q)]$ ,
$\cK_m \in \cN_{\tilde\ve(\widetilde\cK_q)}(\widetilde \cK_q)$.

\item[(b)] There exists $\Delta=\Delta((\widetilde\cK_q)_{q=1}^Q)$ such that $t_{k+1}-t_k 
\le \Delta$.

\item[(c)] There exists $\Delta_0=\Delta_0((\widetilde\cK_q)_{q=1}^Q)>0$ such that 
$t_{k+1}-t_k \ge \Delta_0$.
\end{itemize}
The reasons for (a) are explained in Sections~\ref{sec:magnets}--\ref{sec:mixing}. The purpose of (b) is to 
ensure the exponential estimate in the theorem: at most a fraction $\zeta$
of the still uncoupled measure is matched \emph{per coupling}. The reason for (c) is a little more subtle: it is not necessarily advantageous to couple as often as one can, because
each coupling matches a $\zeta$-fraction of the measure that is available for
coupling, but renders at the same time a fraction $C\lambda^m$ of it improper, hence unavailable 
for coupling in the near future. Intuitively at least, it may be meaningful, especially if
$C$ and $\lambda$ are large, to wait till a sufficiently large part of the uncoupled measure
has recovered, i.e., has rejoined $\mu^{\rm P}_m$, before performing the next coupling.
This is discussed in more detail in Section~\ref{sec:bookkeeping}.

There are many ways to choose~$t_k$. Postponing the choice of~$\Delta_0$ to Section~\ref{sec:bookkeeping}
(and assuming for now it is a preassigned number), an algorithm may go as follows: 
Start by fitting into each time interval $(n_{q-1}, n_q), q=1, \dots, Q$, as many disjoint subintervals of length $s(\widetilde \cK_q)+r(\widetilde \cK_q)$ as one can; by definition,
at least one such interval can be fitted into each $(n_{q-1}, n_q)$. Label these intervals
as $J_i=[t'_i-s(\widetilde \cK_q), t'_i+r(\widetilde \cK_q)]$ with $t'_1< t'_2< \cdots$.
This is not quite our desired sequence of coupling times yet, as it need not respect (c)
above. To fix that, we let $t_1=t'_1$, and 
let $t_2= t'_i$ where $i>1$ is the smallest integer such that $t_2 -t_1 \ge \Delta_0$.
Continuing, we let $t_3 = t'_i$ where $i$ is the smallest number such that
$t'_i>t_2$ and $t_3 -t_2 \ge \Delta_0$, and so on. Then~(a) is satisfied
by definition, and we check that (b) is satisfied with
$$\Delta = 2 \max_{1\leq q\leq Q} s(\widetilde\cK_q) + 2\max_{1\leq q\leq Q} r(\widetilde\cK_q) + \Delta_0\ .$$

\subsection{The coupling procedure and recovery of densities}\label{sec:procedure} We review the procedure briefly, setting some
notation at the same time. Once the coupling times $t_k$
have been
fixed, we construct, for each $k$, a magnet ~$\fS_{t_k}$ using the reference 
configuration $\widetilde \cK_q$ if 
$n_{q-1} < t_k < n_q$. Let $r_k=r(\widetilde\cK_{q})$, $s_k=s(\widetilde\cK_{q})$, $\zeta_k=\zeta(\widetilde\cK_q)$, and $C_k=C(\widetilde\cK_q)$.

Suppose at time $n=t_k-s_k$ we have at our disposal a \emph{proper} measure $\tilde\mu_k$ with  total mass $P_k=\tilde\mu_k(\cM)$, ready to be used in the $k$th coupling.
In particular, $\tilde\mu_1=\mu$ and $P_1 = 1$, since the initial probability measure $\mu$ is assumed to be proper. By Corollary~\ref{cor:magnets}, a $\zeta_k$-fraction of
$(\cF_{t_k,t_k-s_k+1})_*\tilde\mu_k$ is coupled to the reference measure~$\widetilde \fm_{t_k}$ on $\fS_{t_k}$. In the language of Section~\ref{sec:outline}, Paragraph A, 
$\bar\mu_{t_k}$ is the part of $\mu$ such that 
$(\cF_{t_k})_*\bar\mu_{t_k}$ is equal to the part of $(\cF_{t_k,t_k-s_k+1})_*\tilde\mu_k$ coupled
at time $t_k$; in particular, 
\beq\label{eq:coupled}
\bar\mu_{t_k}(\cM) = \zeta_k P_k.
\eeq 
The uncoupled part of $(\cF_{t_k,t_k-s_k+1})_*\tilde\mu_k$ consist of a countable family of measured unstable curves, including arbitrarily short gaps among others, as discussed several times earlier, and has total mass $(1-\zeta_k)P_k$. With the aid of 
Lemma~\ref{lem:proper_part} (with $\tilde\mu_k/P_k$ in the role of $\mu$), we identify
out of its push-forward under $\cF_{t_k+m,t_k+1}$  a \emph{proper} part,  and call the rest the ``non-proper" part, with  the latter rejoining the first at a certain rate. Deviating from the notation of the lemma in order not to overburden the notation, we denote these parts $\tilde\mu_{k,m}^\mathrm{P}$ and $\tilde\mu_{k,m}^\mathrm{G}$, respectively (``$\mathrm{P}$" for proper,
``$\mathrm{G}$" for gap). It can be 
arranged so that  
\beq\label{eq:proper_gap_split}
\tilde\mu_{k,m}^\mathrm{P}(\cM) = (1-\zeta_k-C_k\lambda^m)P_k 
\quad\text{and}\quad
\tilde\mu_{k,m}^\mathrm{G}(\cM) =  C_k\lambda^m P_k
\qquad (m\geq r_k).
\eeq

\subsection{Bookkeeping and exponential bounds}\label{sec:bookkeeping}

Letting $u_k = (t_{k+1}-s_{k+1})-t_k \ge r_k$, the total mass $P_{k+1}$ of the proper measure available for coupling at time $t_{k+1}$ satisfies
\beqn
P_{k+1} = \tilde\mu_{k,u_k}^\mathrm{P}(\cM) + \sum_{j=1}^{k-1}\!\left(\tilde\mu_{j,t_{k-1}+u_{k-1}-t_j}^\mathrm{G}(\cM)-\tilde\mu_{j,t_k+u_k-t_j}^\mathrm{G}(\cM)\right).
\eeqn
The first term comes directly from the $k$th coupling, as explained above. In the second term we take into account the fact that at the $j$th coupling, $1\leq j<k$, some measure was deposited into the ``non-proper'' part, and what remains of that part at a later time $n$ is the measure $\tilde\mu_{j,n-t_j}^\mathrm{G}$. Thus this sum represents the total mass that
was not available for the $k$th coupling but has become available for the $(k+1)$st.
Plugging in the numbers from \eqref{eq:proper_gap_split}, we obtain
\beq\label{eq:P_recursion}
\begin{split}
P_{k+1} &= (1-\zeta_k-C_k\lambda^{u_k}) P_k + \sum_{j=1}^{k-1} C_j\lambda^{t_{k-1}-t_j+u_{k-1}} (1-\lambda^{u_k+s_k}) P_j  .
\end{split}
\eeq
We also have the following expression for  the total mass that remains uncoupled immediately after the $k$th coupling, i.e., $\mu_{t_k}(\cM)$ in the language of Section~\ref{sec:outline}, Paragraph A:
\beq\label{eq:remainder}
\begin{split}
\mu_{t_k}(\cM)
&=(1-\zeta_k)P_k +  \sum_{j=1}^{k-1} C_j\lambda^{t_{k-1}-t_j+u_{k-1}} P_j.
\end{split}
\eeq
Here the first term is the measure that was ``eligible" for coupling at time $t_k$ but was
not coupled, and the second sum consists 
of terms coming from earlier couplings that at time $t_{k-1}+u_{k-1}=t_k-s_k$ were still not ready to be coupled.

\medskip
Notice that \eqref{eq:P_recursion} is a recursion relation for the sequence $(P_k)$ with 
the initial condition $P_1=1$.  We need to show that $P_k$  tends to zero exponentially 
with $k$.

\begin{lem}\label{lem:P_bound} Let $\tilde\zeta = \min_{1\leq q\leq Q}\zeta(\widetilde\cK_q)$. Suppose 
$$
\Delta_0 = \left\lceil\log\Bigl(\tfrac12\tilde\zeta(1-\tilde\zeta)\big/\max_{1\leq q\leq Q} C(\widetilde\cK_q)\Bigr)/\log \lambda\right\rceil + \max_{1\leq q\leq Q} s(\widetilde\cK_q) 
+ n_\mathrm{p}.
$$
Then a coupling strategy satisfying (a) and (c) in Section \ref{sec:couplingtimes} will produce
a sequence of $P_k$ with   
\beq\label{eq:P_bound}
P_k\leq (1-\tfrac12\tilde\zeta)^k \qquad \forall \ k \ge 1\ .
\eeq
\end{lem}

We have included $n_\mathrm{p}$ in the definition of $\Delta_0$ to allow for the
transient loss of properness when a proper measure is pushed forward (Section~\ref{sec:stacks});
it plays no role in the proof below.

\begin{proof}[Proof of Lemma~\ref{lem:P_bound}]
We form a majorizing sequence $(Q_k)$ with $Q_1=1=P_1$ and
\beq\label{eq:Q_recursion}
Q_{k+1} = (1-\tilde\zeta) Q_k + \sum_{j=1}^{k-1} C_j\lambda^{t_{k-1}-t_j+u_{k-1}} Q_j  .
\eeq
Clearly $P_k\leq Q_k$ by comparing \eqref{eq:P_recursion} and \eqref{eq:Q_recursion}.
We want to show that $Q_k$ tends to zero exponentially with $k$, with $Q_{k+1} \leq (1-\tfrac12\tilde\zeta) Q_k$. The bound is certainly implied if
\beqn
\begin{split}
\sum_{j=1}^{k-1} C_j\lambda^{t_{k-1}-t_j+u_{k-1}} Q_j  \leq \tfrac12\tilde\zeta Q_k, \
\quad k \ge 2\ ,
\end{split}
\eeqn
and this is what we will prove.

Observe from $t_{k+1}-t_k \ge \Delta_0$
and the choice of $\Delta_0$ above that
\beq\label{eq:Delta_condition}
\max_{1\leq q\leq Q} C(\widetilde\cK_q) \cdot \lambda^{t_{k+1}-t_k-\max_{1\leq q\leq Q} s(\widetilde\cK_q)} \leq \tfrac12\tilde\zeta(1-\tilde\zeta).
\eeq
Together with $C_{k-1}\geq 1$, this gives
\beqn
\begin{split}
\sum_{j=1}^{k-1} C_j\lambda^{t_{k-1}+u_{k-1}-t_j} Q_j 
&= 
C_{k-1}\lambda^{u_{k-1}}\!\left( Q_{k-1}  + \frac{\lambda^{s_{k-1}}}{C_{k-1}}\sum_{j=1}^{k-2}  C_j\lambda^{t_{k-2}+u_{k-2}-t_j} Q_j\right)
\\
&\leq 
C_{k-1}\lambda^{u_{k-1}}\!\left( Q_{k-1}  + \sum_{j=1}^{k-2}  C_j\lambda^{t_{k-2}+u_{k-2}-t_j} Q_j\right)
\\
&\leq 
\tfrac12\tilde\zeta(1-\tilde\zeta) \!\left(Q_{k-1}  + \sum_{j=1}^{k-2}  C_j\lambda^{t_{k-2}+u_{k-2}-t_j} Q_j\right) \leq \tfrac12\tilde\zeta Q_k.
\end{split}
\eeqn
The second to last inequality uses (\ref{eq:Delta_condition}), and the
last inequality is from \eqref{eq:Q_recursion}. Hence, $P_k \leq Q_k \leq (1-\tfrac12\tilde\zeta)^{k-2}Q_2$ for $k\geq 2$. Finally, $Q_2 = (1-\tilde\zeta)$ and $P_1 = 1$, yield $P_k \leq (1-\tfrac12\tilde\zeta)^{k-1}$ for all $k\geq 1$.
\end{proof}

\begin{cor}\label{lem:uncoupled}
For any $n\geq 0$,
\beqn
\bar\mu_n(\cM) \leq \tilde\zeta (1-\tfrac12\tilde\zeta)^{n/\Delta}
\quad
\text{and}
\quad
\mu_n(\cM)\leq (1-\tfrac12\tilde\zeta)^{n/\Delta-1}.
\eeqn 
\end{cor}

\begin{proof}
By \eqref{eq:Delta_condition} and $C_{k-1}\geq 1$ (Lemma~\ref{lem:proper_part}), we have $C_j\lambda^{t_k-t_j}\leq \bigl(\frac12\tilde\zeta(1-\tilde\zeta)\bigr)^{k-j}$ in \eqref{eq:remainder}. Inserting also the exponential bound on $P_k$ in \eqref{eq:P_bound} and computing the resulting sum yields easily $\mu_{t_k}(\cM)\leq (1-\tfrac12\tilde\zeta)^{k+1}$. Recalling~\eqref{eq:coupled}, also $\bar\mu_{t_k}(\cM) \leq \tilde\zeta (1-\tfrac12\tilde\zeta)^k$. Observe that $t_k \leq k\Delta$ by $t_{j+1}-t_j 
\le \Delta$ and $t_1\leq \Delta$.  Thus,  $\mu_{t_k}(\cM)\leq (1-\tfrac12\tilde\zeta)^{t_k/\Delta+1}$ and $\bar\mu_{t_k}(\cM) \leq \tilde\zeta (1-\tfrac12\tilde\zeta)^{t_k/\Delta}$. By definition, $\mu_n=\mu_{t_k}$ and $\bar\mu_n = 0$ for $t_{k}<n<t_{k+1}$; see~Section~\ref{sec:outline}. We therefore obtain $\bar\mu_n(\cM) \leq \tilde\zeta (1-\tfrac12\tilde\zeta)^{n/\Delta}$ for all $n\geq 0$ and $\mu_n(\cM)\leq (1-\tfrac12\tilde\zeta)^{t_k/\Delta+1}$ for $t_{k}\leq n<t_{k+1}$ ($k\geq 1$). Using $t_{j+1}-t_j 
\le \Delta$ once more, we see in the latter case that $t_k \geq n-\Delta$. In other words, $\mu_n(\cM)\leq (1-\tfrac12\tilde\zeta)^{n/\Delta}$ for all $n\geq t_1$, or $\mu_n(\cM)\leq (1-\tfrac12\tilde\zeta)^{n/\Delta-1}$ for all $n\geq 0$ (as again $t_1\leq \Delta$).
\end{proof}


\subsection{Memory-loss estimate}\label{sec:mem_loss}

Finally, we prove the estimate in~\eqref{eq:convergerate}, along the lines of~\eqref{eq:diff}.

Consider two measures $\mu^i$, $i=1,2$. Recalling that $\mu^i = \mu^i_{n/2} + \sum_{j\leq n/2}\bar\mu^i_j$ and using Lemma~\ref{lem:uncoupled}, we see that 
\beqn
\begin{split}
\left|\int f \circ \cF_n \,\rd\mu^1 - \int f \circ \cF_n \,\rd\mu^2 \right|
\leq  (1-\tfrac12\tilde\zeta)^{n/2\Delta-1} \|f\|_\infty + 
\sum_{j \leq n/2} \left|\int f \circ \cF_n \,\rd\bar\mu^1_j - \int f \circ \cF_n \,\rd\bar\mu^2_j \right|.
\end{split}
\eeqn
Here $\bar\mu^1_j=\bar\mu^2_j=0$ unless $j=t_k$ for some $k=1,2,\dots$. At times $j=t_k$ a coupling occurs: Recall from Paragraph~C of Section~\ref{sec:outline} that (for each $i=1,2$) $(\cF_j)_*\bar\mu_j^i$ is coupled to the reference measure $a_j\widetilde\fm_j(\,\cdot\,\cap\fS_j)$, where $a_j=\bar\mu_j^i(\cM)=((\cF_j)_*\bar\mu_j^i)(\fS_j)$ and $\widetilde\fm_j(\fS_j)=1$.  Moreover, according to \eqref{eq:components} and \eqref{eq:component}, the measure $(\cF_j)_*\bar\mu_j^i$ is a sum of countably many components, namely $\sum_{\alpha\in\cA^i}\sum_{m\in\cI^i_{\alpha,j}}\lambda^i_{\alpha,j,m}\,(\bfh^i_{\alpha,j,m})_*\widetilde\fm_j$, where $\bfh^i_{\alpha,j,m}$ is a holonomy map associated to the stable manifolds of the magnet $\fS_j$ and $\sum_{\alpha\in\cA^i}\sum_{m\in\cI^i_{\alpha,j}}\lambda^i_{\alpha,j,m}=\bar\mu_j^i(\cM)$. 
If $f$~is $\gamma$-H\"older continuous with some exponent $\gamma>0$, we can estimate, similarly to \eqref{eq:memloss_curve}, that
\beqn
\begin{split}
& \left|\int f \circ \cF_n \,\rd\bar\mu^1_j - \int f \circ \cF_n \,\rd\bar\mu^2_j \right| 
= \left|\int f \circ \cF_{n,j+1} \,\rd ((\cF_j)_*\bar\mu^1_j) - \int f \circ \cF_{n,j+1} \,\rd ((\cF_j)_*\bar\mu^2_j)\right|
\\
& \qquad \qquad 
\leq  \sum_{i=1,2}  \sum_{\alpha\in\cA^i}\sum_{m\in\cI^i_{\alpha,j}} \lambda^i_{\alpha,j,m} \left|\int_{\fS_j} f \circ \cF_{n,j+1} \, \rd ((\bfh^i_{\alpha,j,m})_*\widetilde\fm_j)- \int_{\fS_j} f \circ \cF_{n,j+1} \,\rd\widetilde\fm_j \right|
\\
& \qquad \qquad  
\leq  \sum_{i=1,2}  \sum_{\alpha\in\cA^i} \sum_{m\in\cI^i_{\alpha,j}} \lambda^i_{\alpha,j,m} \, |f|_\gamma (\hat c^{-1}\Lambda^{-(n-j)})^{\gamma}
 =
\sum_{i=1,2}\bar\mu_j^i(\cM) |f|_\gamma (\hat c^{-1}\Lambda^{-(n-j)})^{\gamma}.
\end{split}
\eeqn
Since $\sum_{j}\bar\mu_j^i(\cM)=1$,
\beqn
\begin{split}
\sum_{j\leq n/2}\left|\int f \circ \cF_n \,\rd\bar\mu^1_j - \int f \circ \cF_n \,\rd\bar\mu^2_j \right| 
&\leq
2 |f|_\gamma (\hat c^{-1}\Lambda^{-n/2})^{\gamma}.
\end{split}
\eeqn
Combining the above estimates yields the bound in
\eqref{eq:convergerate} with $C_\gamma = 2\max((1-\tfrac12\tilde\zeta)^{-1},\hat c^{-\gamma})$ and $\theta_\gamma = \max((1-\tfrac12\tilde\zeta)^{1/2\Delta},\Lambda^{-\gamma}/2)$. 


Notice that the constants $\Delta$ and $\tilde\zeta$ are determined by the set of reference configurations $\{\widetilde\cK_q\}_{q=1}^Q$; their order and number of appearances are irrelevant. This completes the proof of the countable case of Theorem~\ref{thm:weak_conv}.
\qed


\section{Completing the proofs}\label{sec:completion}

In this section we complete the proofs of Theorems 1'--2' and 4', restricted versions of which 
are stated in Sections~\ref{sec:results} and \ref{sec:tech},
and the full versions in Section~\ref{sec:primed_statements}.
To do that, we must first treat the ``continuous case" of Theorem~\ref{thm:weak_conv}, which is used to
give the full versions of all the other results.

\subsection{Proof of Theorem~\ref{thm:weak_conv}: continuous case}\label{sec:cont_proof}

In Section~\ref{sec:stacks}, we introduced the idea of measured unstable families, defined to be
convex combinations of measured unstable stacks. The ``continuous case" of Theorem~\ref{thm:weak_conv}
refers to the version of Theorem~\ref{thm:weak_conv} for which initial measures are of this form. 
The proof proceeds almost exactly as in the countable case, so we focus here only
on the differences.

Three types of processes are involved in the proof of Theorem~\ref{thm:weak_conv}: 
(i) canonical subdivisions, (ii) coupling to reference measures, and (iii) recovery of 
densities following the couplings.
The process of pushing forward measured unstable families and canonical subdivisions
was discussed in Section~\ref{sec:stacks}. We noted that this process produces objects of
the same kind, i.e., canonical subdivisions of measured unstable families are measured
unstable families. Regularity and properness, including their recovery properties, 
were also discussed: there is no substantive difference between the countable and continuous cases since these are essentially properties on individual unstable curves; 
in the continuous
case, one simply replaces summations in the countable case by integrals.

We provide below more detail on (ii):

\medskip \noindent
\emph{The coupling procedure: continuous case.} Consider the situation in Section~\ref{sec:procedure}, 
where at time $n=t_k-s_k$ we have a proper measure $\tilde \mu_k$ of
continuous type ready to be used in the $k$th coupling. We sketch below a few
issues that require additional care:

\medskip
 In the countable case,
we observed in Corollary~\ref{cor:magnets} that at least half of $\tilde \mu_k$ is 
supported on unstable curves of length $\ge (2C_{\mathrm p})^{-1}$;
the same is true here, as it is a general fact. But then in the countable case, we 
applied Proposition~\ref{prop:magnets} 
to {\it one unstable curve at a time}, comparing the action of~$\cF_s$ to that of~$\widetilde F^s$ 
on each curve to obtain the asserted bound on the fraction that can be coupled. 
Here it is not legitimate to argue one curve at a time, so we proceed as follows: 
Noting that~$\tilde \mu_k$ is supported on a countable number of unstable stacks,
we plan to subdivide these stacks in such a way that there is a 
a collection of countably many ``thin enough" stacks with the following properties:
(i) their union carries at least half of $\tilde \mu_k$, and (ii) on each thin stack 
all the curves have length at least $(2C_{\mathrm p})^{-1}$ (or thereabouts). 
We then treat these thin stacks with long curves one at a time.
The conditions for ``thin enough" 
are basically that the stack should behave as though it was a single curve
in the next $s_k$ steps. 

More precisely, pick one of the measured unstable stacks 
$(\cup_{\alpha\in E}W_\alpha,\mu)$ associated to $\tilde \mu_k$, and consider 
its canonical $s_k$-step subdivision (associated to the sequence $F_{t_k-s_k+1},\dots,F_{t_k}$) into stacks of the form $\cup_{\alpha\in E_{s_k,i,j}}(W_\alpha\cap\overline \cD_{s_k,i})$  as discussed in Section~\ref{sec:stacks}. We first specify what we mean by a ``thin enough"
substack of $\cup_{\alpha\in E}W_\alpha$. Let $\hat \alpha \in E$ be such that
$|W_{\hat \alpha}| \ge (2C_{\mathrm p})^{-1}$, 
and assume the images of this unstable curve
in the next $s_k$ steps do not pass through branch
points of the discontinuity set. Then there is a small neighborhood $E_{\hat \alpha}$
of $\hat \alpha$ in $E$ such that the following holds for all
$\beta \in E_{\hat \alpha}$: 
For each $i$ such that $\cF_{t_k, t_k-s_k+1}(W_{\hat \alpha}\cap\overline \cD_{s_k,i})$ 
crosses $\fS_{t_k}$ properly, the same holds for
 $\cF_{t_k, t_k-s_k+1}(W_{\beta}\cap\overline \cD_{s_k,i})$ with a slightly relaxed
definition of ``proper crossing" that is good enough for our purposes.
Moreover, if $\hat \alpha \in E_{s_k,i,j}$, then $E_{\hat \alpha} \subset E_{s_k,i,j}$.
We are guaranteed that $E_{\hat \alpha}$ exists because there are only finitely 
many such proper crossings for each $W_{\hat \alpha}$.
The stack $\cup_{\alpha\in E_{\hat \alpha}} W_\alpha$
is ``thin enough".

Assuming that the transverse measure $P$ on $E$ has no atoms 
(the argument is easily modified if it does), there is a finite number of disjoint intervals 
of the form $E_{\hat \alpha_l}$ where 
$|W_{\hat \alpha_l}| \ge (2C_{\mathrm p})^{-1}$ and 
$\cup_{\alpha \in \cup_l E_{\hat \alpha_l}}  W_{\alpha}$ carries more than
$99\%$ of the part of $\mu$ supported on $W_\alpha$-curves 
of length $\ge (2C_{\mathrm p})^{-1}$. The procedure is to first subdivide 
$E$ into $\{E_{\hat \alpha_l}\}$ and the connected components of $E \setminus
\cup_l E_{\hat \alpha_l}$.
This corresponds to subdividing the original stack $(\cup_{\alpha\in E}W_\alpha,\mu)$
before proceeding with the canonical subdivision. At time $t_k$, we consider one $l$
at a time: For each $i$ such that $\cF_{t_k, t_k-s_k+1}(W_{\hat \alpha_l}\cap\overline \cD_{s_k,i})$ 
crosses $\fS_{t_k}$ properly, the $\cF_{t_k, t_k-s_k+1}$-image of 
$\cup_{\alpha\in E_{s_k,i} \cap E_{\hat \alpha_l}}(W_\alpha\cap\overline \cD_{s_k,i})$ is 
a single unstable stack every curve in which crosses $\fS_{t_k}$ properly. 
A fraction of the conditional probability measures on each unstable curve is
coupled to $\widetilde \fm_{t_k}$ as before.
These are the only stacks on which couplings will be performed at time $t_k$.

To obtain the desired lower bound on the fraction of 
$(\cF_{t_k, t_k-s_k+1})_*\tilde \mu_k$ coupled, we prove 
a slight generalization of Proposition~\ref{prop:magnets} in which
the measured stack $(\cup_{\alpha\in E_{\hat \alpha_l}} W_\alpha, \mu|_{\cup_{\alpha\in E_{\hat \alpha_l}} W_\alpha})$ takes the place of $(W,\nu)$. The argument is virtually
identical (and omitted); since the conditional measures have the same uniform bounds.

After a coupling, we must also show that the uncoupled part of 
$(\cF_{t_k, t_k-s_k+1})_*\tilde \mu_k$ is again supported on at most a countable number 
of measured unstable stacks. Treating first the curves (without the measures),
we observe that for each $i$ and $l$ in the next to last paragraph, 
after the coupling there are two stacks corresponding to the excess pieces of 
$\cF_{t_k, t_k-s_k+1}(W_{\hat \alpha_l}\cap\overline\cD_{s_k,i})$, a third stack 
which is $\cF_{t_k, t_k-s_k+1}(\cup_{\alpha\in E_{s_k,i} \cap E_{\hat \alpha_l}}(W_\alpha\cap\overline\cD_{s_k,i}))$ minus the first two, 
plus a countable number of stacks one for  each gap.
We also need to decompose the uncoupled part of the measure in the same way as
was done in Section~\ref{sec:recovery}. In particular, a slight generalization of the extension  lemma 
(Lemma~\ref{lem:regular_extension}) leading to 
the ``top conditional densities" in the third stack is needed. We leave this technical but straightforward exercise to the reader.

Finally, we observe that the subdivision of a stack into thinner stacks (without cutting
any of the unstable curves in the stack) does not increase the $\cZ$-value of a family. 

\medskip
The rest of the proof is unchanged from the countable case.

\bigskip 
This concludes the proof of Theorem~4', that is, the extension of Theorem~\ref{thm:weak_conv} to the larger class of initial measures permitted in Theorem~1'. Theorem~1' then follows, in the same way as Theorem~\ref{thm:weak_conv_compact} was deduced from
Theorem~\ref{thm:weak_conv}; see Section~\ref{sec:statements}. 

\subsection{Scatterers with variable geometries}\label{sec:var_geom}

To understand what additional arguments are needed as we go from scatterers
with fixed geometries to scatterers with variable geometries, recall that the proof of Theorem~1' 
has two distinct parts: one is \emph{local}, and the other \emph{global}. The local result
is contained in Theorem~\ref{thm:weak_conv}, which treats essentially time-dependent sequences $(\cK_n)$ near a fixed reference configuration $\widetilde \cK$. It also shows 
how the scheme can be continued as the time-dependent sequence moves from the
sphere of influence of one reference configuration to that of another. The rate of memory loss,
however, depends on the set of reference configurations visited. The global part of the proof
seeks to identify a suitable space, as large as possible,
for which one can have a uniform convergence rate for the measures involved.
For scatterers with fixed geometries, this is done by showing that the entire configuration
space of interest can be ``covered" by a finite number of reference configurations
$\{\widetilde \cK_1, \dots, \widetilde \cK_Q\}$, i.e.,
no matter how long the time-dependent
sequence, it is, at any one moment in time, always ``within radar range" of one of 
the $\widetilde \cK_q, 1\leq q\le Q$. The argument is thus reduced to the local one. 
Details are given at the end of Section~\ref{sec:statements}. 

\smallskip
\begin{proof}[Proof of Theorem 2'] We discuss separately the local and global parts
of the argument.

\medskip \noindent
{\it Local part:} 
We claim that the local part of the proof, i.e., Theorem 4', extends \emph{verbatim} to
the setting of variable scatterer geometry, and leave the step-by-step verification to the reader.
For example, the arguments in Sections~\ref{sec:preliminariesI} and~\ref{sec:preliminariesII} are entirely oblivious to the fact that the shapes of the scatterers change with time, in the same way that they are oblivious to
their changing locations, for as long as their curvatures  and flight times lie within specified ranges. The more sensitive parts
of the proof involve $(\cK_m) \subset \cN_\ve(\widetilde \cK)$, where $\cN_\ve(\cdot)$ is
now defined using the $d_3$-metric introduced in Section~\ref{sec:results}. Notice that as before,
(i) for $\cK, \cK' \in \cN_\ve(\widetilde \cK)$, the singularity set for 
$F=F_{\cK', \cK}$ lies in a small neighborhood of the singularity set for $\widetilde F =
F_{\widetilde \cK, \widetilde \cK}$, and (ii) a fixed distance away from these singularity sets, 
$F$ and $\widetilde F$ can be made arbitrary close in $C^0$ as $\ve \to 0$. These properties
are  sufficient for the arguments needed, including the uniform mixing argument in Section~\ref{sec:magnets}.

\medskip \noindent
{\it Global part:} The argument is along the lines of the one at the end of Section~\ref{sec:statements}, but involves different spaces and different norms. 
In order to reduce to the local argument, we need to establish some compactness.
Decreasing~$\bar \kappa^{\rm min}$ and~$\bar \tau^{\rm min}$, increasing
$\bar \kappa^{\rm max}$ and $\bar \tau^{\rm max}$, as well as increasing~$\Delta$ to some~$\Delta'\geq\Delta$ (to be fixed below), 
we let~$\widetilde \IK'$ denote the configuration space defined analogously to~$\widetilde\IK$ but using 
these relaxed bounds on curvature and flight times. 
We denote the closure of $\widetilde\IK$ with respect to the metric $d_3$ by 
$c\ell(\widetilde \IK)$.
We will show that $c\ell(\widetilde \IK)$ is a compact subset of $\widetilde\IK'$.

First, a constant $\hat \Delta$ can be fixed so that for all $\cK \in \widetilde \IK$,
if $\hat \gamma_i : \IS^1 \to \IT^2$ is the constant speed parametrization
of $\partial \bB_i$ in Section~\ref{sec:results}, then $\|D^k\hat \gamma_i\|_\infty \le \hat \Delta$ for 
$1 \le k \le 3$ and Lip$(D^3 \hat \gamma_i) \le \hat \Delta$. This is true because of 
property (i) in the definition of $\widetilde \IK$ and the fact that derivatives of~$\hat \gamma_i$ and~$\gamma_i$ (unit speed parametrization of the same scatterer)
differ only by a factor equal to the length of~$\partial \bB_i$, which is uniformly bounded above and below due to 
$\bar \kappa^{\min} < \kappa < \bar \kappa^{\max}$.
Next we argue that if the number of scatterers~$\fs$ were fixed, it would follow
that $c\ell(\widetilde \IK)$ is a compact set: Given $\fs$ sequences~$(\hat \gamma_{i,n})_{n\ge 1}$ of parametrizations as above, we first note that they are uniformly bounded. The same is true of the sequences $(D^k\hat\gamma_{i,n})_{n\ge 1}$, $1\leq k\leq 3$, as noted above. Each of these is also equicontinuous because it is uniformly Lipschitz. Hence, the Arzel\`a--Ascoli theorem yields the existence of uniform limits $\hat\gamma_{i} \equiv \lim_{j\to\infty }\hat \gamma_{i,n_j}$ and $\hat\gamma_{i}^{(k)} \equiv \lim_{j\to\infty }D^k\hat \gamma_{i,n_j}$, $1\leq k\leq 3$, $1\leq i\leq \fs$, along a subsequence~$(n_j)_{j\geq 1}$. 
Since Lipschitz constants are preserved in uniform limits, it is easy to check that $\hat\gamma^{(k)}_i = D^k\hat\gamma_i$, $\max_{1\leq k\leq 3}\|D^k\hat\gamma_i\|_\infty\leq \hat\Delta$ and $\mathrm{Lip}(D^3\hat\gamma_i)\leq \hat \Delta$. We now replace the limit parametrizations $\hat\gamma_i$, $1\le i\le\fs$, by the corresponding constant speed parametrizations~$\gamma_i$. Owing to the above bounds, they specify a configuration in~$\widetilde\IK'$, if we choose $\Delta'$ large enough.
While $\widetilde \IK'$ permits in principle an arbitrarily large number of 
scatterers, there is, in fact, a finite upper bound on~$\fs$ imposed by~$\bar \tau^{\min}$,
which is less than or equal to
the minimum distance between any pair of scatterers. We have thus proved that  
$c\ell(\widetilde \IK)$ is compact. 

 We apply the result from the local part to $\widetilde \IK'$, obtaining $\tilde \ve(\cK)$ and 
$\widetilde N(\cK)$ for each $\cK \in \widetilde \IK'$. The collection 
$\{\cN_{\frac12 \tilde \ve(\cK)}(\cK)\,:\, \cK \in c\ell(\widetilde \IK)\}$
is an open cover of $c\ell(\widetilde \IK)$, open as subsets of~$\widetilde \IK'$. 
Let $\bigl\{\widetilde \cN_q=
\cN_{\frac12 \tilde \ve(\widetilde \cK_q)}(\widetilde \cK_q), \ q \in \cQ\bigr\}$ be a finite
subcover. The rest of the proof is as in Section~\ref{sec:statements}: we apply the local
result to the given sequence $(\cK_n) 
\subset \widetilde\IK$, noting that any $(\cK_n)$ with $d_3(\cK_n, \cK_{n+1})$ sufficiently small is adapted
to a sequence of reference configurations chosen from $\{\widetilde \cN_q, q \in \cQ\}$.
\end{proof}

\begin{proof}[Proof of Theorem~\ref{thm:eq_mixing}]
Without loss of generality, we assume that $\int g\,\rd\mu = 0$. Let $a = 1-\inf g$. Then $\rd\mu' = a^{-1}(g+a)\,\rd\mu$ is a probability measure. Moreover, the density $\rho' =  a^{-1}(g+a)$ satisfies the assumptions of Theorem~\ref{var_geom}. Indeed,
\beqn
(1+\|g\|_\infty)^{-1}\leq \rho' \leq 1+\|g\|_\infty,
\eeqn
and, like $g$, $\rho'$ is $\frac16$-H\"older with its logarithm satisfying the estimate
\beqn
\left|\log\rho'(x)-\log\rho'(y) \right| \leq (1+\|g\|_\infty)\left|g(x)-g(y) \right| \leq (1+\|g\|_\infty)|g|_{\frac16}d_\cM(x,y)^{\frac16}.
\eeqn
We thus have
\beqn
\left|\int f\circ F^n\cdot g\,\rd\mu \right| = a \left | \int f\circ F^n\,\rd\mu' - \int f\circ F^n\,\rd\mu\right| \leq (1+\|g\|_\infty)\,C_\gamma ({\| f \|}_\infty + {|f|}_\gamma)\theta_\gamma^n ,
\eeqn
after an application of Theorem~\ref{var_geom} with $\mu^1 = \mu'$ and $\mu^2 = \mu$. Here $C_\gamma$ depends on the bound $(1+\|g\|_\infty)|g|_{\frac16}$ on the H\"older constant of $\log \rho'$ obtained above.
\end{proof}


\subsection{Small external fields}\label{sec:proof_fields}

In this section we discuss  modifications of earlier proofs needed to 
yield Theorem~E.

First we claim that for each $\widetilde \cK\in\widetilde\IK$, there exist
$\hat \delta_0(\widetilde \cK)$ and $E_0(\widetilde\cK)>0$ such that for all 
$\cK,\cK'\in\cN_{\hat\delta_0}(\widetilde\cK)$ and $\bE \in C^2$ with 
$\|\bE\|_\infty \leq E_0(\widetilde\cK)$, $F^{\bE}_{\cK',\cK}$ is defined,
and $\frac{9}{10} \bar \tau^{\min}< \tau^\bE_{\cK',\cK} < \frac{11}{10} \ft$.
Here $\tau^\bE_{\cK',\cK}$ is the flight time between source and target
scatterers following trajectories defined by $\bE$. To prove the asserted upper
bound for $\tau^\bE_{\cK',\cK}$, 
notice that (i) the set of straight line segments of length $\ft$ is compact,
and (ii) for a $C^0$-small $\bE$, particle trajectories deviate only slightly
from straight lines.  Thus
the $(\ft, \varphi)$-horizon property of the $\bE=0$ case guarantees that
any flow-trajectory  of length $\frac{11}{10} \ft$ will also
meet a scatterer at an angle not much below $\varphi$ (measured from the tangent).

Next we claim that there exist $\hat\delta_1(\widetilde\cK) \le 
\hat\delta_0(\widetilde\cK)$ and $E_1(\widetilde\cK) \le E_0(\widetilde\cK)$
such that the basic properties 
in Sections~\ref{sec:preliminariesI} and~\ref{sec:preliminariesII} hold (with relaxed constants) for all sequences
$(\cK_n, \bE_n)$ with the property that for each $n$, there is $\widetilde \cK$
such that $\cK_n, \cK_{n+1} \in \cN_{\hat\delta_1}(\widetilde\cK)$ 
and $\|\bE_n\|_{C^2} \leq E_1(\widetilde\cK)$. More precisely, we claim that
the maps $F_n = F^{\bE_n}_{\cK_{n+1},\cK_n}$
have the same properties as their analogs with $\bE=0$, including the geometry
of the singularity sets, stable and unstable cones, uniform expansion and
contraction rates, distortion and curvature bounds for unstable curves, absolute continuity and bounds on the Jacobians, the Growth Lemma holds, etc. 
For fixed scatterers, the main technical references for fields $\bE$ with small
enough $C^2$-norms are~\cite{Chernov_2001,Chernov_2008}. 
The results above are obtained following the proofs in these references,
except for Lemma~\ref{lem:map_cont} the proof of which is also straightforward
and left as an exercise.

Next we proceed to the analog of Theorem 4' for small external fields,
for sequences of the form $(\cK_n,\bE_n)_{n=0}^N$ adapted, in a sense
to be defined, to a finite
sequence of
reference configurations $(\widetilde \cK_q)_{q \le Q}$:
For each $q$, there exist $\tilde\ve(\widetilde \cK_q),
\tilde\ve^\mathrm{field}(\widetilde \cK_q) > 0$ and
$\widetilde N(\widetilde\cK_q) \in \IZ^+$ such that $(\cK_n)_{n=0}^N$ is adapted to
$\bigl(\widetilde\cK_q, \,\tilde\ve(\widetilde \cK_q),\widetilde N(\widetilde\cK_q) \bigr)_{q=1}^Q$ in the sense of Section~\ref{sec:tech} and, additionally, 
$\|\bE_n\|_{C^2}\le \tilde\ve^\mathrm{field}(\widetilde \cK_q)$ for the
relevant $q$. The argument proceeds as 
in Sections~\ref{sec:preliminaries_continued} and~\ref{sec:proof_countable}. 
There are exactly two places where the argument
is perturbative, and ``perturbative" here means perturbing from 
systems with fixed scatterer configurations $\widetilde \cK$ and
zero external field. One is the construction of the magnet in Section~\ref{sec:magnets}, and
the other is the uniform mixing argument (Proposition~\ref{prop:magnets}) in Section~\ref{sec:mixing}.
For each $\widetilde \cK\in\widetilde\IK$, these two arguments impose 
bounds $\tilde\ve(\widetilde \cK)$ and 
$\tilde\ve^\mathrm{field}(\widetilde \cK) > 0$ on 
$d_3(\cK_n, \widetilde \cK)$ and~$\|\bE_n\|_{C^2}$ respectively.
(We may assume $\tilde\ve(\widetilde \cK) \le \hat \delta_1(\widetilde \cK)$ and
 $\tilde\ve^\mathrm{field}(\widetilde \cK) \le E_1(\widetilde \cK)$.)
Such bounds exist because in Section~\ref{sec:magnets} we require only that the action of $\cF_{n+s',n}$ on
a specific piece of unstable curve follows that of 
$(F_{\widetilde \cK, \widetilde \cK})^{s'}$ closely in the sense of
$C^0$ for a fixed number of iterates, namely~$s'$, during which this
curve stays a positive distance from any discontinuity curve or homogeneity lines. 
The argument in Section~\ref{sec:mixing} requires a little more, but that too involves only
curves that stay away from discontinuity and homogeneity lines and also for 
only a fixed number of iterates. For appropriate choices of 
$\tilde \ve(\widetilde \cK)$ and 
$\tilde\ve^\mathrm{field}(\widetilde \cK)$, the latter made possible
by our extended version of Lemma~\ref{lem:map_cont},
these two proofs as well as others needed go through without change, yielding an analog of Theorem~4' as formulated above.

Finally it remains to go from our ``local" result, i.e., the analog of Theorem~4',
 to the ``global" one, namely Theorem E.
We cover the closure of $\widetilde\IK$
with balls centered at each $\widetilde \cK$ having $d_3$-radius~$\tilde\ve(\widetilde \cK)$ in a slightly enlarged space $\widetilde\IK'$ and choose
as before a finite subcover, consisting of balls centered at 
$\{\widetilde \cK_j\}$.
The uniform bounds $\ve^\bE$ and $\ve$ appearing in the statement of Theorem~E are given by $\ve^\bE = \min_j \tilde \ve^\mathrm{field}(\widetilde \cK_j)$ and 
$\ve = \min_j \tilde \ve(\widetilde \cK_j)$.



\appendix


\section*{Appendix. Proofs}\label{sec:proofs}


\begin{proof}[Proof of Lemma~\ref{lem:map_cont}]

We first prove continuity of the map $(x,\cK,\cK')\mapsto F_{\cK',\cK}(x)$. 
Consider an initial configuration $\cK_0$ and a target configuration $\cK'_0$ and some initial condition $x_0\in\cM$ which corresponds to a \emph{non-tangential collision}. Obviously, there exists an open neighborhood~$U$ of the triplet $(x_0,\cK_0,\cK'_0)$ in which there are no tangential collisions: each $(x,\cK,\cK')\in U$ corresponds to a head-on collision from a scatterer $\bB$ in configuration $\cK$ to a scatterer $\bB'$ in configuration $\cK'$. We can view the scatterers $\bB$ and $\bB'$ as subsets of the \emph{plane} and represent them by two vectors, $(\bc,\bu)$ and $(\bc',\bu')$ in $\IR^2\times \IS^1$, which depend continuously on $\cK$ and $\cK'$ (as long as $(x,\cK,\cK')\in U$). The $\bc$ and $\bu$ components specify the location and orientation of the scatterer, as was explained in Section~\ref{sec:statements}. Let $(\bar\bc,\bar\bu)$ be the relative polar coordinates of $\bB$ with respect to the frame attached to $\bB'$ whose origin is specified by $(\bc',\bu')$. Then $(\bar\bc,\bar\bu)$ depends continuously on the pair $(\cK,\cK')$. We write $(\bar\bc,\bar\bu)=G(\cK,\cK')$ and point out that $\mathrm{id}_\cM\times G:(x,\cK,\cK')\mapsto (x,G(\cK,\cK'))$ is continuous on $U$.
Recall that $x\in\cM$ represents the initial condition in the intrinsic (phase space) coordinates of~$\bB$. Let $(\bar q,\bar v)$ be its projection to the plane, expressed relative to the frame attached to~$\bB'$. The map $\bar\pi:(x,\bar\bc,\bar\bu)\mapsto (\bar q,\bar v)$ is clearly continuous.
Given any plane vector $(\bar q,\bar v)$ expressed relative to the frame attached to $\bB'$, pointing towards~$\bB'$, let $x'=F'(\bar q,\bar v)\in\cM$ denote the post-collision vector as expressed in the intrinsic (phase space) coordinates of $\bB'$. Then $F'$ is continuous (except at tangential collisions, which we have ruled out).
We have $F_{\cK',\cK}(x)=x'=F'\circ\bar\pi\circ (\mathrm{id}_\cM\times G) (x,\cK,\cK')$, where the composition comprises continuous functions. The \emph{uniform} continuity statement follows from a standard compactness argument.
\end{proof}


\begin{proof}[Proof of Lemma~\ref{lem:regular_extension}]
Because $W_\star$ is closed in $W$, the set $W\setminus W_\star$ is a countable union of disjoint, open (i.e., endpoints not included), connected, curves $V\subset W$, which we call gaps. Consider a gap $V$. Notice that its endpoints $x$ and $y$ belong to $W_\star$ whence it follows that $\rho$ satisfies $|\log\rho(x)-\log\rho(y)| \leq C\theta^{s(x,y)}$ for the fixed pair~$(x,y)$. Let $r>0$ be \emph{the first} time such that $\cF_rV$ intersects the set $\partial\cM\cup\cup_{|k|\geq k_0}\partial\IH_k$, in other words $r=s(x,y)$, and pick an arbitrary point $z\in V$ whose image $\cF_r(z)$ is in the intersection. On the curve $V$, placing a discontinuity at $z$ as needed, assign~$\rho$ the constant value~$\rho(x)$ between the points $x$ and $z$ and similarly the value $\rho(y)$ between~$z$ and~$y$. With the exception of the above bound being satisfied on all of $W$, the claims of the lemma are clearly true.

To check the bound, let $(x',y')$ be an arbitrary pair of points in $W$. If $x'\in W_\star$, set $x=x'$. Otherwise $x'$ belongs to a gap $V$ with an endpoint $x\in W_\star$ satisfying $\rho(x)=\rho(x')$. Similarly we define a point $y$ in terms of $y'$. For $x=y$ we simply have $\rho(x')-\rho(y')=0$, so let us assume from now on that $x\neq y$. Since $|\log\rho(x')-\log\rho(y')|=|\log\rho(x)-\log\rho(y)| \leq C\theta^{s(x,y)}$, it remains to check that $s(x,y)\geq s(x',y')$ in order to prove $|\log\rho(x')-\log\rho(y')| \leq C\theta^{s(x',y')}$. 
Indeed, given two points $a,b$ on $W$, let $W(a,b)$ denote the open subcurve of $W$ between the two points. If $W(x,y)\subset W(x',y')$, then the bound $s(x,y)\geq s(x',y')$ is obvious. On the other hand, if $W(x,y)\supset W(x',y')$, then there exist \emph{gaps} $W(x,\bar x)$ and $W(\bar y,y)$ on $W$ such that $x'\in \widetilde W^u(x,\bar x)$ and $y'
\in\widetilde W^u(\bar y,y)$, where $\bar x$ and $y'$ are on the same side of $x$ on $W$, and $\bar y$ and $x'$ are on the same side of $y$. By construction, $s(x,x') \geq s(x',\bar x)$ and $s(y',y)\geq s(\bar y,y')$. These inequalities imply immediately $s(x,y)= s(x',y')$. Showing that $s(x,y)\geq s(x',y')$ also when neither $W(x,y)$ nor $W(x',y')$ is completely contained in the other set can be done by combining ideas from the previous cases and is left to the reader.  
\end{proof}

\begin{proof}[Proof of Lemma~\ref{1'implies1}]

Assume that~$\mu$ has a strictly positive, \mbox{$\tfrac16$-H\"older} continuous density~$\chi$ with respect to the measure $d\mu^0 =  \cN^{-1}\rho^0\,\rd r \,\rd\varphi$, where $\rho^0=\cos\varphi$ and~$\cN$ is the normalizing factor; we then have $d\mu = \cN^{-1}\rho\,\rd r\, \rd\varphi$ with $\rho = \chi\rho^0$. Such a measure can be represented as a measured unstable family in a canonical way: Let $S_j$, $1\leq j\leq \infty$, be an enumeration of all the sets $\IH_k\cap\cM_i$. For each $j$, partition $S_j$ into straight lines $W_\alpha$, $\alpha\in E^{(j)}$, of slope $\kappa^{\min}$ and of maximal length so that $\cup_{\alpha\in E^{(j)}} W_\alpha$ is a regular unstable stack. We assume here that the sets $E^{(j)}$ are disjoint subsets of $\IR$ in order to avoid having to introduce additional superscripts $(j)$ for the line segments. Disintegrating~$\mu^0$ and~$\mu$ using these stacks, we denote the conditional densities on $W_\alpha$ by~$\rho^0_\alpha$ and~$\rho_\alpha$ respectively. Because of the simple geometry of the partition,~$\rho^0_\alpha$ and $\rho_\alpha$, are obtained as the normalized restrictions of~$\rho^0$ and~$\rho$ on $W_\alpha$. In particular, we have the identity
\beq\label{eq:cond_identity}
\rho_\alpha=\chi\rho^0_\alpha.
\eeq 
The conditional densities $\rho^0_\alpha$ have uniformly $\tfrac13$-H\"older continuous logarithms. In other words, there exists a constant $C^0>0$, independent of~$\alpha$, such that $|\log\rho^0_\alpha(x)-\log\rho^0_\alpha(y)|\leq C^0d_\cM(x,y)^{1/3}$ for all $x,y\in W_\alpha$, for all $\alpha\in\cA$. Indeed, denoting by $W_\alpha(x,y)\subset\IH_k$ the segment of $W_\alpha$ connecting the points $x,y\in W_\alpha$, we have $|W_\alpha(x,y)|\leq C_\IH k^{-3}$ for a constant $C_\IH>0$ which is uniform for all $k\geq k_0$. Writing $x=(r_x,\varphi_x)$ and $y=(r_y,\varphi_y)$, the bound \eqref{eq:cos_bound} then yields
\beqn
\begin{split}
& |\log\rho^0_\alpha(x)-\log\rho^0_\alpha(y)|  = |\log\cos\varphi_y-\log\cos\varphi_x| \leq \frac{1}{\min(\cos\varphi_y,\cos\varphi_x)}|\cos\varphi_y-\cos\varphi_x| 
\\
& \qquad\qquad\qquad\quad
\leq C_\mathrm{cos} k^2|\varphi_y-\varphi_x|
 \leq C_\mathrm{cos} (C_\IH |W_\alpha(x,y)|^{-1})^{2/3}|W_\alpha(x,y)| \leq C_\mathrm{cos} C_\IH^{2/3} d_\cM(x,y)^{1/3}
\end{split}
\eeqn
as claimed. The extension to $k=0$ is immediate, observing that $\cos\varphi\geq \cos(\pi/2-k_0^{-2})$ on~$\IH_0$. The logarithm of $\chi$ is also $\tfrac16$-H\"older continuous on $\cM$; let us denote the constant~$|\log \chi|_{1/6}$. In particular, $|\log\rho_\alpha(x)-\log\rho_\alpha(y)|\leq |\log\chi|_{1/6} d_\cM(x,y)^{1/6} + C^0 d_\cM(x,y)^{1/3}\leq (|\log\chi|_{1/6} + C^0 L_0^{1/6}) d_\cM(x,y)^{1/6}$ for all $x,y\in W_\alpha$, for all $\alpha$. (Here $L_0$ is an upper bound on the length of a homogeneous unstable curve.) Following Remark~\ref{rem:Holder_regular}, 
$$|\log\rho_\alpha(x)-\log\rho_\alpha(y)|\leq (|\log\chi|_{1/6} + C^0 L_0^{1/6}) C_\mathrm{s}^{1/6}\theta^{s(x,y)}$$
for \emph{any} configuration sequence. Furthermore, denoting by~$\cZ$ and~$\cZ^0$ the quantity appearing in~\eqref{eq:Z_cont} for $\mu$ and $\mu^0$, respectively, the identity in \eqref{eq:cond_identity} yields
\beqn
\cZ \leq \sup \chi \cdot \cZ^0.
\eeqn
Here $\cZ^0<\infty$ by direct inspection and $\sup \chi \leq e^{a\cdot |\log\chi|_{1/6}}$ for a uniform constant $a>0$. Thus, the initial measures of Theorem~\ref{thm:weak_conv_compact} also satisfy the assumptions of Theorem~1', and $|\log \chi|_{1/6}$ controls the constant $C_\gamma$ as claimed.
\end{proof}

\begin{proof}[Proof of Lemma~\ref{lem:top_density}]
Observe that
\beqn
\begin{split}
\frac{\psi(x)}{\psi(y)} &= 1 + \left[\left(\frac{\rho(x)}{\rho(y)}-1\right)\frac{\rho(y)}{\check\rho(y)}+\left(1-\frac{\check\rho(x)}{\check\rho(y)}\right)\right]\left({\frac{\rho(y)}{\check\rho(y)}-1}\right)^{-1}
\leq 1+A(\exp(C_\mathrm{r}\theta^{s(x,y)})-1),
\end{split}
\eeqn
where $A=(B+1)(b-1)^{-1}$. Using the estimate $\log(1+t)\leq t$ ($t\geq 0$), we obtain
\beqn
\left|\log \frac{\psi(x)}{\psi(y)}\right| \leq A(\exp(C_\mathrm{r}\theta^{s(x,y)})-1),
\eeqn
the absolute value on the left side being justified because the preceding bound continues to hold for~$x$ and $y$ interchanged and because $s(x,y)=s(y,x)$. Next, fix a constant $S>0$ so large that $\exp(C_\mathrm{r}\theta^S)-1\leq 2C_\mathrm{r}\theta^S$. Then, if $s(x,y)\geq S$,
we see that \eqref{eq:dyn_Holder} holds if $C_\mathrm{top}\geq 2C_\mathrm{r}$. To deal with the case $s(x,y)<S$, we use the crude bound $|\log\psi(x)-\log\psi(y)|\leq (e^{C_\mathrm{r}}-1)$ obtained from above. Putting things together, \eqref{eq:dyn_Holder} is true for \emph{any} $x$ and $y$, at least for $C_\mathrm{top}=\max((e^{C_\mathrm{r}}-1)\theta^{-S},2C_\mathrm{r})$.
\end{proof}


\begin{proof}[Proof of Lemma~\ref{lem:top_aid}]
Write $W_{n,i}$ for the homogeneous components of $\cF_n W$ and $\nu_{n,i}$ for the push-forward of $\nu(W_{n,i}^-\cap\,\cdot\,)$ under $\cF_n$. Here $W_{n,i}^-$ denotes the element of the canonical $n$-step subdivision of $W$ which maps bijectively onto $W_{n,i}$ under $\cF_n$. It is a regular curve. Set $\cZ_n = \sum_i\nu_{n,i}(W_{n,i})/|W_{n,i}|$. Our task is to show that $\cZ_n\leq C_\mathrm{p}\nu(W)$ eventually. The small nuisance is that we cannot apply Lemma~\ref{lem:Z} directly, as $\nu$ is not necessarily regular. Our trick is to compare the evolutions of $\nu$ and the uniform measure $\fm_W$. Since the latter is obviously regular, Lemma~\ref{lem:Z} does apply: Writing $\cZ_n^{\fm_W} = \sum_i(\fm_W)_{n,i}(W_{n,i})/|W_{n,i}|$, we have
\beqn
\frac{\cZ_n^{\fm_W}}{|W|}
\leq \frac{C_\mathrm{p}}2 \! \left(1+\vartheta_\mathrm{p}^n\frac{1}{|W|}\right),
\eeqn
as $\cZ_0^{\fm_W}=1$. Next, fix $n$ and the component index $i$. We write $x_{-n} = (\cF_n|_{W_{n,i}^-})^{-1}(x)\in W_{n,i}^-$ for the preimage of any $x\in W_{n,i}$. Denoting  by $\ell_{n,i}$ the density of the push-forward $(\fm_W)_{n,i}$, we have $\nu_{n,i}(W_{n,i}) = \int_{W_{n,i}} \rho(x_{-n})\ell_{n,i}(x)\,\rd\fm_{W_{n,i}}(x)
\leq 
\sup_{W_{n,i}^-}\rho\cdot (\fm_W)_{n,i}(W_{n,i})$. From here, using the bound on $\rho$,
\beqn
\cZ_n \leq \sup_W\rho\cdot \cZ_n^{\fm_W} \leq e^C\inf_W\rho\cdot \cZ_n^{\fm_W} \leq e^C \frac{\cZ_n^{\fm_W}}{|W|}\nu(W).
\eeqn
If $n\geq n'=\max\bigl(\log(\frac{C_\mathrm{r}}{2}/(C-\frac{C_\mathrm{r}}{2}))/\log\theta,\,\log|W|/\log\vartheta_\mathrm{p}\bigr)$, Lemma~\ref{lem:pair_image} guarantees that the measures $\nu_{n,i}$ are all regular, and the bounds above yield
\beqn
\frac{\cZ_n}{\nu(W)} \leq e^C C_\mathrm{p}.
\eeqn
We can therefore apply Lemma~\ref{lem:Z}, which results in
\beqn
\frac{\cZ_n}{\nu(W)} \leq \frac{C_\mathrm{p}}2 \! \left(1+\vartheta_\mathrm{p}^{n-n'} e^C C_\mathrm{p}\right).
\eeqn
The measure $(\cF_n)_*\nu$ is therefore proper, provided that $n\geq n' +(C+\log C_\mathrm{p})/|\log\vartheta_\mathrm{p}|$. Because we are assuming $C>C_\mathrm{r}>2$, there exist a uniform constant $A_\mathrm{p}>0$ such that the preceding is implied by $n\geq A_\mathrm{p}\!\left(\max\bigl(\log C,\,|\log|W||\bigr)+C\right)$. The condition in the lemma follows.
\end{proof}


\newpage

\bibliography{Moving}{}
\bibliographystyle{plainurl}


\vspace*{\fill}


\end{document}